\documentclass[letterpaper, 11pt,  reqno]{amsart}

\usepackage{amsmath,amssymb,amscd,amsthm,amsxtra, esint, amsfonts, mathrsfs}

\usepackage[margin=1.2in,marginparwidth=1.5cm, marginparsep=0.5cm]{geometry}

\usepackage[implicit=true]{hyperref}

\allowdisplaybreaks[2]

\usepackage{color}

\definecolor{gr}{rgb}   {0.,   0.69,   0.23 }
\definecolor{bl}{rgb}   {0.,   0.5,   1. }
\definecolor{mg}{rgb}   {0.85,  0.,    0.85}
\definecolor{yl}{rgb}   {0.8,  0.7,   0.}
\definecolor{or}{rgb}  {0.7,0.2,0.2}

\newtheorem{theorem}{Theorem}[section]

\newtheorem{lemma}[theorem]{Lemma}
\newtheorem{proposition}[theorem]{Proposition}
\newtheorem{remark}[theorem]{Remark}

\newtheorem{definition}[theorem]{Definition}

\newtheorem*{ackno}{Acknowledgements}

\newcommand{\cA}{\mathcal{A}}
\newcommand{\cB}{\mathcal{B}}

\newcommand{\cE}{\mathcal{E}}
\newcommand{\cF}{\mathcal{F}}

\newcommand{\fH}{\mathfrak{H}}
\newcommand{\fm}{\mathfrak{m}}
\newcommand{\ff}{\mathfrak{f}}
\newcommand{\fg}{\mathfrak{g}}

\newcommand{\cZ}{\mathcal{Z}}
\newcommand{\cG}{\mathcal{G}}
\newcommand{\NN}{\mathcal{N}}

\newcommand{\bE}{\mathbb{E}}
\newcommand{\N}{\mathbb{N}}
\newcommand{\bfN}{\mathbf{N}}
\newcommand{\scrN}{\mathscr{N}}
\newcommand{\bC}{\mathbb{C}}
\newcommand{\PP}{\mathbb{P}}
\newcommand{\R}{\mathbb{R}}

\newcommand{\bR}{\mathbb{R}}

\newcommand{\noi}{\noindent}
\newcommand{\ind}{\mathbf 1}

\newcommand{\Var}{\textup{Var}}
\newcommand{\Lip}{\textup{Lip}}
\newcommand{\w}{\textup{w}}

\newcommand{\1}{\hspace{0.5mm}\text{I}\hspace{0.5mm}}
\newcommand{\II}{\text{I \hspace{-2.8mm} I} }
\newcommand{\III}{\text{I \hspace{-2.9mm} I \hspace{-2.9mm} I}}

\newcommand{\bul}{\bullet}

\newcommand{\al}{\alpha}
\newcommand{\be}{\beta}
\newcommand{\dl}{\delta}
\newcommand{\ze}{\zeta}
\newcommand{\ga}{\gamma}
\def\th{\theta}
\newcommand{\Om}{\Omega}
 \newcommand{\s}{\sigma}

\newcommand{\nb}{\nabla}

\newcommand{\les}{\lesssim}

\newcommand{\dx}{\partial_x}
\newcommand{\dt}{\partial_t}

\newcommand{\wh}{\widehat}
\newcommand{\wt}{\widetilde}
\newcommand{\dom}{\textup{dom}}
\newcommand{\Leb}{\textup{Leb}}
\newcommand{\eqd}{\stackrel{\rm(law)}{=}}

\numberwithin{equation}{section}
\numberwithin{theorem}{section}

\makeatletter
\@namedef{subjclassname@2020}{%
  \textup{2020} Mathematics Subject Classification}
\makeatother

\begin{document}
\baselineskip = 14pt

\title[Hyperbolic Anderson model with L\'evy  noise]
{Hyperbolic Anderson model with L\'evy  white noise:\\
spatial ergodicity and  fluctuation}

\author[R.M.~Balan and G.~Zheng]
{Raluca M. Balan  and Guangqu Zheng}

\address{Raluca M. Balan (corresponding author),
Department of Mathematics and Statistics\\
University of Ottawa\\
150 Louis Pasteur Private\\
Ottawa, ON, K1N 6N5\\
Canada
}

\email{Raluca.Balan@uottawa.ca}

\address{Guangqu Zheng,
Department of Mathematical Sciences\\
University of Liverpool\\
Mathematical Sciences Building\\
Liverpool, L69 7ZL\\
United Kingdom
 }

\email{Guangqu.Zheng@liverpool.ac.uk}

\dedicatory{Dedicated to
Professor David Nualart  on the
occasion of his retirement}



\begin{abstract}
In this paper, we study  one-dimensional
hyperbolic Anderson models  (HAM) driven by
space-time {\it pure-jump L\'evy white  noise} in a finite-variance setting.
Motivated by recent active research on limit theorems for stochastic
partial differential equations driven by Gaussian noises,
we present  the first  study in this L\'evy setting.
In particular, we  first establish the spatial ergodicity of the
solution  and then a quantitative central limit theorem (CLT) for the spatial
averages of the solution to HAM in both
Wasserstein distance and Kolmogorov distance, with the same rate
of convergence. To achieve the first goal (i.e. spatial ergodicity), we
exploit some basic properties of the  solution and apply
a Poincar\'e inequality in the Poisson setting, which requires
 delicate moment estimates on the Malliavin derivatives of the solution.
Such moment estimates are obtained in a soft manner
by observing a natural connection between the Malliavin derivatives of HAM
and a HAM with Dirac delta velocity.
To achieve the second goal (i.e. CLT), we need two key ingredients: (i)
 a univariate   second-order Poincar\'e inequality
in the Poisson setting  that goes back to Last, Peccati, and Schulte
(Probab. Theory Related Fields, 2016)
and has been recently improved by Trauthwein (arXiv:2212.03782);
(ii) aforementioned  moment estimates of Malliavin derivatives
up to second order.
We also establish a corresponding
functional central limit theorem by (a) showing the convergence in
finite-dimensional distributions  and (b) verifying Kolmogorov's
tightness criterion.   Part (a) is made possible by
a linearization trick
and the univariate  second-order Poincar\'e inequality,
while part (b) follows from a standard moment estimate with
an application of Rosenthal's inequality.

\end{abstract}

\date{\today}

\maketitle

\vspace*{-8mm}

\tableofcontents

\section{Introduction}\label{SEC1}


\subsection{Stochastic linear wave equation with pure-jump  L\'evy white  noise}

Stochastic partial differential equations (SPDEs)
have been studied intensively in the last 30 years, using different approaches.
In the semigroup approach (developed in \cite{DZ92}) or
the variational approach (pioneered in \cite{Par75} and
developed further in \cite{KR79}),
the solution and the noise are processes,
which evolve in time and take values in a Hilbert space.
The random field approach
(initiated by Walsh \cite{Walsh86} and developed further by Dalang \cite{Dalang99})
deviates significantly from these approaches by
proposing a different framework  for
viewing  the noise and the solution.
In Walsh-Dalang's approach, the solution is a space-time indexed  process
(i.e. a random field) and
the noise is a process
indexed by subsets  of the space-time domain (or functions on this domain).
We refer the readers to \cite{PR07, DKMNX09, RL18}
for an overview of the study of SPDEs using these approaches; see
also the paper \cite{DQS11}
for their close connections.
Regardless of the approach,
one can think of the noise (and the initial condition) as the input,
and the solution as the output.
One of the fundamental problems for SPDEs
is the well-posedness problem (i.e. existence, uniqueness, and stability
under perturbation
of the initial data and/or the noise).
And probabilists have been driven to study/discover
new properties of the SPDE solutions, for example,   {\it stationarity, ergodicity}, and
{\it  intermittency property} (i.e. exponential growth of the $p$-th moment
for large time), to name a few.

Various classes of processes have been proposed as models for the noise
perturbing a partial differential equation,
often derived by an analogy with the noises appearing in the classical SDEs:
Brownian motion, L\'evy processes, and fractional Brownian motions.
But the introduction of the infinite dimensional (and spatial) component
changes drastically the problem and leads to  new challenges.
The class of SPDEs perturbed by L\'evy noise
have been studied extensively in the monograph \cite{PZ07}
using the semigroup approach,
where they are naturally interpreted as  extensions of SDEs
driven by L\'evy processes.
One way of which L\'evy noise occurs is in the so-called Schr\"odinger
problem of probabilistic evolution, and several relativistic Hamiltonians
are known to generate L\'evy noises; see, for example, \cite{GKO95}.

In the present article, we will take Walsh-Dalang's random field perspective
and study the following stochastic linear wave equation with a {\it  multiplicative
L\'evy noise} on $\R_+\times\R$:

\noi
\begin{align}
\begin{cases}
&\dt^2 u(t,x) = \dx^2 u (t,x)+u(t,x)\dot{L}(t,x),
 \,\,\, (t, x)\in (0,\infty) \times \R  \\[1em]
&\text{$u(0,x)=1$ \,\,\, and  \,\,\, $\dt u(0,x)=0$, $x\in\R$,}
\end{cases}
\label{wave}
\end{align}

\noi
where $\dot{L}$ denotes a space-time pure-jump L\'evy white noise and the product
 $u \dot{L}$ is interpreted in  It\^o  sense.
 The equation \eqref{wave} is also known as the {\it hyperbolic Anderson model},
 by an analogy of the parabolic Anderson model with the wave operator
 $\dt^2- \dx^2$ replaced by the heat operator $\dt - \dx^2$.

Let us briefly set up the framework.
Let $\cB_0(\R_+\times\R)$ denote the collection of
Borel subsets $A$ of  $\R_+ \times \R$ with $\Leb(A)<\infty$,
where $\Leb$ denotes the Lebesgue measure on $\R_+\times\R$.
Let

\noi
\begin{align}
Z=\R_{+} \times \R \times \R_0, \quad \cZ= \textup{Borel $\sigma$-algebra on $Z$},
 \,\,\, {\rm and} \,\,\,
 \fm={\rm Leb}   \times \nu,
\label{ZZm}
\end{align}

\noi
where
the space $\R_0:=\R\verb2\2\{0\}$ is equipped with the distance $d(x,y)=|x^{-1}-y^{-1}|$, and $\nu$
 is a $\sigma$-finite   measure on $\bR_0$ subject to

\noi
\begin{align}
\label{LevyM}
\int_{\bR_0}\min (1,|z|^2)\nu(dz)<\infty.
\end{align}

\noi
Let $N$ be a Poisson random measure on the space $(Z,\cZ)$ with intensity $\fm$,
and let $\widehat{N}=N- \fm$ be the compensated version of $N$;
see  Definition \ref{def:PRM} for more details.
Fix   $b\in\R$.
 For $A\in\cB_0(\R_+\times\R)$,
we define

\noi
\begin{align}\label{LbA}
\begin{aligned}
L_b(A)
&\equiv \int_{\R_+\times\R} \ind_A(t,x) \dot{L}_b(t, x)dtdx
\equiv \int_{\R_+\times\R} \ind_A(t,x) L_b(dt, dx) \\
&= b  \cdot \Leb(A)+\int_{A \times \{|z|\leq 1\}}z \widehat{N}(dt,dx,dz)+\int_{A \times \{|z|> 1\}}z N(dt,dx,dz),
\end{aligned}
\end{align}

\noi
which is an {\it infinitely divisible} random variable with
\begin{align}
\begin{aligned}
\bE\big[ e^{i \lambda L_b(A)} \big]
&  = \exp\bigg(
i \lambda b \Leb(A) + \Leb(A) \int_{|z|\leq 1} (e^{i \lambda z} - 1 - i\lambda z) \nu(dz) \\
&\qquad\qquad\qquad +\Leb(A) \int_{|z| > 1} (e^{i \lambda z} - 1 ) \nu(dz)
\bigg)
\end{aligned}
\label{IDlaw}
\end{align}
for any $\lambda\in\R$.\footnote{In \eqref{LbA}, 
the stochastic integral over $A\times \{ |z| \leq 1 \}$ lives in 
the first Poisson Wiener chaos $\mathbb{C}_1$ and 
coincides with $I_1(\phi)$,
where  $\phi(t, x, z) = \ind_A(t,x) z \ind_{\{|z| \leq 1 \}}$
belongs to $L^2(Z, \mathcal{Z}, m)$ in view of the condition
\eqref{LevyM}; see Subsection \ref{SUB21} for more details. 
The other stochastic integral over $A\times\{|z| >1\}$  defines 
a finite compound Poisson  random variable with
characteristic function given by 
$\lambda\in\mathbb{R}\mapsto \exp ( \Leb(A) \int_{|z| > 1} (e^{i \lambda z} - 1 ) \nu(dz)
 )$, since $N$, restricted to $A\times\{ |z| >1 \}$, is a Poisson
 random measure with finite intensity measure;
see, for example, \cite[Proposition 19.5]{Sato99}.\label{ft1}}
Besides, one can easily verify that for any $p>0$,

\noi
\begin{align}
 \bE\big[ | L_b(A) |^p \big]<\infty
 \,\,\,
   \Longleftrightarrow
\,\,\,  M_p:=\int_{|z|>1}|z|^p\nu(dz)<\infty.
\label{def:MP}
\end{align}
See Appendix  \ref{APP} for    a proof of \eqref{def:MP}.
  In particular,
 $L_b(A)$ has finite variance if and only if $M_2<\infty$.
   In fact, throughout this paper,
 \begin{center}
we always assume that $M_2 <\infty$.
\end{center}

\noi
By choosing $b = - \int_{|z|>1} z\nu(dz)$,\footnote{This
integral is finite due to the condition \eqref{LevyM} and $M_2<\infty$.}
we put
\begin{align}
L(A) = \int_{A\times \R_0} z \wh{N}(dt, dx, dz),
\label{LA1}
\end{align}
which has mean zero and differs from \eqref{LbA} by a constant.
We say that
\begin{center}
$\{L(A): A \in \cB_0(\R_+ \times \R)\}$ is
 a  {\bf pure-jump  space-time L\'evy   noise}.
 \end{center}
Note that \eqref{LbA}
is the analogue of the {\em L\'evy-It\^o   decomposition}
(\cite[Theorem 19.2]{Sato99})
 of a classical L\'evy process
$X=\{X(t)\}_{t\geq 0}$ without   a Gaussian component, whereas \eqref{IDlaw} is the analogue of the
{\em L\'evy-Khintchine formula} (\cite[Theorem 8.1]{Sato99}).
In    the classical L\'evy process  setting,  there is no space component $x\in\R$, and
the corresponding Poisson random measure on $\R_{+} \times \R_0$
with intensity $\Leb\times\nu$ contains information about the location
and the size of the jumps of $X$.
That being said,  we also call $\nu$ the  jump intensity measure
for the space-time L\'evy noise $L$.

In  \cite{BN16},   the first author and Ndongo proved
the existence, uniqueness, and intermittency property
for the stochastic  nonlinear wave equation
 in dimension $d=1$,
 i.e. with $u\dot{L}$ replaced by $\s(u)\dot{L}$,
 where $\s:\R\to\R$  is Lipschitz.
For a general L\'evy noise, the existence of the solution of
the wave equation in dimension $d\leq 2$ was established in \cite{B21},
together with some path properties.

In this article, we consider the hyperbolic Anderson model \eqref{wave}
and establish the first ergodicity and central limit theorem in
a finite-variance setting, namely,  when $M_2<\infty$.
In view of the condition \eqref{LevyM}, we assume the following
equivalent condition throughout this paper:

\noi
\begin{align}
m_2:=\int_{\R_0} |z|^2 \nu(dz)   \in (0,\infty).
\label{m2}
\end{align}

\noi
$\bul$ {\bf Mild solution.}
We say  that $u$ is a (mild) {\em solution}
to hyperbolic Anderson model \eqref{wave} if
$u=\{ u(t,x): (t,x)\in\R_+\times\R\}$
is a predictable\footnote{Predictability is defined with respect
to the filtration generated by the noise $L$; see \eqref{filtraF}. } process
with $u(0,x)=1$
for any $x\in \bR$ such that
for any $t>0$ and $x \in \bR$, we have

\noi
\begin{align*}
u(t,x)=1+\int_0^t \int_{\R}G_{t-s}(x-y)u(s,y)L(ds,dy),
\end{align*}

\noi
almost surely,
where
\begin{align}
G_t(x)=\frac{1}{2} \ind_{\{|x|<t\}}
\label{FSol}
\end{align} is the fundamental solution
to the deterministic wave equation on $\R_{+} \times \R$,
and the stochastic integral is interpreted in the  It\^o sense,
which is a particular case of the Kabanov-Skorohod integral;
see   Lemma \ref{lem:CE2} (iv).
This mild formulation
 was introduced in \cite{Walsh86},
being motivated by the   Duhamel's principle in   PDE theory.
Since the stochastic integral has zero-mean,
\[
\bE[u(t,x) ] =1 \,\,\, \mbox{for any $(t,x)\in\R_+\times\R$}.
\]

Throughout this paper,  we make
the following convention:

\noi
\begin{align}
\text{$G_t(x)=0$ for all $t\leq 0$ and $x \in \R$}.
\label{convention}
\end{align}

By Theorem 1.1 of \cite{BN16}, the equation \eqref{wave}
has a unique solution satisfying
\[
\sup_{(t,x) \in [0,T] \times \R}\bE[|u(t,x)|^2] <\infty \quad \mbox{for any} \ T>0.
\]
Put
\begin{equation}
\label{mp}
m_p:=\int_{\R_0}|z|^p\nu(dz) \quad\text{for $p\in[1,\infty)$.}
\end{equation}

\noi
The same theorem shows that if $m_p<\infty$ for some finite $p\geq 2$, then

\noi
\begin{align}
\label{KPT}
K_p(T):=\sup_{(t,x)\in [0,T] \times \R} \big( \bE[|u(t,x)|^p] \big)^{\frac1p}<\infty \quad \mbox{for any} \ T>0.
\end{align}
See \cite{BN16,BN17} for more details. See also Remark \ref{rem:main} ({\bf a}) for a discussion
on the finiteness of $m_p$ for $p\in [1, \infty).$
It is   known that due to the linearity of the noise in $u$,
 the solution $u(t,x)$ to \eqref{wave}   admits the following Wiener chaos expansion:

\noi
\begin{align}
u(t,x) = \sum_{n\geq 0} I_n(F_{t,x,n}) ,
\label{u1}
\end{align}

\noi
where $F_{t,x,0}=1$ and for $n\in\N_{\geq 1}$,
the (non-symmetric) kernel $F_{t,x,n}(\pmb{t_n},\pmb{x_n}, \pmb{z_n})$ is given by

\noi
\begin{align}
F_{t,x,n}(\pmb{t_n},\pmb{x_n}, \pmb{z_n})
&=G_{t-t_n}(x-x_n)z_n \ldots G_{t_2-t_1}(x_2-x_1)z_1
\ind_{\{ t> t_n > ...> t_1>0\}};
\label{KER:F}
\end{align}

\noi
see \cite{BN17} and see also Subsection \ref{SUB22}.
 From the orthogonality relation (see \eqref{int3d}) with $\wt{F}_{t,x,n}$ denoting
 the symmetrization of $F_{t,x,n}$ (see \eqref{defh2}), we see that

\noi
\begin{align}
\begin{aligned}
{\rm Cov}(u(t,x),u(s, y))
=\sum_{n\geq 1}n!  \langle \widetilde{F}_{t,x,n},\widetilde{F}_{s,y,n} \rangle_{L^2(Z^n)}.
\end{aligned}
\label{rho_t}
\end{align}
 Note that $\textup{Cov}(u(t, x), u(t, 0)) =0$  when $|x| > 2t$, 
which can be seen from the definition of $F_{t,x,n}$ in {\rm (1.14)}, convention {\rm(1.10)}, 
definition \eqref{FSol} of $G$,
and an application of triangle inequality.
Moreover,
it is not difficult to see from \eqref{KER:F} that the covariance \eqref{rho_t}
depends on $(x,y)$ only via the difference $x-y$.
This hints that
for any fixed $t\in\R_+$, the process $\{u(t,x)\}_{x \in\R}$ is  stationary.
In fact, as we will see in Lemma \ref{lem:stat},
the process $\{u(t,x)\}_{x \in\R}$ is  {\it strictly stationary}
 in the sense that
for any  $x_1, ... , x_m, y \in\R$ with any $m\in\N_{\geq 2}$,

\noi
\begin{center}
$( u(t, x_1 +y), ... , u(t, x_m + y)  ) = ( u(t, x_1), ... , u(t, x_m)  ) $ in law.
\end{center}
Then, it is natural to define an associated family of shifts $\{\th_y\}_{y\in\R}$
by setting

\noi
\begin{align}
\th_y\big(  \{ u(t, x)\}_{x\in\R} \big) :=   \{ u(t, x+y)\}_{x\in\R}, \label{shifty}
\end{align}

\noi
which preserve the law of the (spatial) process.  Then,
the following question arises:
\begin{align}
\textit{Are the invariant sets for $\{\th_y\}_{y\in\R}$  trivial?
\textup{(i.e.} is $u(t,\bul)$ spatially ergodic?\textup{)}}
\label{quest}
\end{align}

\noi
One can refer to,  for example,  the book \cite{Peter89} for
an introduction to the ergodic theory.

\medskip

To the best of authors' knowledge, the question \eqref{quest}  of spatial ergodicity
has not been investigated for the hyperbolic Anderson model
\eqref{wave} driven by L\'evy noise. See the work \cite{NZ20} by Nualart
and the second author for the study of stochastic nonlinear wave
equation driven by Gaussian noises and see also \cite{CKNP21} for
similar study for parabolic SPDEs.
In this paper, we present the first ergodicity result
for the equation \eqref{wave}, and thus answer the question affirmatively;
see Theorem \ref{thm:main} (i).
Consequently, the spatial ergodicity implies the following first-order
fluctuation (`law of large number type'):  letting
\noi
\begin{align}
F_R(t) := \int_{-R}^R \big( u(t,x) - 1   \big)dx,
\label{FRT}
\end{align}
we have

\noi
 \begin{align}
\frac{F_R(t)}{R} 
\to 0 \,\, \text{in $L^2(\Om)$ {and almost surely} as $R\to\infty$.}
\label{LLN}
\end{align}
See also Remark \ref{rem:erg}.  After establishing the first-order fluctuation,
it is natural to investigate the second-order fluctuation:
we will show that $F_R(t) $ (with $t>0$) admits Gaussian fluctuation as $R\to\infty$;
see Theorem \ref{thm:main} (iii). The central limit theorems (CLT)  therein are of
quantitative nature, described by Wasserstein distance and Kolmogorov distance.
We are also able to obtain a functional CLT (see part (iv) in Theorem \ref{thm:main}).

 \subsection{Main results}
 Now we are ready to state the main theorem in this paper.

\begin{theorem} \label{thm:main}
Recall the definition of  $m_p$ in \eqref{mp} and assume $0< m_2 < \infty$  as in \eqref{m2}.
Let $u$ solve the hyperbolic Anderson model \eqref{wave}. Then,
the following statements hold.

\smallskip
\noi
{\rm (i)} Fix $t\in\R_+$. Then,
 $\{ u(t,x): x\in\R\}$ is
 strictly stationary and ergodic.

\smallskip
\noi
{\rm (ii)} The spatial integral $F_R(t)$,
defined in \eqref{FRT} has the following limiting covariance: 
\[
\lim_{R \to \infty}\frac{1}{R}\bE[F_R(t)F_R(s)] =\Sigma_{t,s} \quad \mbox{for any $t,s\geq 0$},
\]
 where $\Sigma_{t,s}$ is given by  

\noi
\begin{align}\label{COV_S}
\Sigma_{t,s} := 2 m_2	\int_0^{t\wedge s} (t-r) (s-r) \cosh\bigg(r \sqrt{ \frac{m_2}{2} } \, \bigg) dr.
\end{align}

\noi
In particular, 
$\s_R^2(t) := \Var (F_R(t)) \sim  \Sigma_{t,t} R$ as $R\to\infty$. 

\smallskip
\noi
{\rm (iii)} Assume additionally  that
\begin{align}\label{cond:al}
 \text{$m_{2+2\al}$ and $m_{1+\al}$ are finite  for some $\al\in(0,1]$.}
 \end{align}

 \noi
 Fix $t\in(0,\infty)$.
Then, the spatial integral $F_R(t)$
admits Gaussian fluctuation as $R\to\infty$. More precisely, $F_R(t)/\s_R(t)$
converges in law to the standard normal distribution $\NN(0,1)$.
Moreover, the following rates of convergences hold:

\noi
\begin{align}
\textup{dist}\Big( \frac{F_R(t)}{ \s_R(t) } , \NN(0,1)   \Big)
\les  {R^{-\frac\al{1+\al}}} ,
\label{CLT}
\end{align}

\noi
where the implicit constant in \eqref{CLT} does not depend on $R$
and  one can choose the distributional metric $\textup{dist}$ to be  one of the
following:
Fortet-Mourier distance, $1$-Wasserstein distance, and Kolmogorov distance;
see Subsection \ref{SUB23} for the definitions of these distances.

\smallskip
\noi
{\rm (iv)} For any fixed $R\geq 1$,
the process  $\{ F_R(t) \}_{t\in\R_+}$ admits a
locally $\be$-H\"older continuous modification for any $\be\in(0, \frac12)$.
Let $\cG := \{\cG_t\}_{t\in\R_+}$ denote a real
centered continuous Gaussian process
with covariance $\bE[\cG_t \cG_s]=\Sigma_{t,s}$.
Moreover, under the assumption \eqref{cond:al},
the process $\{ \frac{1}{\sqrt{R}}F_R(t) \}_{t\in\R_+}$
converges in law to $\cG$ in the space $C(\R_+; \R)$ as $R\to\infty$.\footnote{The
space $C(\R_+; \R)$ consists  of continuous functions from $\R_+$ to $\R$.
Equipped with the compact-open topology (the topology of uniform
convergence on compact sets), the space $C(\R_+; \R)$ is Polish
(i.e. a complete separable metrizable topological space).
}

\end{theorem}

Theorem \ref{thm:main}  presents the first result of spatial ergodicity
and the (quantitative) central limit theorem for SPDEs driven by
space-time L\'evy noise.  Our work is motivated by
a recent line  of investigations for SPDEs with Gaussian noise.
In \cite{HNV20}, Huang, Nualart, and Viitasaari
initiated the study of central limit theorems for SPDEs in Dalang-Walsh's
random field framework. More precisely, they established
the first Gaussian fluctuation result
 for the spatial integral of the  solution
to a   stochastic  nonlinear heat equation driven by space-time Gaussian white noise.
Since then,  we have witnessed a rapidly growing
literature on similar CLT results for heat equations with various
Gaussian homogeneous
noises; see, for example,
\cite{HNVZ20,NZejp,  NSZ, CKNP21, CKNP22, CKNPjfa,
ANTV22, NXZ22, PU22, LP22}.
Meanwhile, such a program was   carried out by Nualart,
the second author, and their collaborators
to investigate the stochastic  (nonlinear)  wave equation driven by Gaussian noises;
see \cite{DNZ20,BNZ21, NZ22, NZ20, BNQSZ}.
All these references address SPDEs (heat or wave) with Gaussian   noises, and 
currently we have already seen a well-developed strategy based on Gaussian analysis,
Malliavin calculus, and Stein's method.

In the present article, we  carry out a similar program for the  SPDE
with L\'evy noises,
by first investigating the hyperbolic Anderson model \eqref{wave}
with multiplicative
space-time L\'evy   noise of pure-jump type. This setting is much more 
complicated than the Gaussian setting, since various tools from Gaussian analysis 
(such as the neat chain rule of Malliavin derivative operator, hypercontractivity property 
of the Ornstein-Uhlenbeck semigroup) are no more available. 
Another technical difficulty is that unlike the Gaussian setting in aforementioned references,
the random field solution to \eqref{wave} does not have finite moment of any order, unless
we impose restrictive conditions on the L\'evy measure of the L\'evy noise 
(see, e.g., \eqref{mp}-\eqref{KPT}).
As such, we choose to first consider the finite-variance setting,
in which we develop an $L^2$ theory of Malliavin calculus
associated with the  space-time L\'evy   noise.
Our approach is  then built  on some  recent results  of
Malliavin calculus on the Poisson space
(see \cite{PSTU10, PT13, Last16, LPS16,  DP18, DP18b, DVZ18, LP18, LMS22, Tara}).
 Our main tool is a second-order Poincar\'e inequality first derived
 in \cite{LPS16} by Last, Peccati, and Schulte and recently improved
 by Trauthwein \cite{Tara}. In this paper, we combine these second-order Poincar\'e inequalities
  with some key moment estimates for the Malliavin derivatives of the solution
 (relations \eqref{D1est} and \eqref{D2est} below).
 These new moment estimates are obtained using the explicit chaos expansions of
 these Malliavin derivatives,
 and a connection with the solution to  the stochastic wave equation
 with delta initial velocity
 (which is studied in Section \ref{SUB31} and may be of independent interest).
 This line of arguments in establishing the moment estimates of Malliavin derivatives 
 of SPDE solutions deviates greatly from those in \cite{BNQSZ},
 which rely heavily on the hypercontractivity property (Wiener chaos estimates)
 applied to the explicit form of the Malliavin derivatives. 
 Note that in general the Ornstein-Uhlenbeck semigroup does not satisfy
 the  hypercontractivity property in the Poisson setting except in some restrictive 
 framework; see, for example, \cite{NPY20}.

 In the case of the  stochastic nonlinear wave equation
(with $u\dot{L}$ replaced by $\s(u)\dot{L}$ in \eqref{wave}),
the solution does not have an explicit chaos expansion,
so that the approach in current paper is not applicable. 
And it is not straightforward at all (even in our linear setting) to adapt
the method in, e.g., \cite{HNV20, DNZ20, BNZ21, CKNP22, CKNPjfa} for establishing
similar CLTs for the  wave equation with L\'evy noises. 
The immediate obstacle arises due to a lack of derivation property
of the Malliavin derivative operator (i.e. no neat chain rule; see Remark \ref{rem:add1} (iii)),
and then in the process of bounding the Malliavin derivative of 
the nonlinear solution, we will encounter
 the term $D_{s,y, z} \s( u(r, \w) )$ that appears in the equation
 for Malliavin derivative
 \[
D_{s, y, z}  u(t,x) = G_{t-s}(x-y)z \s\big(  u(s,y)  \big) 
+ \int_0^t \int_{\R}  G_{t-r}(x-\w) D_{s,y, z} \s( u(r, \w) ) L(dr, d\w).
\]
We plan to investigate this problem in a future project.
Another interesting and more challenging direction is
to investigate  the infinite-variance setting;
for example, one may begin with the  hyperbolic Anderson model
\eqref{wave} with  $L$   replaced by  a $\al$-stable  L\'evy noise (see \cite{B14}).
We expect that some {\it noncentral} limit theorems would arise.
In the recent work \cite{DT22},
Dhoyer and Tudor   considered
a stochastic heat equation with Rosenblatt noise and
established a noncentral limit theorem with the limiting process
being a Rosenblatt process that lives
in the second  Gaussian Wiener chaos and thus has all the moments.  We expect  it to be much more difficult
to obtain the conjectured noncentral limit theorem in the aforementioned
infinite-variance setting.

 At the end of this introduction, let us also mention that  the stochastic
heat equation with multiplicative  L\'evy noises   $\s(u) \dot{L}$, with $\s$ Lipschitz,
has  been  studied in a series of recent papers.
The existence of the solution was proved in  \cite{chong17},
weak intermittency property was established in \cite{chong-kevei19},
some path properties were obtained in \cite{CDH19},
and the exact tail behavior was described in \cite{chong-kevei22} in the case
of additive noise (i.e. when $u\dot{L}$ is replaced by $\dot{L}$).
Uniqueness and strong intermittency of the solution were obtained in \cite{BCL}
in the case of multiplicative noise when $\sigma(u)=u$.
All these results are valid for a general L\'evy noise with
possibly infinite variance (such as the $\alpha$-stable L\'evy noise).
See also earlier investigations \cite{Mueller98, My02, MMS06}
by Mueller, Mytnik, and Stan. 

\medskip

We conclude this introduction with several remarks.

\begin{remark} \label{rem:pam}  
\rm
 In \cite{chong-kevei19}, the authors studied the moment asymptotics for the solution  to
the stochastic heat equation driven by a     space-time   L\'evy white noise
 (with a Gaussian component), 
 whose L\'evy measure $\nu$ satisfies the condition $m_{p}<\infty$ 
 for some $p \in [1,1+\tfrac{2}{d})$. 
 If $d\geq 2$, this value $p$ is strictly smaller than 2, 
 the noise may have infinite variance, 
 and the CLT becomes problematic. 
 (Even in the classical case of i.i.d. random variables, 
 the CLT holds if and only if the variable is the domain of attraction of the normal law, 
 which means that its variance is finite, or its truncated variance is slowly varying.) 
  When $d=1$ and $m_p<\infty$ for some $p \in [2,3)$, 
 the L\'evy noise has finite variance; then 
 it might be possible to prove a CLT 
(similar to the one given by Theorem \ref{thm:main}) 
for the solution to the stochastic heat equation with L\'evy noise.
As far as we know, this problem
has not been treated in the literature, even in case of the parabolic Anderson model (PAM). 
A key step is to obtain the estimates
\eqref{D1est} and \eqref{D2est} for the Malliavin derivatives of the  PAM  solution.
 In the present article, these estimates are derived using 
the connection with the solution $v$ of the  wave equation with Dirac initial velocity, 
and the crucial identity \eqref{wave_dl5}, which  heavily relies on  the fact that the fundamental
wave solution $G$ is an indicator function  (see \eqref{FSol}). 
  Therefore, one needs a  different method   to study  the CLT problem for the heat equation.

\end{remark}

 \begin{remark} \label{rem:main}
 \rm
 ({\bf a}) In view of \eqref{LevyM} and interpolation, one can deduce that
 \[
 m_p < \infty \Longrightarrow m_q <\infty
 \]
 for $2\leq q\leq p < \infty$. In particular,
 the condition  that $m_{2+2\al}<\infty$ for some $\al>0$
 implies the finiteness of   $m_2$.
 However, the finiteness of $m_{1+\al}$  in \eqref{cond:al} 
 with $\al\in(0,1)$ and that of $m_2$
 are independent in general,  illustrated by the following example.
Consider, for example,
\[
\nu_{a,b}(dz) = \big( c_1 |z|^{-a-1} \ind_{\{ 0< |z| \leq 1 \}} + c_2 |z|^{-b-1} \ind_{\{  |z| > 1 \}} \big)dz,
\]
where  $c_1, c_2\in\R_+$. 
It is easy to verify that $\nu_{a,b}$ is a L\'evy measure satisfying \eqref{LevyM}
if and only if $a<2$ and $b>0$. In this case ($a<2$, $b>0$, and $0< \al < 1$), 
we can further    verify that

\noi
\begin{align*}
\int_{\R} |z|^{1+\al} \nu(dz) <\infty &\Longleftrightarrow a < 1+ \al < b,  \\
\int_{\R} |z|^2 \nu(dz) <\infty &\Longleftrightarrow  a< 2 < b.
\end{align*}

\noi
It is also clear that the assumption \eqref{cond:al} holds if and only if $a< 1+\al$ and $b>2+2\al$.

\smallskip
\noi
({\bf b}) Assume $m_4<\infty$, then one can prove (functional) CLTs by using the chaotic 
central limit theorems in the spirit of \cite[Section 8.4]{NN18}.\footnote{The chaotic CLT there addresses 
the weak convergence of Gaussian functionals, 
while we are dealing with the Poisson functionals that 
will lead to more complicated computations of contractions.} 
More concretely, one can prove  ``$F_R(t)/\s_R(t) \to \NN(0,1)$'' 
(a qualitative result compared to \eqref{CLT}) as follows:
\begin{itemize}

\item[(i)] With explicit chaos expansion \eqref{u1} of $u(t,x)$, we can write down the chaos expansion
of  $F_R(t)/\s_R(t)$ in the following form:
\[
F_R(t)/\s_R(t) = \sum_{n=1}^\infty I_n( g_{n, R})
\]
with unique symmetric kernels $\{ g_{n, R}: n\geq 1 \}$.

\item[(ii)] It is not difficult to show that the tail in the above series can be uniformly controlled, meaning 
that 
\begin{align} \label{CCLT1}
\lim_{N\to\infty}\sup_{R\geq 1} \Var \sum_{n\geq N} I_n( g_{n, R}) = 0.
\end{align}

\item[(iii)] For any fixed integer $N\geq 2$, the random vector $( I_n( g_{n, R}): n=1,..., N  )$
has diagonal covariance matrix (due to orthogonality \eqref{int3d}) that tends to an explicit 
covariance matrix $\wt{C}$ as $R\to\infty$. Then, the weak convergence of $( I_n( g_{n, R}): n=1,..., N  )$
to a Gaussian vector $\NN(0, \wt{C})$ can be proved if one can show 
\begin{align}\label{CCLT2}
 \bE\big[ I_n( g_{n, R})^4 \big] \xrightarrow{R\to\infty} 3 \wt{C}_{nn}^2.
 \end{align}
This sufficiency is a consequence of the multivariate fourth moment theorem in the Poisson setting
 first established   by D\"obler, Vidotto, and the second author \cite{DVZ18}. 

\item[(iv)] Finally the verification of \eqref{CCLT2} proceeds with an application of product 
formula (for example, the one in \cite{DP18b}), which inevitably requires 
the finite fourth moment assumption (i.e. $m_4 < \infty$ in current context).

\item[(v)] Combining (ii) and (iii) with a triangle inequality yields easily the announced  CLT 
``$F_R(t)/\s_R(t) \to \NN(0,1)$''. 
\end{itemize}

\noi
The convergence in finite-dimensional distributions via chaotic CLTs
can be established in a similar manner, and 
we leave details to interested readers, who shall expect that the computations are 
more involved than what we are doing in the current paper. 
 In the current paper, we have access to bounds on Malliavin derivatives
of solution so that we can take advantage of the recent work \cite{Tara}
to derive the quantitative CLT in \eqref{CLT}. 
We believe that the above road map (i)-(v) 
would be useful in other Poisson context, when we do not have 
Malliavin differentiability.

\smallskip
\noi
({\bf c}) As already mentioned in (b), one of the key technical ingredients in establishing our
quantitative CLTs is the second-order Gaussian Poincar\'e inequalities
by T. Trauthwein \cite{Tara} that improved previous 
work \cite{LPS16} by Last, Peccati, and Schulte. 
Applying the bounds from \cite{LPS16} will force us to impose the condition ``$m_4 < \infty$'',
while the work \cite{Tara} allows us to work with a more general condition
\eqref{cond:al}. In the following, we present a discussion for
the assumptions \eqref{cond:al} when $\al=1$ and when $\al < 1$:
\begin{itemize}
\item
As we know, the L\'evy measure $\nu$ encodes the information of jump sizes.
The assumption `$m_4<\infty$' (that corresponds to $\al=1$ in  \eqref{cond:al}) is equivalent to
the condition `$M_4 <\infty$' (see \eqref{def:MP}), which imposes the condition on the
`large jumps'. See also Appendix \ref{APP}.

 \item If one weakens the condition on the `large jumps' by allowing only $m_{2+2\al}$ to be finite
 with $0< \al < 1$,
 the assumption \eqref{cond:al} indicates that we need to impose  `$m_{1+\al} < \infty$',
   a further condition
 on `small jumps'. 
 Then it is natural to see that these conditions on `large/small jumps'
 affect the rate of convergence to normality. A similar condition involving interacting 
 conditions for the large/small jumps of the noise
was considered in \cite{chong17} in the case of the stochastic heat equation 
driven by a L\'evy noise, 
with possible infinite variance. In \cite{chong17}, 
by requiring that $m_q<\infty$ and $M_p<\infty$ for some $0<q<p <1+ \tfrac{2}{d}$, 
the author was able to construct the solutions  to the equation with truncated noises, 
up to a stopping time, and then paste these solutions to obtain the solution for the equation 
with the general noise.

 \end{itemize}

 \noi
({\bf d})  By using the explicit covariance structure $\Sigma$ of the limiting 
Gaussian process   $\cG$, it is not difficult  to derive that 
  $\cG$ is almost surely  locally $\be$-H\"older
 continuous   for any $\be\in(0,1)$.

 \end{remark}

\begin{remark} \label{rem:erg} \rm
{The $L^2(\mathbb{P})$ and almost sure convergences in \eqref{LLN}
follow from von Neumann's mean ergodic theorem and Birkhoff's pointwise ergodic theorem;
see \cite[Chapter 2]{Peter89} and see also \cite[Chapter XI, Section 2]{Doob53}.
Alternatively,}   we can easily deduce the $L^2(\mathbb{P})$-convergence in
\eqref{LLN} (as $R\to\infty$) from Theorem  \ref{thm:main} (ii).
Moreover, if we   assume   $m_{2+2\al} < \infty$ for some positive $\al>0$,
then we   also have a simple proof of
 the almost sure convergence as $R\in\N\to\infty$:
we first deduce from \eqref{Rosen2b} with $p=2+2\al$ that
\begin{align}
\sum_{k\in\N} \bE\Big[  \frac{|F_k(t) |^{2+2\al}}{k^{2+2\al}} \Big]
\les \sum_{k\in\N} \frac{1}{k^{1+\al}} <\infty,
\notag
\end{align}

\noi
and thus  from Fubini's theorem, it follows that 
\[
\sum_{k\in\N} \frac{|F_k(t)|^{2+2\al}}{k^{2+2\al}} <  \infty
\]
almost surely, which implies that $F_k(t)/k \to 0$ almost surely as $k\in\N \to \infty$.

\end{remark}

\medskip
\noi
$\bul$ {\bf Organization of this paper.} In Section \ref{SEC2},
we introduce the framework,
and include some basic definitions and results regarding:
stochastic analysis on the Poisson space, Poincar\'e inequalities,
and moment inequalities.
In Section \ref{SEC3}, we present moment estimates for the
Malliavin derivatives of the solution.
Section \ref{SEC4} is devoted to the proof of Theorem \ref{thm:main}.

\section{Preliminaries}\label{SEC2}

\subsection{Notations}  \label{SUB21} By $a\les b$, we mean that
$a\leq C b$ for some positive finite constant $C$ that does not
depend on $(a, b)$. And we write $a\sim b$ if $a\les b$ and $b\les a$.
 For conciseness, we write $a\wedge b  = \min(a, b)$ and
$a\vee b = \max(a,b)$
for any $a,b\in\R$. Throughout this paper, we may fix a rich enough probability
space $(\Om, \cF, \PP)$, on which all the random objects in this
paper are defined.  We denote by $\bE$ the associated expectation operator.
For a real-valued random variable
$X\in L^p(\Om, \cF, \PP)$, we write
$\| X\|_p := \| X \|_{L^p(\Om)} =   ( \bE[ | X|^p ]  )^{\frac1p}$ for finite $p\geq 1$,
while $\| X\|_\infty$ is defined as the essential supremum of $X$.
 To indicate that two random objects
$X, Y$
have the same distribution,
we write  $X\eqd Y$; and we write $Y\sim \NN(0,1)$ to mean that
$Y$ is a standard Gaussian random variable.
We denote by $\s\{X\}$ the $\s$-algebra generated
by the random object $X$. For example, $L^2(\Om, \s\{ N\}, \PP)$
denotes the space of real-valued, square-integrable random variables
that are measurable with respect to $\s\{N\}$.

Let $(Z, \mathcal{Z}, \fm)$ be the $\sigma$-finite measure space
given as in \eqref{ZZm}.
 The Poisson random measure $N$, over which 
our space-time pure-jump L\'evy white noise $L$ is built,
is a set-indexed family $\{ N(A): A\in \cZ\}$ of Poisson random variables. 
Alternatively, one can define $N$ as a random variable with values
in the set of point measures. To be more precise, 
let  $\bfN_\s$ be the set of all  $\s$-finite  measures $\chi$
on $(Z, \mathcal{Z})$ with $\chi(B)\in \N_{\geq 0} \cup \{+\infty\}$
for each $B\in\cZ$.
 Let $\mathscr{N}_\s$ be the smallest $\s$-algebra
that makes the mapping
$\chi\in \bfN_\s \mapsto \chi(B)\in[0,\infty]$
measurable for each $B\in\cZ$. Now we are ready to state the definition of the Poisson
random measure that suits our application; see also Remark \ref{rem:add1} for more discussions.  

\begin{definition}\label{def:PRM}

A Poisson random measure with intensity measure $\fm$ is a 
$(\bfN_\s,\mathscr{N}_\s)$-valued random element $N$ defined on a probability space 
$(\Omega, \mathcal{F}, \mathbb{P})$ such that:
\begin{itemize}

\item for each $A\in\cZ$, the random variable $N(A)$
follows a Poisson distribution  with mean $\fm(A)$;\footnote{If $\fm(A)=\infty$,
we set $N(A)=\infty$ almost surely.}

\item  for any finite sequence $A_1, ... , A_k\in\cZ$ of pairwise
disjoint sets, the random variables $N(A_1), ... , N(A_k)$ are independent.

\end{itemize}
For $A\in\cZ$ with $\fm(A)<\infty$, we define $\wh{N}(A) = N(A) - \fm(A)$ and we
call $\wh{N}$   the compensated Poisson random measure on $(Z, \cZ, \fm)$.




\end{definition}

 Assume that $M_2<\infty$  (see \eqref{def:MP} and \eqref{m2}),
 and
let $L=\{L(A): A \in \cB_0(\bR_+ \times \bR)\}$ be the
finite-variance  space-time  L\'evy   noise given
as in \eqref{LA1}. We set $L(1_{A})=L(A)$,
and we extend this definition by linearity to simple functions.
Then, by approximation, for any function
$\varphi \in L^2(\bR_+ \times \bR)$,
we define the stochastic integral
$L(\varphi)=\int_{\bR_+ \times \bR}\varphi(t,x) L(dt,dx)$.
Note that

\noi
\begin{align}
\label{L1}
L(\varphi)=\int_{\bR_+ \times \bR \times \bR_0} \varphi(t,x) z \widehat{N}(dt,dx,dz).
\end{align}

\noi
Similarly to the Gaussian white noise, this integral satisfies an isometry property:
\[
\bE[L(\varphi)L(\psi)]=m_2 \langle \varphi,\psi \rangle_{L^2(\bR_+ \times \bR)}
\]

\noi
with $m_2$ as in \eqref{m2}.
Moreover, the family $\{L_t(A)=L([0,t]\times A): t\geq 0, A \in \cB_0(\bR)\}$
is a worthy martingale measure, as defined in \cite{Walsh86}.
The It\^o-type stochastic integral $\int_0^t \int_{\bR}X(s,x)L(ds,dx)$
with respect to $L$ is well-defined for any predictable process
$X=\{X(t,x): t\geq 0,x\in \bR\}$ with
\[
\bE \int_0^t \int_{\bR}|X(s,x)|^2 dxds<\infty \quad \mbox{for any $t>0$},
\]
and is related to the It\^o-type stochastic integral with respect to $\widehat{N}$
as follows:
\[
\int_0^t \int_{\bR}X(s,x)L(ds,dx)
=
\int_0^t \int_{\bR}\int_{\bR_0}X(s,x)z \widehat{N}(ds,dx,dz).
\]

\noi
Predictability is defined with respect to the filtration $\mathbb{F}$ induced by $N$, given by \eqref{filtraF} below. More concretely,
a predictable process is a process that is measurable with the
predictable $\sigma$-field on $\bR_+ \times \bR\times\R_0$,
which is the
$\sigma$-field generated by linear combinations of elementary
processes of the form

\noi
\begin{align}
V(t,x, z)=Y \ind_{(a,b]}(t) \ind_{A\times\Gamma}(x, z),
\label{simpleX}
\end{align}

\noi
where $0<a<b$,
$A\times\Gamma \in \cB(\bR) \times\cB(\bR_0)$ satisfies
$\Leb(A)+\nu(\Gamma)<\infty$,
and $Y$ is bounded  $\cF_a$-measurable.\footnote{We can additionally
restrict $Y$ to be Malliavin differentiable, in view of Remark \ref{rem:add1} (iii)
and a limiting argument. 
 This additional restriction
will be used in the proof of Lemma \ref{lem:CE2} (iv). \label{ft77}}
We refer readers to \cite{B15,BN16}, and Section 8.7 of \cite{PZ07}
for more details about integration with respect to $L$ and $\wh{N}$.

Recall that the stochastic integral $L(\varphi)$ given by \eqref{L1}
is  a centered and square-integrable random variable with

\noi
\begin{align*}
\Var\big( L(\varphi) \big)
&= \int_{\R_{+} \times \R \times \R_{0}}  | \varphi(t,x)z  |^2  \, dt   dx \nu(dz) \\
&= m_2 \| \varphi \|_{L^2(\R_+\times \R)}^2
\end{align*}

\noi
with $m_2$ as in \eqref{m2}.
Note that $L(\varphi)$ lives in the first Poisson Wiener chaos associated
to the Poisson random measure $N$
and it coincides with the first-order Wiener-It\^o-Poisson integral
 $I_1( \varphi\otimes z )$.  Let us now construct $I_1(\phi)$
  for a deterministic function $\phi\in L^2(Z, \cZ, \fm)$.
First, there is a sequence of simple functions $\{ \phi_n\}_n$ of the form

\noi
\begin{align}
\phi_n = \sum_{j=1}^{M_n} \al_j \ind_{A_j\times B_j\times C_j }
\label{int0}
\end{align}

\noi
with $\al_j\in\R$, $M_n\in\mathbb{N}$, and
$(A_j, B_j, C_j) \in \cB(\R_+) \times \cB(\R)\times \cB(\R_0) $
with finite measure,
 such that  $\phi_n$ converges to $\phi$ in $L^2(Z, \cZ, \fm)$
as $n\to\infty$. Then,
\begin{align}
I_1(\phi_n)  :=  \sum_{j=1}^{M_n}  \al_j   \wh{N}( A_j\times B_j\times C_j )
\label{int1}
\end{align}

\noi
is well defined with $\| I_1(\phi_n) \|_2 = \| \phi_n\|_{ L^2(Z, \cZ, \fm)}$,
and thus
\begin{align}
I_1(\phi) = \lim_{n\to\infty} I_1(\phi_n) \,\,\, \text{in  $L^2(\PP)$}
\label{int2}
\end{align}
is well defined.\footnote{It is clear that
the definition of $I_1(\phi)$ in \eqref{int2} does not depend on
the choice of approximating sequence $\{\phi_n\}_n$.
The same comment applies to the definition of $I_k(h)$ in \eqref{int4}.}
The set $\mathbb{C}_1= \{ I_1(\phi):  \phi\in L^2(Z, \cZ, \fm) \}$
is called the first Poisson Wiener chaos associated with $N$ (or $\wh{N}$).
See   Subsection \ref{SUB22} for higher-order Poisson Wiener chaoses.

We denote by $\cF^0_t$ the $\s$-algebra
generated by the random variables
$N([0,s]\times A \times B)$ with $s \in [0,t]$ and $\Leb(A)+ \nu(B)<\infty$.
And let $\cF_t = \s\big( \cF_t^0 \cup \NN \big)$ be the
$\s$-algebra generated by $\cF_t^0$ and  the set $\NN $ of $\PP$-null sets.
This gives us a filtration
\begin{align}
\mathbb{F}:= \{ \cF_t: t\in\R_+\}.
\label{filtraF}
\end{align}

It is not difficult to see from \eqref{int0}, \eqref{int1}, and an approximation argument
that for $\phi\in L^2(Z, \cZ, \fm)$,

\noi
\begin{align}
\bE\big[ I_1(\phi) \vert \cF_t \big] = I_1( \phi \ind_{[0,t] \times\R\times\R_0} ).
\label{CE1}
\end{align}

 For conciseness of notations, we denote by $\fH$ the Hilbert space
 $L^2(Z, \cZ, \fm)$ and by $\fH^{\otimes n}$ the $n$-th tensor product
 of $\fH$ for any   integer $n\geq 1$.
We often write $\pmb{x_n} = (x_1,
\dots, x_n)$ for an element in $\R_+^n$, $\R^{n}$, or $\R_0^n$;
  $d\pmb{x_n}$ is an abbreviation for $dx_1 \cdots dx_n$,
  and
$\nu(d\pmb{z_n}) = \nu(dz_1)\cdots \nu(dz_n)$.
 From time to time, we
write $\xi = (r, y, z)$ to denote a point in $Z$
and $\fm(d\xi)  = dr  dy \nu(dz)$.
For a  function $h\in \fH^{\otimes n}$,
we often write
\begin{align}
h(\pmb{\xi_n}) = h(\pmb{t_n}, \pmb{x_n},\pmb{z_n})
= h(t_1, x_1,z_1, \dots, t_n, x_n, z_n),
\label{defh1}
\end{align}

\noi
whenever no confusion appears.

%
%

For $h$ as in \eqref{defh1}, we define its canonical symmetrization $\wt{h}$
by setting

\noi
\begin{align}
\begin{aligned}
\wt{h}(\pmb{\xi_n}) &=
\wt{h}(\pmb{t_n}, \pmb{x_n},\pmb{z_n})  \\
&= \frac{1}{n!} \sum_{\pi\in \mathfrak{S}_n} h( \xi_{\pi(1)}, \ldots, \xi_{\pi(n)}   ) \\
&= \frac{1}{n!} \sum_{\pi\in \mathfrak{S}_n}
h( t_{\pi(1)}, x_{\pi(1)}, z_{\pi(1)}, \ldots,  t_{\pi(n)}, x_{\pi(n)}, z_{\pi(n)} ),
\end{aligned}
\label{defh2}
\end{align}

\noi
where $\mathfrak{S}_n$ denotes the set of permutations over $\{1, ..., n\}$.
Let $\fH^{\odot n}$ denote the symmetric subspace of $\fH^{\otimes n}$.
That is, $\fH^{\odot n}$ consists of all elements $h\in \fH^{\otimes n}$
with $h = \wt{h}$.

To ease the notations, we introduce the cut-off of a function $h\in\fH^{\otimes n}$
in the temporal variable:
\begin{align}
h^t(\xi_1, ... ,\xi_n) = h(t_1, x_1, z_1 ... , t_n, x_n, z_n)\ind_{[0,t]^n}(t_1, ..., t_n).
\label{ht1}
\end{align}
With the above notation, we can rewrite \eqref{CE1} as
$\bE[ I_1(\phi)  | \cF_t  ] = I_1(\phi^t)$.

\subsection{Basic stochastic analysis on the Poisson space} \label{SUB22}

Let $N$ be the Poisson random measure on $(Z, \cZ, \fm)$ as in Subsection \ref{SUB21}.
A well-known theorem due to K. It\^o states that the $L^2(\PP)$ probability space generated
by the Poisson random measure $N$ can be written as a direct sum
of mutually orthogonal subspaces:
\begin{align}
L^2(\Om, \s\{N\}, \PP) = \bigoplus_{k\in\mathbb{N}_{\geq 0}} \bC_k ,
\label{CD1}
\end{align}
where $\bC_k$ is called the $k$-th Poisson Wiener chaos associated to $N$;
see \cite{Ito56, Last16, NN18}.

Let us begin with the construction of Poisson Wiener chaoses $\bC_k$, $k\in\N_{\geq 0}$.

\medskip

\noi
$\bul$ {\bf Poisson Wiener chaoses.}
%
The zero-th chaos $\bC_0\simeq \R$
is the set of (almost surely)
constant random variables in $L^2(\Om, \s\{N\}, \PP)$.
We have already defined the first Poisson Wiener chaos
\[
\bC_1:= \big\{ I_1(\phi): \phi\in \fH \big\} ,
\]

\noi
where $I_1(\phi)$ is defined as in \eqref{int1}-\eqref{int2},
and we recall that $\fH = L^2(Z, \cZ, \fm)$.

Now we define $\bC_k$ for $k\geq 2$.
First, we denote by $\cE^0_k$ the set  of simple functions of the form

\noi
\begin{align}
h(\xi_1, ... , \xi_k) = \sum_{i_1, ..., i_k =1}^m \be_{i_1,..., i_k}
\ind_{F_{i_1} \times \cdots \times F_{i_k}}(\xi_1, ... , \xi_k),
\label{int3a}
\end{align}

\noi
where $m\in\N_{\geq 1}$,  $F_1, ... , F_m\in\cZ$ are pairwise disjoint
sets of finite measures, and the coefficients $\be_{i_1, ... , i_p}$
vanish whenever any two of the indices $i_1, ... , i_k$ are equal.
It is known that because of the atom-less nature\footnote{Even if $\nu$ may not
be atom-less, the product measure $\fm = \Leb\times \nu$ on $(Z, \cZ)$ does
not have any atom.}
of the $\s$-finite measure space $(Z, \cZ, \fm)$, the set $\cE^0_k$
is dense in $\fH^{\otimes n} \equiv  L^2(Z^n)$; see, for example,
\cite[page 10]{Nua06}.
Since $\ind_{F_i}$ can be further
approximated by functions as in \eqref{int0},
we will then work with the  dense subset $\cE_k$ of $\fH^{\otimes n}$
that consists of simple functions $h\in\cE_k^0$ as in \eqref{int3a}
such that $F_i = A_i \times B_i \times C_i$
for some $(A_i, B_i, C_i)\in\cB(\R_+)\times\cB(\R)\times \cB(\R_0)$
with $\fm(F_i)<\infty$,
$i=1,2, ...,m$.
For such a simple function $h\in\cE_k$ as in \eqref{int3a},
we define

\noi
\begin{align}
I_k(h) = \sum_{i_1, ... , i_k=1}^m \be_{i_1, ... , i_k}
\prod_{j=1}^k \wh{N}(  A_{i_j} \times  B_{i_j} \times C_{i_j}     ),
\label{int3b}
\end{align}

 \noi
and the  following properties hold, as one can easily verify:
 \begin{itemize}

 \item[(i)] for $h\in \cE_k$, $I_k(h) = I_k( \wt{h})$,
 with $\wt{h}$ denoting the canonical symmetrization of $h$;
 see \eqref{defh2};

 \item[(ii)]

for $h_1\in \cE_k$ and $h_2\in\cE_\ell$ ($k, \ell\in\N_{\geq 1}$),

\noi
\begin{align}
\bE[  I_k(h_1)I_\ell(h_2) ] = k! \ind_{\{ k= \ell\}}
\langle \wt{h}_1, \wt{h}_2 \rangle_{\fH^{\otimes k}} ;
\label{int3c}
\end{align}

\item[(iii)]

for    $h\in \cE_k$ as in \eqref{int3a}, $I_k(h)$ as in \eqref{int3b}, and
for $t \in( 0,\infty)$, we have

\noi
\begin{align}
\begin{aligned}
\bE[  I_k(h) |   \cF_t  ]
&=   \sum_{i_1, ... , i_k=1}^m \be_{i_1, ... , i_k}
\prod_{j=1}^k \wh{N}\big(  ( A_{i_j} \cap [0,t] )  \times  B_{i_j} \times C_{i_j}     \big) \\
&= I_k (h^t  ) ,
\end{aligned}
\label{rel_3c}
\end{align}

\noi
where $h^t$ is introduced in \eqref{ht1}.

 \end{itemize}
 The relation \eqref{int3c} in property (ii) is known as the orthogonality
 and the $k=\ell$ case gives the modified isometry on $\cE_k$, and hence
 allows one to define for any $h\in\fH^{\otimes k}$,

 \noi
 \begin{align}
 I_k(h) := \lim_{n\to\infty} I_k(h_n)\,\,\, \text{in $L^2(\PP)$},
 \label{int4}
 \end{align}

 \noi
 where $h_n\in\cE_k$ converges to $h$ in $\fH^{\otimes k}$ as $n\to\infty$.
 This defines the $k$-th Poisson Wiener chaos associated to $N$:
 \begin{align}
 \bC_k := \{ I_k(h) : h\in \fH^{\otimes k} \} = \{ I_k(h) : h\in \fH^{\odot k} \}.
\notag 
 \end{align}
 We call  $ I_k(h) $ the $k$-th multiple integral of $h$ with respect to
 the compensated Poisson random measure $\wh{N}$.
 Note that the properties (i)-(iii) still hold for general functions $h, h_1\in \fH^{\otimes k}$
 and  $h_2\in\fH^{\otimes \ell}$:

 \noi
 \begin{align}
 \bE[ I_k(h_1)I_\ell(h_2) ] &=  k! \ind_{\{ k= \ell\}}
\langle \wt{h}_1, \wt{h}_2 \rangle_{\fH^{\otimes k}},
  \label{int3d}  \\
\bE[  I_k(h) |   \cF_t  ]
&= I_k (h^t  ) \,\,\, \text{with $h^t$ as in \eqref{ht1}}.  \label{int3db}
 \end{align}

 \noi
Then  the  chaos decomposition  \eqref{CD1}
 reads as follows: for any $F\in L^2(\Om, \s\{N\}, \PP)$,

 \noi
 \begin{align}
 F = \bE[F] + \sum_{n=1}^\infty I_n( f_n),
 \label{CD3}
 \end{align}

 \noi
 where $f_n\in \fH^{\odot n}$, $n\in\N_{\geq 1}$, are uniquely determined by $F$
 up to a null set with respect to $\fm$;
 see also \cite[Section 4]{Last16}.
 Using  \eqref{int3d},
we have
 \begin{align}
\Var(F) =  \sum_{n=1}^\infty n! \| f_n \|^2_{\fH^{\otimes n}} < \infty.
 \label{CD4}
 \end{align}

Unlike in the Gaussian setting, elements in a Poisson chaos may not
have all the moments and product of two random variables in Poisson chaoses
may not be in a sum of finitely many chaoses.

\medskip
 \noi
$\bul$ {\bf Product formula.}
For  $f \in \fH^{\otimes n}$ and $g \in \fH^{\otimes m}$ with $m,n\in\N_{\geq 1}$,
we define the {\it modified contractions} as follows:
\begin{itemize}

\item[(i)] $f\star^0_0 g = f\otimes g$ is the usual tensor product of $f$
and $g$;

\item[(ii)] for $1\leq k\leq n\wedge m$,  $f\star^0_k g$ is a real measurable function
on $Z^{m+n-k}$, given by

\noi
\begin{align}
\begin{aligned}
&(\ze_1, ... , \ze_k, \xi_1, ... , \xi_{n-k}, \th_1, ... , \th_{m-k}) \\
&\qquad \longmapsto
f(\ze_1, ... , \ze_k,  \xi_1, ... , \xi_{n-k})  g(\ze_1, ... , \ze_k, \th_1, ... , \th_{m-k}),
\end{aligned}
\label{rule1}
\end{align}

\noi
where $\ze_1, ... , \ze_k, \xi_1, ... , \xi_{n-k}, \th_1, ... , \th_{m-k}$ are points in
$Z = \R_+\times\R\times\R_0$;

\item[(iii)] for $1\leq \ell \leq k \leq n\wedge m$, $f\star^\ell_k g$ is a real measurable function
on $Z^{m+n-k-\ell}$, given by

\noi
\begin{align}
\begin{aligned}
&(\ze_1, ... , \ze_{k-\ell}, \xi_1, ... , \xi_{n-k}, \th_1, ... , \th_{m-k}) \\
&  \longmapsto
\int_{Z^\ell}f(\ga_1, ... , \ga_\ell, \ze_1, ... , \ze_{k-\ell},  \xi_1, ... , \xi_{n-k})
g(\ga_1, ... , \ga_\ell, \ze_1, ... , \ze_{k-\ell}, \th_1, ... , \th_{m-k})
\, \fm(d{\pmb\ga_{\pmb\ell}}) .
\end{aligned}
\label{rule2}
\end{align}

\end{itemize}
In other words,  $f\star^\ell_k g$ is obtained by first fixing $k$ arguments of both $f$
and $g$, and then integrating out $\ell$ variables out of these fixed arguments
according to the rules \eqref{rule1}-\eqref{rule2}.
When $k=\ell$ in \eqref{rule2}, $f\star^k_k g$ coincides with
the usual $k$-contraction $f\otimes_k g$
and by Cauchy-Schwarz's inequality, $f\star^k_k g\in\fH^{\otimes n+m-2k}$; see, for example, \cite[Appendix B]{NP12}. However, for $\ell < k$,
$f\star^\ell_k g$ may not belong to $\fH^{\otimes n+m-k-\ell}$.
For example, given $f\in\fH$,  $f\star^0_1 f \in \fH= L^2(Z, \cZ,  \fm)$ if and only if
$f\in L^4(Z, \cZ, \fm)$.

The next result gives a product formula for elements of Poisson Wiener chaoses.
It was first proved by Kabanov for $m=1$ (see \cite[Theorem 2]{Kab75})
and extended by Surgailis to a product of several elements of chaoses
(see \cite[Proposition 3.1]{Sur84}).
The form that we present below corresponds to    \cite[(9.22)]{NN18}
and  Proposition 5 in \cite[page 22]{Last16}; see also \cite[Proposition 2.1]{DP18b}.

\begin{proposition}[Product Formula]
\label{prop:prod}
Let $f \in \fH^{\odot n}$ and $g \in \fH^{\odot m}$ be such that
$f \star_{k}^{\ell} g\in \fH^{\otimes (m+n-k-\ell)}$
for any $k=1,\ldots, n\wedge m$ and $\ell=0,1,\ldots,k$.
Then,
\[
I_n(f) I_m(g)
= \sum_{k=0}^{n\wedge m}k! \binom{n}{k}\binom{m}{k}\sum_{\ell=0}^{k}\binom{k}{\ell}I_{n+m-k-\ell}(f \star_{k}^{\ell} g ).
\]
\end{proposition}

When $f\star^1_k g = 0$,  we deduce from the definition of modified contractions
that   $f\star^\ell_k g = 0$ for all $\ell = 2, ... , k$.
In this case, we have a simpler
form of the product formula.

\begin{proposition} \label{prop:prod2}

Let $f \in \fH^{\otimes n}$ and $g \in \fH^{\otimes m}$ be not necessarily symmetric
such that $\wt{f} \star_{k}^{\ell} \wt{g}\in \fH^{\otimes(n+m-k-\ell)}$
for any $k=1, \ldots, n\wedge m$ and $\ell=1,\ldots,k$.
Suppose  $\wt{f}\star_{k}^{1} \wt{g}=0$ for any $k=1,\ldots,n\wedge m$.
Then,
\[
I_n(f)I_m(g)=I_{n+m}(f \otimes g)+\sum_{k=1}^{n\wedge m}k! \binom{n}{k}\binom{m}{k}I_{n+m-k}(\widetilde{f} \star_{k}^{0} \widetilde{g}).
\]
\end{proposition}

\begin{proof} As $(I_n(f), I_m(g) ) =  (I_n(\wt{f}), I_m(\wt{g}) )$,
the desired product formula follows from Proposition \ref{prop:prod},
the fact that $\wt{f} \star^\ell_k \wt{g} = 0$ for all $1\leq \ell \leq k$,
and by  noting that
$\wt{f} \otimes \wt{g}$ and
$f \otimes g$ have the same symmetrization.
\qedhere

\end{proof}

\medskip

\noi
$\bul$ {\bf Malliavin derivatives.} Let $\dom(D)$ denote  the set of random variables
$F$ as in \eqref{CD3} with the  symmetric kernels $\{f_n\}_n$ satisfying
\begin{align}
\sum_{n=1}^\infty n! n \| f_n\|^2_{\fH^{\otimes n}} < \infty.
\notag 
\end{align}
For such a random variable $F\in\dom(D)$, we define the Malliavin derivative
$DF$ of $F$ to be a $\fH$-valued random variable, given by
\begin{align}
D_\xi F = \sum_{n=1}^\infty n I_{n-1}(f_n(\xi, \bul) ), \,\, \, \xi\in Z,
\label{CD5b}
\end{align}

\noi
 where for fixed $\xi\in Z$, $f_n(\xi, \bul)\in\fH^{\odot (n-1)}$.
 By using orthogonality relation \eqref{int3c}, we have
 \[
 \bE\big[ \| DF \|_{\fH}^2  \big] = \sum_{n=1}^\infty n! n \| f_n\|^2_{\fH^{\otimes n}} <\infty.
 \]
Comparing this equality with \eqref{CD4} yields
the following Poincar\'e inequality:
\begin{align}
\Var(F) \leq \bE\big[ \| DF \|_{\fH}^2  \big]
\label{Poi1}
\end{align}
for any $F\in\dom(D)$,
with equality when and only when $F\in\bC_0 \oplus \bC_1$.

Similarly, we can define the second Malliavin derivative $D^2 F$ as follows:
for $F$ as in \eqref{CD3},

\noi
\begin{align}
\begin{aligned}
D^2_{\ze,\xi} F &:= D_\xi D_\ze F
= \sum_{n=2}^\infty n(n-1) I_{n-2}( f_{n-2}(\ze, \xi, \bul) ),
\end{aligned}
\label{CD6}
\end{align}
provided the above series in \eqref{CD6} converges in $L^2(\PP)$.
That is, the domain of $D^2$ is given by
\[
\dom(D^2) = \Big\{ \text{$F$ as in \eqref{CD3}} :
\sum_{n=2}^\infty n^2 n! \| f_n \|^2_{\fH^{\otimes n}} < \infty \Big\} .
\]

\medskip

\noi
$\bul$ {\bf Kabanov-Skorohod integral $\dl$.} This is an adjoint operator of $D$,
characterized by the following duality relation:

\noi
\begin{align}
\bE[ \langle DF, V \rangle_\fH  ] = \bE[ F \dl(V) ] 
\label{dualR}
\end{align}

\noi
for any $F\in\dom(D)$.
In view of Riesz's representation theorem, we let $\dom(\dl)$ be the set of
$V\in L^2(\Om ; \fH)$ such that there is some finite constant $C=C(V) > 0$ such that
\[
\big| \bE[ \langle DF, V \rangle_\fH  ]  \big| \leq C \| F \|_2
\]
for any $F\in\dom(D)$.
Then, the duality relation \eqref{dualR} holds for any $(F, V)\in \dom(D)\times \dom(\dl)$.

Suppose $V\in L^2(\Om; \fH)$. Then, for $\fm$-almost every $\xi\in Z$,
$V(\xi)\in L^2(\PP)$ by Fubini's theorem. Then, by chaos decomposition,
we can write
\begin{align}
V(\xi) = \bE[ V(\xi) ] + \sum_{n=1}^\infty I_n\big(h_n(\xi, \bul) \big) ,
\label{CD7}
\end{align}

\noi
where $h_n(\xi, \bul)\in \fH^{\odot n}$ may not be symmetric
in all its $(n+1)$ arguments,
and we write $h_0(\xi) = \bE[ V(\xi) ]$.
Note that $V\in L^2(\Om; \fH)$ forces $h_n\in \fH^{\otimes (n+1)}$ for every $n$.
Assume first that there are finitely many chaoses in the above series \eqref{CD7}:

\noi
\begin{align}
\text{$h_n(\xi,\bul) = 0$ for $n\geq M$,}
\label{CD7b}
\end{align}

\noi
where $M\geq 1$ is any given integer that does not depend on $\xi$ nor $n$.
Then, for $F\in\dom(D)$ having the form \eqref{CD3},
we  deduce from \eqref{CD5b}, \eqref{CD7}, Fubini's theorem
and orthogonality relation \eqref{int3d} that

\noi
\begin{align}
\begin{aligned}
\bE \big[ \langle DF, V \rangle_\fH \big]
&= \bE  \int_Z   \Big( \sum_{n=1}^\infty n I_{n-1}(f_n(\xi, \bul ) \Big)
\Big( \sum_{m=0}^M I_m\big( h_m(\xi, \bul) \big) \Big) \fm(d \xi) \\
&=  \int_Z \sum_{n=1}^M  n! \langle f_n(\xi, \bul ), h_{n-1}(\xi, \bul)
\rangle_{\fH^{\otimes (n-1)}} \fm(d\xi) \\
&= \sum_{n=1}^M  n! \langle f_n, h_{n-1} \rangle_{\fH^{\otimes n}}
= \sum_{n=1}^M n!  \langle f_n,  \wt{h}_{n-1} \rangle_{\fH^{\otimes n}} ,
\end{aligned}
\label{CD8}
\end{align}

\noi
which, together with Cauchy-Schwarz's inequality,
implies that

\noi
\begin{align}
\begin{aligned}
\big| \bE \big[ \langle DF, V \rangle_\fH \big] \big|
&\leq  \bigg(\sum_{n=1}^M n! \| f_n\|^2_{\fH^{\otimes n}}  \bigg)^{\frac12}
\bigg(\sum_{n=1}^M n! \| \wt{h}_{n-1}\|^2_{\fH^{\otimes n}}  \bigg)^{\frac12} \\
&\leq  \| F\|_2
\bigg(\sum_{n=1}^M n! \| \wt{h}_{n-1}\|^2_{\fH^{\otimes n}}  \bigg)^{\frac12}.
\end{aligned}
\label{CD9}
\end{align}
In particular, we proved that for $V\in L^2(\Om; \fH)$ satisfying \eqref{CD7b},
$V$ belongs to $\dom(\dl)$;\footnote{This
also tells us that $\dom(\dl)$ is dense in $L^2(\Om; \fH)$. \label{ft10}}
and in this case, we deduce again from \eqref{CD8}
and \eqref{int3d} that

\noi
\begin{align}
\begin{aligned}
\bE \big[ \langle DF, V \rangle_\fH \big]
& =  \sum_{n=1}^M  \bE\big[  I_n( f_n) I_n(  \wt{h}_{n-1} )  \big] \\
&= \bE \bigg[ F \sum_{n=1}^M I_n(\wt{h}_{n-1}   ) \bigg]
\end{aligned}
\label{CD9b}
\end{align}
for any $F\in\dom(D)$, and thus,
\begin{align}
\dl(V) = \sum_{n=1}^\infty I_n( h_{n-1}   ).
\label{CD9c}
\end{align}

One can easily generalize this particular case of \eqref{CD7b}
to the following result, whose proof is  sketched.

\begin{lemma}\label{lem:dl}
Suppose $V\in L^2(\Om; \fH)$ has the expression \eqref{CD7}
with
\begin{align} \label{condV}
\sum_{n=1}^\infty n! \| \wt{h}_{n-1}\|^2_{\fH^{\otimes n}}  <\infty.
\end{align}
Then, $V\in\dom(\dl)$ and $\dl(V)$ is given as in \eqref{CD9c}.
\end{lemma}

\begin{proof}
Let $V\in L^2(\Omega; \fH)$ be given as in \eqref{CD7} 
subject to the condition \eqref{condV},
and we define 
\[
V_M(\xi) = \bE [ V(\xi)  ] + \sum_{n=1}^M I_n( h_n(\xi, \bul) )
\quad\text{for any integer $M \geq 1$.}
\]

\noi
It is immediate that $V_M$ converges to $V$ in $L^2(\Omega; \fH)$
as $M\to\infty$, and thus,

\noi
\begin{align}\label{CD10a}
\bE\big[ \langle DF, V \rangle_\fH \big] = \lim_{M\to\infty}\bE\big[ \langle DF, V_M \rangle_\fH \big],
\quad
\forall F\in\dom(D).
\end{align}

\noi
In view of the above discussions \eqref{CD7b}--\eqref{CD9b}, 
we have $V_M\in\dom(\dl)$ and 
\[
\dl(V_M) = \sum_{n=1}^M I_n(h_{n-1}),
\]
which converges in $L^2(\Omega)$ to $ \sum_{n=1}^\infty I_n(h_{n-1})$ by \eqref{condV}.
Moreover, we deduce from \eqref{CD10a}, \eqref{CD9b},  Cauchy-Schwarz,
and the condition \eqref{condV} with the orthogonality relation \eqref{int3d}
 that 

\noi
\begin{align*}
\big|  \bE\big[ \langle DF, V \rangle_\fH \big]  \big|
=  \lim_{M\to\infty} \big|     \bE [ F \dl(V_M)  ]\big|  
\leq \| F\|_2  \bigg( \sum_{n=1}^\infty n! \| \wt{h}_{n-1}\|^2_{\fH^{\otimes n}}   \bigg)^{\frac12}.
\end{align*}
This implies that $V\in\dom(\dl)$ and $\dl(V)$, as
 the $L^2(\Omega)$-limit of $\dl(V_M)$, is given by    \eqref{CD9c}.
\qedhere

\end{proof}

As a consequence, for a deterministic function $\phi\in \fH$,
we have
\begin{align}\label{EXT1}
 \dl(\phi) = I_1(\phi).
\end{align}

 The following lemma generalizes \eqref{CE1},
\eqref{rel_3c}, and \eqref{int3db}; it also shows that
the It\^o integral is a particular case
of the Kabanov-Skorohod integral and provides
a Clark-Ocone formula;
 see  Theorems 10.2.7 and  10.4.1  in \cite{NN18}
 for the results for the classical L\'evy processes.

\begin{lemma} \label{lem:CE2}

{\rm (i)}  Suppose that the assumptions in  Lemma \ref{lem:dl} hold and fix $t\in(0, \infty)$.
Recall also the notation \eqref{ht1}.
Then, $V^t \in \dom(\dl)$
and
\[
\bE\big[ \dl(V) | \cF_t \big] = \dl(V^t) = \sum_{n=1}^\infty I_n( h^t_{n-1}   ).
\]

\smallskip
\noi
{\rm (ii)} Suppose $F\in\dom(D)$ is $\cF_t$-measurable
for some fixed $t \in(0,\infty)$.
Then,  $D_{s, y, z} F = 0$ almost surely
for  $\fm$-almost every $(s, y, z)\in (t,\infty)\times\R\times\R_0$.

\smallskip
\noi
{\rm (iii)}  Suppose $F\in\dom(D)$ is $\cF_t$-measurable
for some fixed $t \in(0,\infty)$.
Then, the following Clark-Ocone formula holds:
\begin{align}
F = \bE[F] + \dl(V),
\notag
\end{align}

\noi
where $  (r,y,z)\in Z \mapsto  V(r, y, z) := \bE\big[ D_{r, y, z} F | \cF_r \big]$ belongs to $\dom(\dl)$.

\smallskip
\noi
{\rm (iv)}  Suppose $V\in L^2(\Omega; \fH)$ is $\mathbb{F}$-predictable,
 with $\mathbb{F}$ as in \eqref{filtraF}.
 Then,
$V\in\dom(\dl)$ and
$\dl(V)$ coincides with the  It\^o integral of $V$ against the compensated Poisson random
measure $\wh{N}$:
\begin{align}\label{EXT0}
\dl(V) = \int_0^\infty \int_{\R} \int_{\R_0} V(t, x, z) \wh{N}(dt, dx, dz).
\end{align}

\end{lemma}

\begin{proof}

By going through \eqref{CD8}, \eqref{CD9}, and \eqref{CD9b}
with $M=\infty$ and $V^t$ in place of $V$,
we get $V^t\in \dom(\dl)$ and
\begin{align}
\dl(V^t) =\sum_{n=1}^\infty I_n( h^t_{n-1}   ).
\label{pst1}
\end{align}

On the other hand, since the conditional expectation is a bounded operator
on $L^2(\PP)$, we deduce from \eqref{int3db} that

\noi
\begin{align*}
\bE\big[ \dl(V) | \cF_t \big]
& = \sum_{n=1}^\infty \bE \big[ I_n(h_{n-1}) | \cF_t \big]
 =  \sum_{n=1}^\infty  I_n(h^t_{n-1}),
\end{align*}
which, together with \eqref{pst1}, concludes the proof of (i).

\smallskip

Next, we prove (ii). We can deduce from part (i)
and the duality relation \eqref{dualR}
for several times  that

\noi
\begin{align*}
\bE\big[ \langle DF, V \rangle_\fH \big]
&=  \bE[ F \dl(V) ] = \bE\big[  F \bE( \dl(V) | \cF_t ) \big] \\
&=\bE\big[  F  \dl(V^t)  \big] = \bE\big[ \langle DF, V^t \rangle_\fH \big]
\end{align*}
for any $V\in\dom(\dl)$. It follows that
\[
\bE \big[ \langle  DF- (DF)^t, V \rangle_\fH \big] = 0
\]
for any  $V\in\dom(\dl)$. 
Then, the density of $\dom(\dl)$ in $L^2(\Om; \fH)$
(see, e.g.,  Footnote \ref{ft10})
implies that $(DF)^t = DF$ almost surely.  Therefore, part (ii) is proved.

\smallskip

Now we prove the Clark-Ocone formula in (iii);
see also Theorem 10.4.1 in \cite{NN18}.
Assume that  $F$ has the form \eqref{CD3}. Then,

\noi
\begin{align}
\begin{aligned}
V(r,y,z)  &=  \bE\big[ D_{r, y, z} F | \cF_r \big] \\
& = \sum_{n=1}^\infty n \bE\big[ I_{n-1}(f_n(r,y,z, \bul) ) | \cF_r \big] \\
&=  \sum_{n=1}^\infty n   I_{n-1}(f^r_n(r,y,z, \bul) ) .
\end{aligned}
\label{def:V}
\end{align}

\noi
Put
\[
h_n(t_1, y_1, z_1, ... , t_n, y_n, z_n) = n f_n^{t_1}(t_1, y_1, z_1, ... , t_n, y_n, z_n  ).
\]
Then, (omitting the dummy variables $y_i, z_i$ to ease the notations)
\begin{align*}
\wt{h}_n(t_1, t_2, ... , t_n)
&= \frac{1}{n!} \sum_{\s\in\mathfrak{S}_n}
n f_n^{t_{\s(1)}  }(t_{\s(1)}, t_{\s(2)}, ..., t_{\s(n)} ) \\
&=  \frac{1}{(n-1)!} \sum_{k=1}^n \sum_{\s\in\mathfrak{S}_n}\ind_{\{\s(1)=k\}}
f_n(t_{\s(1)}, t_{\s(2)}, ..., t_{\s(n)} ) \ind_{\{  t_k \geq t_i, \, \forall i\neq k  \}} \\
& = f_n(t_1, ... , t_n) \,\, \,  \text{almost everywhere, since $f_n\in\fH^{\odot n}$.}
\end{align*}

\noi
Therefore, we deduce from Lemma \ref{lem:dl} that
$V$, given as in \eqref{def:V}, belongs to $\dom(\dl)$
and
\begin{align*}
\dl(V) &= \sum_{n=1}^\infty I_n(\wt{h}_n) =  \sum_{n=1}^\infty I_n(f_n)  \\
&= F - \bE[F].
\end{align*}

Finally, we prove the statement (iv).  First we consider the case where
$V$ is an elementary process as in \eqref{simpleX}:
\begin{align}\label{Bella}
V(t,x,z) =Y \ind_{(a, b]\times A\times\Gamma}(t, x, z)
\end{align}
with $Y\in\dom(D)$ bounded $\cF_a$-measurable, $a<b$, and
$\Leb(A) + \nu(\Gamma) <\infty$.
In this case,

\noi
\begin{align*}
\textup{RHS of \eqref{EXT0}}
= Y \wh{N}\big(  (a, b] \times A \times\Gamma\big)
= Y \dl\big( \ind_{(a,b]\times A\times\Gamma} \big),
\end{align*}

\noi
where the last equality follows from \eqref{EXT1}.
Let $F$ be any bounded random variable in $\dom(D)$.
 Then, in view of   Remark \ref{rem:add1} (iv),
we have   $YF\in\dom(D)$ with

\noi
\begin{align*} 
Y D_\xi F  = D_\xi(YF) - FD_\xi Y - (D_\xi F)(D_\xi Y).
\end{align*}

\noi
Thus, we can write 
\begin{align*}
\langle DF, V \rangle_\fH
&= \langle Y DF,   \ind_{(a,b]\times A\times\Gamma} \rangle_\fH \\
&=  \langle  D(YF),   \ind_{(a,b]\times A\times\Gamma} \rangle_\fH
- \langle  FDY,   \ind_{(a,b]\times A\times\Gamma} \rangle_\fH
- \langle  (DF)(DY),   \ind_{(a,b]\times A\times\Gamma} \rangle_\fH;
\end{align*}

\noi
and moreover, by part (ii) of Lemma \ref{lem:CE2},
we get
\begin{align}\label{EXT2}
\langle DF, V \rangle_\fH
&=  \langle  D(YF),   \ind_{(a,b]\times A\times\Gamma} \rangle_\fH.
\end{align}

\noi
Therefore,  we deduce from the duality relation
\eqref{dualR} with \eqref{EXT1} and \eqref{EXT2} that

\noi
\begin{align}
\begin{aligned}
\bE\big[ \langle D F,  V \rangle_\fH \big]
&= \bE\big[ \langle  D(YF),   \ind_{(a,b]\times A\times\Gamma} \rangle_\fH \big]  \\
&= \bE\big[  F   Y \wh{N}\big( (a, b]\times A\times\Gamma \big) \big]
\end{aligned}
\notag
\end{align}
for any $F$ bounded Malliavin differentiable,
 which implies       \eqref{EXT0} with $V\in\dom(\dl)$.

For a general process  $V\in L^2(\Omega; \fH)$ that is predictable,
there is a sequence $\{ V^{(k) } \}_{k\geq 1}$ of elementary processes (i.e. linear combination of functions as
in \eqref{simpleX}) such that
\[
\| V^{(k) }  - V \|_{L^2(\Omega; \fH)} \to 0
\]
as $k\to\infty$; see, e.g.,  \cite{B15}.
 By previous step, we know that \eqref{EXT0}
holds   for $V = V^{(k) }$, $k\geq 1$;
and $\dl(V^{(k) })$ converges in $L^2(\Omega)$ to some limit $G$,
by It\^o isometry.
Applying duality relation \eqref{dualR} again, we see that
$\dl$ is a closed operator meaning that $V$, as the $L^2(\Omega; \fH)$-limit of
$V^{(k) } \in \dom(\dl)$, also belongs to $\dom(\dl)$:
for any $F\in\dom(D)$,

\noi
\begin{align*}
\bE\big[ \langle DF, V \rangle_\fH \big]
&=\lim_{k\to\infty} \bE\big[ \langle DF, V^{(k) }  \rangle_\fH \big] \\
&=\lim_{k\to\infty} \bE\big[  F \dl( V^{(k) } )   \big]
=  \bE\big[  F G  \big].
\end{align*}

\noi
It follows that $V\in\dom(\dl)$ and $\dl(V) = G$. This concludes
the proof of Lemma \ref{lem:CE2}.
\qedhere

\end{proof}

\begin{lemma} \label{lem:Fubi}
Let $(E, \mu)$ be a finite measure space.

\smallskip
\noi
{\rm (i)} Suppose that $F(\th)\in\dom(D)$ for every $\th\in E$ such that

\noi
\begin{align}
\bE \int_E    \big( | F(\th) |^2 +  \| D F(\th) \|^2_\fH \big) \mu(d\th) <\infty.
\label{Fubi1}
\end{align}

\noi
Then, $\int_E F(\th) \mu(d\th)$ belongs to $\dom(D)$ with
\[
D_\xi  \int_E F(\th) \mu(d\th) = \int_E D_\xi F(\th) \mu(d\th)
\]
almost surely for $\fm$-almost every $\xi\in Z$.

\smallskip
\noi
{\rm (ii)} \textup{(Stochastic Fubini's theorem)} Suppose  that $G(\th) \in \dom(\dl)$  for each $\th\in E$ such that
$\int_{E} G(\th) \mu(d\th)$ also belongs to  $\dom(\dl)$
and

\noi
\begin{align}
\bE \int_E  \big(  | \dl( G(\th) ) |^2 +   \| G(\th) \|_{\fH}^2\big)  \mu(d\th) < \infty.
\label{Fubi2}
\end{align}

\noi
Then,

\noi
\begin{align}
\int_{E} \delta\big( G(\th) \big) \mu(d\th)
= \delta \bigg( \int_{E} G(\th) \mu(d\th) \bigg).
\label{Fubi3}
\end{align}

\end{lemma}

\begin{proof}
%
%
(i) Suppose $F(\th)\in\dom(D)$ admits the chaos expansion
\begin{align}
F(\th) =  f_0(\th) + \sum_{n=1}^\infty I_n( f_n(\th)),
\notag 
\end{align}

\noi
where  $f_n(\th)\in \fH^{\odot n}$ for every $n\in\N_{\geq 1}$ and for every $\th\in E$.
Then, the condition \eqref{Fubi1} implies that

\noi
\begin{align}
\sum_{n\geq 1} n! n \int_E \| f_n(\th)\|^2_{\fH} \mu(d\th ) <\infty.
\label{Fubi5}
\end{align}

Fix any $g\in\fH^{\odot n}$ with $n\geq 1$. Then,
we deduce from   \eqref{int3d} and Fubini's theorem
with \eqref{Fubi5} that

\noi
\begin{align*}
\bE\bigg[  I_n(g) I_n\Big(  \int_E f_n(\th) \mu(d\th) \Big) \bigg]
&= n! \int_E \langle f_n(\th), g \rangle_{\fH^{\otimes n}} \mu(d\th) \\
&=  \int_E  \bE\big[ I_n(g) I_n(  f_n(\th) ) \big] \mu(d\th)  \\
&= \bE\bigg[  I_n(g) \int_E  I_n\big(  f_n(\th) \big) \mu(d\th)  \bigg],
\end{align*}

\noi
which implies that almost surely
\begin{align}
I_n\Big(  \int_E f_n(\th) \mu(d\th) \Big)
=  \int_E  I_n\big(  f_n(\th) \big) \mu(d\th).
\label{Fubi6}
\end{align}
Then  it is straightforward to generalize the above argument to
show that

\noi
\begin{align}
\int_E F(\th) \mu(d\th) = \int_E f_0(\th)\mu(d\th)
+  \sum_{n=1}^\infty I_n\Big(  \int_E f_n(\th) \mu(d\th) \Big),
\notag 
\end{align}

\noi
which, together with \eqref{Fubi5}-\eqref{Fubi6} and orthogonality
relation  \eqref{int3d}, implies that
$\int_E F(\th) \mu(d\th)$ belongs to $\dom(D)$
and

\noi
\begin{align*}
D_\xi \int_E F(\th) \mu(d\th)
&= \sum_{n=1}^\infty  n I_{n-1}\Big(  \int_E f_n(\th, \xi, \bul) \mu(d\th) \Big) \\
&= \sum_{n=1}^\infty  n \int_E I_{n-1} \big( f_n(\th, \xi, \bul)\big)  \mu(d\th)  \\
&= \sum_{n=1}^\infty   \int_E  D_\xi I_{n} \big( f_n(\th)\big)  \mu(d\th)
=  \int_E  D_\xi  F(\th)  \mu(d\th)
\end{align*}

\noi
almost surely. This proves (i).

Next, we prove (ii). Let $F\in\dom(D)$.
Then, we deduce from duality relation \eqref{dualR}
and Fubini's theorem with the condition \eqref{Fubi2}
that

\noi
\begin{align}
\begin{aligned}
\bE\bigg[ F \dl  \Big(\int_E G(\th) \mu(d\th)  \Big)  \bigg]
&= \bE\Big\langle DF, \int_E G(\th) \mu(d\th) \Big\rangle_{\fH} \\
&= \bE \int_Z D_\xi F  \int_E G(\th, \xi) \mu(d\th)  \fm(d\xi) \\
&= \int_E \bE   \langle D F, G(\th) \rangle_\fH  \,\mu(d\th)   \\
&=  \int_E \bE \big[  F \dl( G(\th) ) \big]   \,\mu(d\th) \\
&= \bE \bigg[ F  \int_E  \dl( G(\th) )  \,\mu(d\th)  \bigg].
\end{aligned}
\label{Fubi8}
\end{align}
Since $\dom(D)$ is dense in $L^2(\Om, \sigma\{N\}, \PP)$,
we obtain \eqref{Fubi3} from  \eqref{Fubi8}.
\qedhere

\end{proof}

We conclude this subsection with a remark on the add-one cost operator
$D^+_\xi$ that coincides with Malliavin derivative operator $D$
on $\dom(D)$.

\begin{remark} \label{rem:add1} \rm
(i) In this paper, we are mainly concerned with distributional properties.
In view of \cite[Corollary 3.7]{LP18}, we assume that the Poisson
random measure $N$ (from Definition \ref{def:PRM}) is a proper simple point process of the form
\[
N = \sum_{n=1}^\kappa \dl_{Z_n},
\]
where $\{Z_n\}_{n\geq 1}$ are independent random variables with values in $Z$,  $\kappa$ is a
random variable with values in $\N_{\geq 1}\cup\{+\infty\}$,
and $\dl_z$ is the Dirac mass at $z\in Z$. With probability $1$, these points are distinct (since $\fm$ 
is diffusive).

\smallskip
\noi
(ii) Since $N$ is a random variable with values in $\mathbf{N}_\s$, 
according to 
Doob's functional representation,
for any real-valued random variable $F$ that is $\s\{N\}$-measurable,
we can write $F = \ff(N)$ for some  representative $\ff: \bfN_\s\to \R$
that is $\scrN_\s$-measurable; see, e.g., \cite[Lemma 1.14]{Kall21}.
 With such a functional representation,
  the add-one cost operator is given by
\[
D^+_\xi F :=  \ff(N+\dl_\xi) - \ff(N).
\]
 Since the points $\{Z_n\}_{n\geq 1}$ in the above representation are distinct, 
 the ``add-one'' cost operator indeed adds one more point $\xi$ to this representation, 
 hence justifying its name.
It is known that  for $F\in\dom(D)$, one has $D^+F = DF$; see, e.g.,
\cite[Theorem 3]{Last16}. 
A similar result also holds: if $\bE \int_Z | D^+ F |^2 \fm(d\xi) <\infty$,
then $F\in\dom(D)$ and $D^+ F = DF$;
see, e.g., \cite[Lemma 3.1]{PT13}.

\smallskip
\noi
(iii) Suppose that $F = \ff(N) \in\dom(D)$ and $\phi: \R\to\R$ is Lipschitz continuous
with Lipschitz constant $\Lip(\phi)$. Due to a lack of derivation property of $D$,
the neat chain rule $D^+ \phi(F) = \phi'(F) D^+ F$ does not hold in general.
Nevertheless, one has  $\phi(F)\in\dom(D)$. Indeed,

\noi
\begin{align}\notag
\big| D^+_\xi \phi(F) \big| &= \big|\phi( \ff(N+\dl_\xi)) - \phi( \ff(N)) \big|
\leq \Lip(\phi)  | D^+_\xi F |,
\end{align}

\noi
which, together with (ii), implies that $\phi(F)\in\dom(D)$
with

\noi
\begin{align}\label{add1b}
|D_\xi \phi(F)| \leq  \Lip(\phi)| D_\xi F|.
\end{align}

\noi
This   leads to a generalization of the Poincar\'e inequality
\eqref{Poi1}:
\begin{align}
\Var  ( \phi(F)  ) \leq  \Lip^2(\phi) \bE [ \| DF \|_\fH^2  ].
\label{Poi1b}
\end{align}

\noi
 Note that the inequalities \eqref{add1b}-\eqref{Poi1b} will be used
in the proof of Theorem \ref{thm:main} (i); see \eqref{ERG2}.
Besides, one can observe that 
 for any $F\in\dom(D)$, the truncated random variable $F_M : =( M \wedge F) \vee (-M)$
is a bounded random variable that  belongs to $\dom(D)$ for any $M>0$.
 Such an observation has been implicitly used in the proof of Lemma \ref{lem:CE2} (iv);
see \eqref{Bella} and see also Footnote \ref{ft77}.

\smallskip
\noi
(iv) 
Let $\cA = L^\infty(\Omega, \s\{N\}, \mathbb{P}) \cap \dom(D)$.
Then, $\cA$ is stable under multiplications.
Indeed, for $F = \ff(N), G= \fg(N)\in \cA$ 
(with $\ff, \fg$ bounded $\mathscr{N}_\s$-measurable),
we have 

\noi
\begin{align}
\begin{aligned}
D^+_\xi (FG)
& = \ff(N+\dl_\xi) \fg(N + \dl_\xi) -  \ff(N) \fg(N) \\
&= \big[  \ff(N+\dl_\xi) - \ff(N) \big] \fg(N) +   \ff(N) \big[  \fg(N+\dl_\xi) -  \fg(N)\big]   \\
&\qquad\qquad  +  \big[  \ff(N+\dl_\xi) - \ff(N) \big]  \cdot \big[ \fg(N+\dl_\xi) - \fg(N) \big]    \\
& =  FD^+_\xi G +   GD^+_\xi F +  (D^+_\xi F) D^+_\xi G
\end{aligned}
\notag 
\end{align}

\noi
 with $D^+_\xi F,  D^+_\xi G$ uniformly bounded, 
so that  $D^+(FG)\in L^2(\Om; \fH)$. This implies $FG\in\dom(D)$,
in view of the aforementioned result from \cite[Lemma 3.1]{PT13}.
Therefore, $\cA$ is stable under multiplications.
In particular, we can write for 
 $F, G\in \cA $ that 
 
 \noi
\begin{align} \notag
 D_\xi (FG)
= FD_\xi G + GD_\xi F  +  (D_\xi F)D_\xi G
\end{align}
almost surely for $\fm$-almost every $\xi\in Z$.

\end{remark}

\subsection{Poincar\'e inequalities} \label{SUB23}
Recall from the Poincar\'e inequality \eqref{Poi1}  (see also \eqref{Poi1b})
that the variance $\Var(F)$ of a Malliavin differentiable random variable  $F$
is controlled by the first Malliavin derivative $DF$.
That is, if $\| DF \|_\fH$ is typically small, then
the random variable $F$ has small fluctuations.
It was first in a paper by Chatterjee \cite{Ch09} that
a possible second-order extension of (Gaussian) Poincar\'e
inequality was investigated.
Suppose  $F = g(X_1, ... , X_m) $ is a nice function of i.i.d. standard normal
random variables $\{ X_i\}_{i=1}^m$.
If the squared operator norm of
the Hessian matrix $\nb^2 g(X_1, ... , X_m)$ is typically smaller compared
to $\nb g(X_1, ... , X_m)$, then $F$ is close to a linear combination of $X_i$'s
and thus
 approximately Gaussian,
with  the proximity  measured in total-variation distance \eqref{distTV};
see Theorem 2.2 in \cite{Ch09}
within the development of Stein's method. 
 This quantitative bound is
then known as the second-order Gaussian Poincar\'e inequality.
And it has been generalized by Nourdin, Peccati, and Reinert \cite{NPR09}
to the case where $F$ is a general Malliavin differentiable
random variable (with respect to an isonormal Gaussian process)
and may depend on infinitely many coordinates (e.g.,
$F= g( \{X_i\}_{i\in\N})$). See also Vidotto's improvement in
\cite{Vid20}. In a recent  joint work \cite{BNQSZ} with Nualart and Quer-Sardanyons,
we implemented this second-order Gaussian Poincar\'e
inequality to prove the quantitative central limit theorem (CLT)
for stochastic wave equation driven by colored-in-time Gaussian noise.
See also a  study for stochastic heat equation in
 \cite{NXZ22} by Nualart, Xia, and the second author.

\medskip
\noi
$\bul$ {\bf Second-order Poincar\'e inequality on the Poisson space.}
In  \cite{LPS16},  Last, Peccati, and Schulte extended the second-order
Gaussian Poincar\'e inequality to the Poisson setting. 
One can  apply the results in \cite{LPS16} to obtain the quantitative
CLTs under the assumption of finite $m_4$,
which is a more restrictive assumption than \eqref{cond:al}.
In a recent work  \cite{Tara},   T. Trauthwein has 
improved the second-order Poincar\'e
inequalities by  imposing   minimal moment assumptions.
With this new ingredient,
we are able to obtain the quantitative CLT (and a corresponding functional CLT)
for the hyperbolic Anderson model \eqref{wave}
under the   assumption \eqref{cond:al};
see Theorem \ref{thm:main}.

Let us first introduce several distances  for distributional approximation.
Suppose  $F, G$ are   real random variables with
distribution measures $\mu$ and  $\nu$, respectively.

\medskip
\noi
(i)  $d_{\rm FM}$ denotes the Fortet-Mourier metric, also known
as the bounded Wasserstein distance:

\noi
\begin{align}
\begin{aligned}
d_{\rm FM}(F, G)  &= d_{\rm FM}(\mu, \nu) \\
&= \sup\big\{ | \bE[ h(F)  ]  - \bE[ h(G)  ] | :
\|h \|_\infty + \Lip(h) \leq 1 \big\}.
\end{aligned}
\notag 
\end{align}

\noi
It is well known that $d_{\rm FM}$ characterizes the weak convergence
on $\R$.

\medskip
\noi
(ii) $d_{\rm Wass}$ denotes the
$1$-Wasserstein distance:

\noi
\begin{align}
\begin{aligned}
d_{\rm Wass}(F, G)  &= d_{\rm Wass}(\mu, \nu) \\
&= \sup\big\{ | \bE[ h(F)  ]  - \bE[ h(G)  ] | :
  \Lip(h) \leq 1 \big\}.
\end{aligned}
\notag 
\end{align}

\noi
It is trivial that  $d_{\rm Wass}(F, G) \geq d_{\rm FM}(F, G)$.

\medskip
\noi
(iii) $d_{\rm Kol}$ denotes the Kolmogorov distance:

\noi
\begin{align}
\begin{aligned}
d_{\rm Kol}(F, G)  &= d_{\rm Kol}(\mu, \nu) = \sup\big\{ | \bE[ \ind_{(-\infty, t] }(F)  ]  - \bE[ \ind_{(-\infty, t] }(G)  ] | :
 t\in\R \big\} \\
 &= \sup\big\{ |  \PP(F\leq t) - \PP(G\leq t)   | :
 t\in\R \big\}.
\end{aligned}
\notag 
\end{align}

\noi
Kolmogorov distance is   a very natural metric in studying the normal approximation,
in view of the fact that for a sequence of real-valued random variables $\{F_n\}_{n\in\N}$,
$F_n$ converges in law to a standard normal random variable $Y$ (i.e.
$d_{\rm FM}(F_n, Y)\to 0$)
 if and only if
$d_{\rm Kol}(F_n, Y) \to 0$ as $n\to\infty$; see \cite[Proposition C.3.2]{NP12}.
It is also well known that
\begin{align}
d_{\rm Kol}(F, Y) \leq   \sqrt{ d_{\rm Wass}(F, Y)},
\label{KolWass}
\end{align}

\noi
when $Y \sim \NN(0,1)$;
see, for example, \cite[Proposition 1.2]{Ross11}.

\medskip
\noi
(iv)   The aforementioned total-variation distance is defined by
\begin{align}
\begin{aligned}
d_{\rm TV}(F, G) &= d_{\rm TV}(\mu, \nu)
= \sup\big\{ | \PP(F\in B) - \PP(G\in B) | :
 B\in\cB(\R) \big\}.
\end{aligned}
\label{distTV}
\end{align}

\noi
It is trivial that $d_{\rm TV}(F, G) \geq d_{\rm Kol}(F, G)$.
The total-variation distance is much stronger than weak convergence.
For example, consider $\{Y_i\}_{i\in\N}$ i.i.d. Poisson random variables
with mean $1$,  $F_n:=\frac{1}{\sqrt{n}} (Y_1+ ... + Y_n - n)$,
which is an element
of the first Poisson Wiener chaos $\bC_1$,
converges in law to $Y\sim\NN(0,1)$ as $n\to\infty$; while due to discrete nature
of $F_n$, $d_{\rm TV}(F_n, Y) = 1$ for all $n$. For this reason, we will not consider
total-variation distance for our quantitative CLTs.

\medskip

In what follows,  we present  the second-order $p$-Poincar\'e inequality by Trauthwein
\cite{Tara}.\footnote{The bounds in \cite{Tara} are stated in terms of add-one cost operator $D^+$.
There, Trauthwein used the notation $D$ to denote the add-one cost operator. As these two
operator coincide on $\dom(D)$, the notational difference shall not cause any
ambiguity for readers of  the current paper.
}
Recall that in our paper, all Poisson functionals are defined over the Poisson
random measure $N$ on $Z=\R_+\times\R\times\R_0$
with intensity measure $\fm = \Leb\times\nu$; see \eqref{ZZm}-\eqref{LevyM}
and Section \ref{SEC2}.

\begin{proposition} 
\label{prop:tara}   \textup{(\cite[Theorem 3.4]{Tara})}
Let $F\in\dom(D)$ with $\bE[F]=0$ and  $\Var(F)=\sigma^2 > 0$. Then, 
for any $p,q\in (1,2]$,

\noi
\begin{align}
d_{\rm FM}\left( \frac{F}{\sigma}, Y\right) \leq d_{\rm Wass}\left(\frac{F}{\sigma}, Y\right)  \leq    \ga_1 +  \ga_2 +  \ga_3
\label{2nd:WassB}
\end{align}
and

\noi
\begin{align}
  d_{\rm Kol}\left(\frac{F}{\sigma}, Y\right)  \leq  \sqrt{ \frac{\pi}{2}} (  \ga_1 +  \ga_2)   +  \ga_4 +  \ga_5+  \ga_6 +  \ga_7,
\label{2nd:KolB}
\end{align}

\noi
where  $Y\sim \NN(0,1)$ and the seven quantities $ \ga_1, ...,  \ga_7$ are given as follows:

\noi
\begin{align}
\begin{aligned}
 \ga_1&:=  \frac{2^{\frac2p + \frac12} }{\sqrt{\pi}}
 \sigma^{-2} \bigg( \int_{Z}  \bigg[ \int_Z  \| D^+_{\xi_2}F \|_{2p}  \| D^+_{\xi_1}D^+_{\xi_2}F\|_{2p} 
  \,  \fm(d\xi_2) \bigg]^p
  \fm(d\xi_1) \bigg)^{\frac1p}
\\
 \ga_2&:=\frac{2^{\frac2p - \frac12} }{\sqrt{\pi}}
\sigma^{-2} \bigg( \int_{Z}  \bigg[ \int_Z   \| D^+_{\xi_1}D^+_{\xi_2}F\|^2_{2p}  \,  \fm(d\xi_2) \bigg]^p
  \fm(d\xi_1) \bigg)^{\frac1p}
\\
 \ga_3&:= 2 \sigma^{-(q+1)}\int_Z \| D^+_\xi F \|_{q+1}^{q+1}  \, \fm(d\xi)
\\
 \ga_4&:= 2^{\frac{2}{p}} \sigma^{-2} \bigg( \int_Z \| D^+_\xi F \|_{2p}^{2p}  \, \fm(d\xi) \bigg)^{\frac1p}
\\
 \ga_5&:= (4p)^{\frac1p} \sigma^{-2} \bigg( \int_{Z^2} \| D^+_{\xi_1}D^+_{\xi_2}  F \|_{2p}^{2p}
\, \fm(d\xi_1)\fm(d\xi_2)  \bigg)^{\frac1p}
\\
 \ga_6&:=(2^{2+p}p)^{\frac1p} \sigma^{-2} \bigg( \int_{Z^2} \| D^+_{\xi_1}D^+_{\xi_2}  F \|_{2p}^{p}
 \| D^+_{\xi_1}F \|_{2p}^p
\, \fm(d\xi_1)\fm(d\xi_2)  \bigg)^{\frac1p},
\end{aligned}
\label{gamma16}
\end{align}
\noi
and
\begin{align}
 \ga_7:= (8p)^{\frac1p} \sigma^{-2}
 \bigg( \int_{Z^2} \| D^+_{\xi_1}D^+_{\xi_2}  F \|_{2p}
 \| D^+_{\xi_1}F \|_{2p}  \| D^+_{\xi_2}F \|_{2p}^{2(p-1)}
\, \fm(d\xi_1) \fm(d\xi_2)  \bigg)^{\frac1p}.
\label{gamma7}
\end{align}

\end{proposition}

Recall from Remark \ref{rem:add1} that $D^+$ denotes the add-one cost operator
that coincides with Malliavin derivative oprerator $D$ on $\dom(D)$.
The quantities  $\ga_1, \ga_2$ control the size of the fluctuations
of the second-order difference operator in a relative and an absolute
way so that a small size of $\ga_1 + \ga_2$
leads to the proximity of $F$ to its projection to the first Poisson Wiener chaos $\bC_1$.
And a small value of $\ga_3$ heuristically indicates that this projection to $\bC_1$
is close in distribution to a Gaussian random variable.
See \cite{LPS16, Tara} for more discussions. 
Note that the estimations of the three quantities $\ga_1, \ga_2$, and $\ga_3$
are sufficient to control the rate of convergence in the Wasserstein distance. 
Within Stein's method,
 it is in general
much more difficult to prove bounds in the Kolmogorov distance 
than to prove bounds in the Wasserstein distance, while maintaining  
the same rate of convergence.
In the current application, 
we will need to further estimate four more quantities 
($\ga_4, \ga_5$, $\ga_6$, and $\ga_7$), and we can get  
the same rate of convergence. 
Note that in view of the bound \eqref{KolWass}, we deduce from \eqref{2nd:WassB}
that
\[
d_{\rm Kol}\left(\frac{F}{\sigma}, Y\right) \leq \sqrt{\ga_1 +  \ga_2 + \ga_3},
\]
which would   lead to sub-optimal rates compared to  \eqref{2nd:KolB}.

\subsection{Moment inequalities}

Recall the definition of $G_t$ from \eqref{FSol} and define
\begin{align}
\begin{aligned}
\varphi_{t, R}(r,y)  &:=  \int_{-R}^R G_{t-r}(x-y)dx.
\end{aligned}
\label{Rosen6b}
\end{align}

\noi
We record below a few simple facts.

\begin{lemma} \label{lem:G}
{\rm (i)} For $t\in\R_+$, we have

\noi
\begin{align}
\begin{aligned}
 \int_{\R} G_t(y)dy &= t .
\end{aligned}
\label{factsG}
\end{align}

\smallskip
\noi
{\rm (ii)} For  $t  \geq s >0$, we have   $0 \leq  \varphi_{t,R} - \varphi_{s,R} \leq t -s$
and

\noi
\begin{align}
\begin{aligned}
\int_\R   \big[ \varphi_{t, R}(r,y)  - \varphi_{s, R}(r,y) \big] dy &= 2(t-s)R
\end{aligned}
\label{Rosen6c}
\end{align}

\noi
for any $r \in (0, s]$.

\smallskip
\noi
{\rm (iii)} For $0< s<t$,  we have
\noi
\begin{align}
\begin{aligned}
 \int_s^t  \int_{\R}  \varphi^2_{t, R}(r,y)   \, drdy
&\leq \frac{4}{3} R   (t-s)^3
\\
\int_s^t  \int_{\R}  \varphi^4_{t, R}(r,y)   \, drdy
&
\leq   2 R^2 (t-s)^4.
\end{aligned}
\label{Rosen7c} 
\end{align}
As a consequence, we have

\noi
\begin{align}
\int_s^t  \int_{\R}  \varphi^p_{t, R}(r,y)   \, drdy
\leq
\begin{cases}
 2^{\frac p2} R^{\frac p2} (t-s)^{2 + \frac p 2 } \,\,\,  &\text{for $p\in[2,4]$} \\
   2^{p-1} (t-s)^p R^2   &\text{for $p\in(4,\infty)$.}
 \end{cases}
\label{Rosen7z}
\end{align}

\end{lemma}

\begin{proof} (i) is trivial. Let us prove (ii) now.

Let $t \geq s \geq 0$. Then,
\begin{align}
 \varphi_{t,R}(r,y)  - \varphi_{s,R}(r,y)
 = \frac{1}{2} \int_{-R}^R  \ind_{\{ s-r \leq |x-y | < t-r  \}} dx,
 \label{LG1}
\end{align}
which   implies that
$ \varphi_{t,R}(r,y)  - \varphi_{s,R}(r,y) \in  [0,t-s]$ for any $(r, y)\in\R_+\times\R$.

It is also  easy to see from   \eqref{LG1}
that for $0< r \leq s$

\noi
\begin{align*}
\int_\R   \big[ \varphi_{t, R}(r,y)  - \varphi_{s, R}(r,y) \big] dy
& = \frac{1}{2} \int_{-R}^R  \bigg( \int_\R  \ind_{\{ s-r \leq |x-y | < t-r  \}}  dy \bigg)dx \\
&= 2(t-s)R.
\end{align*}
That is, the equality \eqref{Rosen6c} is proved.

To prove the first  bound in part (iii),
  we write
\noi
\begin{align}
\begin{aligned}
& \int_s^t    \bigg( \int_{\R}  \varphi^2_{t, R}(r,y)   \, dy  \bigg) dr \\
&\quad
= \int_s^t  \bigg[  \int_{\R}  \bigg( \int_{-R}^R  \int_{-R}^R  G_{t-r}(x_1-y) G_{t-r}(x_2-y) dx_1dx_2 
\bigg) dy   \bigg] dr  \\
&\quad
\leq \int_s^t  \int_{-R}^R   \bigg[ \int_{-R}^R  G_{2t-2r}(x_1-x_2)
\bigg( \int_{\R}  G_{t-r}(x_2-y) dy \bigg)
 dx_1 \bigg] dx_2 dr   \\
&\quad
\leq \int_s^t 2(t-r)^2 \cdot 2R dr =  \frac{4}{3} R |t-s|^3,
\end{aligned}
\label{Rosen7d} 
\end{align}

\noi
where the second step in \eqref{Rosen7d}  follows from the triangle inequality
\[
\ind_{\{  |x_1- y| <t-r  \}}\cdot \ind_{\{  |x_2- y| <t-r  \}}
\leq \ind_{\{  |x_1- x_2| <2t-2r  \}}\cdot \ind_{\{  |x_2- y| <t-r  \}}.
\]
And similarly,

\noi
\begin{align}
\begin{aligned}
 \int_s^t  \bigg( \int_{\R}  \varphi^4_{t, R}(r,y)   \,  dy  \bigg) dr 
&= \int_s^t    \bigg[  \int_{\R} \bigg( \int_{[-R,R]^4} \prod_{j=1}^4 G_{t-r}(x_j-y)   d\pmb{x_4} \bigg)   dy  \bigg] dr \\
& 
\leq  \int_s^t  \int_{[-R, R]^2} G_{2t-2r}(x_1-x_2)   \bigg[ \int_{[-R,R]^2}  G_{2t-2r}(x_2-x_3)
 \\
&\quad 
\cdot  G_{2t-2r}(x_3-x_4) \bigg(  \int_{\R}  G_{t-r}(x_4-y) dy\bigg) dx_4 dx_3  \bigg] dx_2 dx_1  dr \\
& 
\leq  \int_s^t 2R^2 \cdot 4(t-r)^3 dr =
2 R^2 (t-s)^4.
\end{aligned}
\label{Rosen7e}
\end{align}

\noi

It remains to show the inequality \eqref{Rosen7z}.
The case $p\in[2,4]$ follows from the inequalities in \eqref{Rosen7c}
by interpolation (i.e. an application of H\"older's inequality).
For $p\geq 4$ an integer, one can repeat the steps in \eqref{Rosen7e}
to arrive at
\noi
\begin{align*}
\begin{aligned}
& \int_s^t  \int_{\R}  \varphi^p_{t, R}(r,y)   \, drdy
\leq  \int_s^t 2R^2 \cdot  [2(t-r)]^{p-2} (t-r) dr
\leq   2^{p-1}  (t-s)^p R^2,
\end{aligned}
\end{align*}

\noi
and therefore,
the general case  follows  by interpolation.
This concludes   the proof.
\qedhere

\end{proof}

Finally, we end this section with a consequence
of Rosenthal's inequality; see Theorem 2.1, Theorem 2.3, and Corollary 2.5 in
\cite{BN16}.

\begin{proposition} \label{Prop:Rosen}

Recall the definition of $G_t$ from \eqref{FSol}.
Then, the following statements hold.

\smallskip
\noi
{\rm (i)}
Let $\{\Phi(s,y)\}_{ (s, y) \in \R_+ \times \R}$ be a predictable process such that

\noi
\begin{align}
\bE\int_0^t \int_{\R}G_{t-s}^2(x-y)|\Phi(s,y)|^2 dyds<\infty.
\label{Rosen1}
\end{align}

\noi
 Suppose \eqref{mp} holds  for some finite $p\geq 2$. Then,

\noi
\begin{align}
\begin{aligned}
&\bE \bigg[   \Big| \int_0^t \int_{\R}G_{t-s}(x-y)\Phi(s,y)L(ds,dy) \Big|^p \bigg]  \\
& \qquad\quad
\leq C_{p}(t) \int_0^t \int_{\R}G_{t-s}^p(x-y)\bE \big[  |\Phi(s,y)|^p \big] dsdy,
\end{aligned}
\label{Rosen2}
\end{align}
where $C_{p}(t)=2^{p-1}B_p^p\big( m_2^{\frac p2} t^{p-2} + m_p \big)$
with  $B_p$   the constant in Rosenthal's inequality.

\smallskip
\noi
{\rm (ii)} Suppose $m_p<\infty$ for some finite $p\geq 2$. Recall  $F_R(t)$ from \eqref{FRT}.
Then, for any finite $T > 0$,
there is some constant  $A_T$ only depending on $T$
such that

\noi
\begin{align}
\| F_R(t) - F_R(s) \|^p_p \leq A_{T} \cdot R^{\frac{p}{2}}  |t-s|^p
\label{Rosen2a}
\end{align}

\noi
for any $t,s\in [0,T]$ and  for any $R\geq 1$.
\noi
In particular, it holds for any $R\geq 1$ that

\noi
\begin{align}
\sup_{t\leq T}\| F_R(t) \|^p_p \leq  A_{T} \cdot R^{\frac{p}{2}}  T^p .
\label{Rosen2b}
\end{align}

\end{proposition}

\begin{proof}

Fix $t\in(0,\infty)$. We first prove the bound \eqref{Rosen2} in part (i).

 By Theorem 2.3 in  \cite{BN16}
and the condition \eqref{Rosen1},
the process $\{Y_r\}_{r\in[0,t]}$, given by
\begin{align*}
Y_r  & =  \int_0^r \int_{\R}G_{t-s}(x-y)\Phi(s,y)L(ds,dy)  \\
&= \int_0^r \int_{\R} \int_{\R_0} G_{t-s}(x-y)\Phi(s,y) z \wh{N}(ds,dy, dz), \,\, r\in[0,t],
\end{align*}

\noi
has a c\`adl\`ag (i.e. right continuous with left limits) modification,
which is  a
martingale with

\noi
\begin{align}
\begin{aligned}
\| Y_t\|^p_p &=   \Big\| \int_0^t \int_{\R}G_{t-s}(x-y)\Phi(s,y)L(ds,dy) \Big\|_p \\
&\leq B^p_p \bigg[
\Big\|  \int_0^t \int_{\R} \int_{\R_0} G^2_{t-s}(x-y)\Phi^2(s,y) |z|^2  dsdy\nu(dz)
\Big\|_{\frac{p}{2}}^{\frac12} \\
&\qquad\qquad\qquad
+ \bigg( \bE  \int_0^t \int_{\R} \int_{\R_0} G^p_{t-s}(x-y)|\Phi|^p(s,y) |z|^p ds dy\nu(dz) \bigg)^{\frac1p} \bigg]^p ,
\end{aligned}
\label{Rosen3}
\end{align}

\noi
where $B_p$ is the constant in the Rosenthal's inequality;
see Theorem 2.1 in \cite{BN16}.
Then, we deduce from \eqref{Rosen3},   $|a+b|^p \leq 2^{p-1} ( |a|^p + |b|^p)$,
and  Minkowski's inequality with  \eqref{m2} and \eqref{mp} that

 \noi
\begin{align}
\begin{aligned}
\| Y_t\|^p_p
&\leq 2^{p-1} B^p_p \bigg[  m_2^{\frac p 2}
 \bigg( \int_0^t \int_{\R}  G^2_{t-s}(x-y) \| \Phi(s,y)\|_p^2  dsdy  \bigg)^{\frac p2} \\
&\qquad\qquad\qquad
+   m_p \int_0^t \int_{\R}   G^p_{t-s}(x-y) \|  \Phi (s,y)\|_p^p   dsdy  \bigg].
\end{aligned}
\label{Rosen4}
\end{align}

 Note that $G_{t-s}(x-y) = 0$ for $|x-y| \geq t-s$ and
 \begin{align}
\int_0^t \int_\R \ind_{\{ | x-y| < t-s\}} dsdy = t^2.
 \label{Rosen5}
 \end{align}

 \noi
Thus, it follows from  Jensen's inequality with \eqref{Rosen5} that

\noi
\begin{align}
\begin{aligned}
& \bigg( \int_0^t \int_{\R}  G^2_{t-s}(x-y) \| \Phi(s,y)\|_p^2  dsdy  \bigg)^{\frac p2}  \\
&\quad \leq  ( t^2 )^{\frac{p}{2} - 1} \int_0^t \int_{\R}  G^p_{t-s}(x-y) \| \Phi(s,y)\|_p^p  dsdy.
\end{aligned}
 \label{Rosen6}
 \end{align}
 Hence, the desired inequality \eqref{Rosen2} in part (i) follows
 from \eqref{Rosen4} and \eqref{Rosen6}.

\medskip

Now we prove the difference estimate \eqref{Rosen2a} in part (ii).
Without losing any generality, we assume $0 \leq s < t \leq T$.
By Lemma \ref{lem:Fubi},
 we can rewrite $F_R(t)$ as
\begin{align*}
F_R(t) = \int_0^t \int_\R \int_{\R_0}    \varphi_{t, R}(r,y) u(r,y) z \wh{N}(dr, dy, dz)
\end{align*}

\noi
with $\varphi_{t,R}$ as in \eqref{Rosen6b}.
Note that we can write
\begin{align}
\begin{aligned}
F_R(t) - F_R(s)
& = \int_0^s \int_{\R\times\R_0}  \big[ \varphi_{t, R}(r,y)  - \varphi_{s, R}(r,y) \big]
u(r,y) z \wh{N}(dr, dy, dz) \\
&\qquad
+ \int_s^t \int_{\R\times\R_0}  \varphi_{t, R}(r,y)  u(r,y) z \wh{N}(dr, dy, dz)\\
&:= \mathbf{T}_1 +  \mathbf{T}_2.
\end{aligned}
\label{Rosen6ba}
\end{align}

\noi
As in \eqref{Rosen4},
we can deduce from Rosenthal's inequality
(Theorem 2.3 in \cite{BN16}), Minkowski inequality, \eqref{KPT},  the fact that $\varphi_{t,R}-\varphi_{s,R} \in [0,t-s]$, and \eqref{Rosen6c},
that

\noi
\begin{align}
\begin{aligned}
\| \mathbf{T}_1 \|_p^p
&\leq  2^{p-1} B_p^p \bigg[  m_2^{\frac p2 }  \bigg( \int_0^s \int_{\R}
\big| \varphi_{t, R}(r,y)  - \varphi_{s, R}(r,y) \big|^2 \| u(r,y) \|_p^2 \, drdy  \bigg)^{\frac{p}{2}} \\
&\qquad\qquad
+ m_p \int_0^s \int_{\R}
\big| \varphi_{t, R}(r,y)  - \varphi_{s, R}(r,y) \big|^p \| u(r,y) \|_p^p \, drdy \bigg] \\
&\leq  2^{p-1} B_p^p \big[ m_2^{\frac p2 }   \cdot (2t)^{\frac{p}{2}} K_p^p(t) (t-s)^p  R^{\frac{p}{2}}
+ m_p \cdot 2t K_p^p(t) (t-s)^p  R  \big] \\
&\les K_p^p(t)(t+t^{\frac{p}{2}}) (t-s)^p R^{\frac{p}{2}} \,\,\, \text{for $R\geq 1$}
\end{aligned}
\label{Rosen7a} 
\end{align}

\noi
and

\noi
\begin{align}
\| \mathbf{T}_2 \|_p^p
&\leq  2^{p-1}B_p^p \bigg[  m_2^{\frac p2 } \bigg( \int_s^t \int_{\R}
   \varphi^2_{t, R}(r,y)  \| u(r,y) \|_p^2 \, drdy  \bigg)^{\frac p2} 
+ m_p \int_s^t \int_{\R}
 \varphi^p_{t, R}(r,y)   \| u(r,y) \|_p^p\, drdy \bigg]  \notag \\
&\les K_p^p(t) \bigg( \int_s^t \int_{\R}
   \varphi^2_{t, R}(r,y)   \, drdy  \bigg)^{\frac{p}{2}}
  + K_p^p(t) \int_s^t \int_{\R}
 \varphi^p_{t, R}(r,y)   \, drdy. \label{Rosen7b} 
\end{align}

\noi
Therefore,  we can deduce from \eqref{Rosen6ba}, \eqref{Rosen7a},
and \eqref{Rosen7b} with    \eqref{Rosen7z}
that
\[
\big\| F_R(t) - F_R(s)\|_p^p
\les K_p^p(t)[ 1 + t + t^{\frac p2} ] R^{\frac p2} | t-s|^p
\]
for $R\geq 1$. This proves the bound \eqref{Rosen2a},
and thus the uniform bound \eqref{Rosen2b} by noting that
 $F_R(0)=0$.

Hence, the proof of Proposition \ref{Prop:Rosen} is completed.
\qedhere

\end{proof}

\section{Malliavin derivatives of the hyperbolic Anderson model} \label{SEC3}

In this section, we will establish $L^p(\Om)$-bounds for Malliavin derivatives
of hyperbolic Anderson model \eqref{wave}.
As an intermediate step,
we will first study the stochastic wave equation with delta initial velocity in
Subsection \ref{SUB31}.

\subsection{Stochastic wave equation with delta initial velocity} \label{SUB31}

In this subsection,   we study  the following stochastic wave  equation:

\noi
\begin{align}
\begin{cases}
\dt^2 v (t,x)
=\dx^2 v(t,x)+ v(t,x)\dot{L}(t,x), \quad t> r, \ x \in \R \\[1em]
v(r,\cdot) = 0, \quad
\dt v(r,\cdot) = z  \dl_{y} ,
\end{cases}
\label{wave_dl}
\end{align}

\noi
where $(r, y, z)\in\R_+\times\R\times\R_0$ is fixed
and $\dot{L}$ is the space-time L\'evy   noise as in \eqref{wave}.

 We say  that a predictable  process $v=v^{(r,y,z)}$ is
 a solution to the equation \eqref{wave_dl}
 provided that:
\begin{itemize}
\item[(i)]

   $v(r,x)=0$ for any $x \in \R$,

\item[(ii)]

for  any $t> r$ and $x \in \R$, the following  equation holds almost surely:

\noi
\begin{align}
v(t,x)=G_{t-r}(x-y)z + \int_r^t \int_{\R}G_{t-s}(x-y') v(s,y')L(ds,dy'),
\label{wave_dl2}
\end{align}

\noi
where the stochastic integral in \eqref{wave_dl2} is interpreted  in It\^o sense and
coincides with the Kabanov-Skorohod integral $\dl( H)$ with
$H(s, y', z') =G_{t-s}(x-y') v(s,y') z'$.

\end{itemize}

As we will see shortly,
the solution $v^{(r,y,z)}$ is related to the Malliavin derivative $D_{r,y,z}u(t,x)$,
via relation \eqref{dec1}.

\begin{proposition} \label{prop:dl}
Fix $(r, y,z)\in\R_+\times\R\times\R_0$
and suppose $m_2<\infty$ as in \eqref{m2}. Then the following statements hold.

\smallskip
\noi
{\rm (i)} The equation \eqref{wave_dl} has a unique solution $v=v^{(r,y,z)}$.
Moreover, if $m_p<\infty$ for some $p\geq 2$ as in \eqref{mp},
we have for any $T>0$ that

\noi
\begin{align}
 \sup_{ r \leq t \leq T}\, \sup_{x,y \in \R}
\| v^{(r,y,z)}(t,x)\|_p \leq C_{T, p, \nu}|z|,
\label{wave_dl3}
\end{align}

\noi
where  
$C_{T,p, \nu}$ is a constant given in \eqref{def:CTPNU}.

\smallskip
\noi
{\rm (ii)} Let   $t>r$ and $x  \in \R$. Then,  $v^{(r,y,z)}(t,x)$ admits the following
chaos expansion  in $L^2(\Om)$:

\noi
\begin{align}
v^{(r,y,z)}(t,x)= G_{t-r}(x-y)z
+ \sum_{n \geq 1} I_n\big( G_{t,x,n+1} (r,y,z; \bul ) \big),
\label{wave_dl4}
\end{align}

\noi
 where\footnote{That is,
 $G_{t,x,k+1}(r, y, z; \bul)
 = F_{t,x,k+1}(\pmb{t_{k+1}},\pmb{x_{k+1}}, \pmb{z_{k+1}}) |_{(t_1, x_1, z_1) = (r,y,z)} $  with $F_{t,x,n}$ given by \eqref{KER:F}.
  In particular, $G_{t,x,1}(r,y,z) = F_{t,x,1}(r,y,z) = G_{t-r}(x-y)z$.
 \label{ft12b}}

 \noi
\begin{align}
\begin{aligned}
&G_{t,x,n+1}(r, y, z; \pmb{t_n},\pmb{x_n}, \pmb{z_n})  \\
& = G_{t-t_n}(x-x_n) G_{t_n-t_{n-1}}(x_n-x_{n-1})\cdots G_{t_2-t_1}(x_2-x_1)G_{t_1-r}(x_1-y) z \prod_{j=1}^n z_j .
\end{aligned}
\label{KER:F2}
\end{align}

\smallskip
\noi
{\rm (iii)}
For any $t > r$ and  $x\in \R$, we have

\noi
\begin{align}
G_{t-r}(x-y) v^{(r,y,z)}(t,x)=\tfrac{1}{2}v^{(r,y,z)}(t,x).
\label{wave_dl5}
\end{align}

\end{proposition}

 Note that the equality \eqref{wave_dl5} holds only in the 
one-dimensional setting, where the fundamental wave solution takes the 
specific form \eqref{FSol}.

\begin{proof}[Proof of Proposition \ref{prop:dl}] (i)  Throughout this proof, we fix $T>0$ and omit the
fixed superscripts
$r, y, z$.

Consider the sequence $\{v_n\}_{n\geq 0}$ of Picard iterations defined as follows:
\begin{itemize}
\item we set $v_n(r,x)=0$ for any $x \in \R$ and $n\in\N_{\geq 0}$;
\item  for $t>r$,    we let $v_0(t,x)=G_{t-r}(x-y)z$  and

\noi
\begin{align}
v_{n+1}(t,x) = G_{t-r}(x-y) z + \int_r^t \int_{\R}G_{t-s}(x-y')v_n(s,y') L(ds,dy')
\label{wave_dl6}
\end{align}
for any $n\in\N_{\geq 0}$.
\end{itemize}

Defining $v_{-1}(t,x)=0$, we see that
\[
v_{n+1}(t,x)-v_n(t,x)=\int_r^t \int_{\R}G_{t-s}(x-y') \big[ v_n(s,y')-v_{n-1}(s,y')\big] L(ds,dy').
\]
 for any $n\in\N_{\geq  0}$,  $t\geq r$,  and $x \in \R$.
 Then, we can deduce from  Proposition \ref{Prop:Rosen} with \eqref{FSol}
 and \eqref{factsG} that

 \noi
\begin{align}
\begin{aligned}
&\bE\big[  |v_{n+1}(t,x)-v_n(t,x)|^p \big]  \\
&\quad
\leq C_p(t) \int_r^t \int_{\R}G_{t-s}^p(x-y') \bE \big[ |v_n(s,y')-v_{n-1}(s,y')|^p \big] ds dy'  \\
&\quad
\leq C_p(t) 2^{1-p}t \int_r^t  \bigg(\sup_{y' \in \R} \bE\big[ |v_n(s,y')-v_{n-1}(s,y')|^p\big] \bigg) ds ,
\end{aligned}
\label{wave_dl7}
\end{align}

\noi
where  $C_{p}(t)=2^{p-1}B_p^p\big( m_2^{\frac p2} t^{p-2} + m_p \big)$
with  $B_p$   the constant in Rosenthal's inequality.
Letting $H_n(t):= \sup\big\{ \bE\big[ |v_n(t,x)-v_{n-1}(t,x)|^p\big]: x\in\R\big\}$, 
we obtain from \eqref{wave_dl7} that
\begin{align}
H_{n+1}(t)\leq C_p(T) T 2^{1-p} \int_r^t H_n(s)  ds  \,\,\, \mbox{for all $t \in [r,T]$.}
\label{wave_dl8}
\end{align}
Note that
\begin{align}
M:=\sup_{t \in [r,T]}H_0(t)=\sup_{t \in [r,T]}\sup_{x\in \R}G_{t-r}^p(x-y) |z|^p = 2^{-p}|z|^p.
\label{wave_dl9}
\end{align}
Therefore, iterating \eqref{wave_dl8} with \eqref{wave_dl9} yields

\noi
\begin{align}
H_{n+1}(t) 
\leq \frac{ \big( C_p(T) T 2^{1-p}  \big)^{n+1} t^{n+1}  }{(n+1)!} M
\leq \frac{ \big( C_p(T) T^2 2^{1-p}  \big)^{n+1}   }{(n+1)!} 2^{-p} |z|^p
\,\,\,\text{for $t\in[r, T]$},
\notag
\end{align}

\noi
and thus, 
we get with $C_{p}(t)=2^{p-1}B_p^p\big( m_2^{\frac p2} t^{p-2} + m_p \big)$,
\begin{align}
\sum_{n\geq 0} \sup_{(t,x)\in[r, T]\times\R}  \|  v_n(t,x)-v_{n-1}(t,x) \|_p 
\leq C_{T,p, \nu} |z|,
\notag
\end{align}

\noi
where $C_{T,p, \nu} $ is a constant defined by 

\noi
\begin{align}\label{def:CTPNU}
 C_{T,p, \nu}  : = 2^{-p} \exp\big[  B_p^p( m_2^{\frac{p}{2}} T^p + m_p T^2 )   \big]
\end{align}
with  $B_p$   the constant in Rosenthal's inequality.
This proves that $\{v_n(t,x)\}_{n\geq 1}$ is  Cauchy   in $L^p(\Om)$,
uniformly in $(t,x) \in [r,T] \times \R$.
Its limit $v$ is the unique solution to \eqref{wave_dl} with

\noi
\begin{align}
\sup_{(t,x)\in [r,T]\times \R}\|v(t,x)\|_p \leq C_{T,p, \nu}      |z|.
\label{wave_dl9b}
\end{align}
The case $p=2$ is exactly the first part of (i). And for the other part with  $p\geq 2$,
the uniform bound \eqref{wave_dl3} is exactly \eqref{wave_dl9b},
since the bound in \eqref{wave_dl9b} does not depend on $r$ or  $y$ .

\medskip

(ii)  From part (i), we know that $v(t,x)$ is the $L^2(\Om)$-limit of $v_{n+1}(t,x)$
as $n\to\infty$. We will show that $v_{n+1}(t,x)$ lives in finitely many chaoses
with some explicit expression for each $n$, and then
the chaos expansion \eqref{wave_dl4} for $v(t,x)$ follows by
sending $n$ to infinity.

Recall $v_0(t,x) = G_{t-r}(x-y) z$ and

\noi
\begin{align}
v_{n+1}(t,x) = G_{t-r}(x-y) z + \dl(V_{t,x,n}),
\label{wave_dl6b}
\end{align}

\noi
where
\begin{align}
V_{t,x,n}(s, y', z') := \ind_{(r,t)}(s)G_{t-s}(x-y') v_n(s, y')z' .
\label{def_Vn}
\end{align}

\noi
In what follows, we first show that for each $n\in\mathbb{Z}_{\geq -1}$,
$v_{n+1}(t,x)$ admits the following chaos expansion
\begin{align}
v_{n+1}(t,x) = G_{t-r}(x-y)z + \sum_{k=1}^{n+1} I_k( G_{t,x,k+1}(r, y, z; \bul ) ),
\label{claimV}
\end{align}

\noi
where
$G_{t,x,k+1}(r, y, z; \bul)$ is as in \eqref{KER:F2}.
To prove \eqref{claimV}, we proceed with mathematical induction.
The base case where $n=-1$ is trivial.
And for the case where $n=0$, we deduce from \eqref{wave_dl6} and the base case
 that
\begin{align*}
v_1(t,x) &=  G_{t-r}(x-y)z + \int_r^t \int_\R \int_{\R_0} G_{t-s}(x-y')   z G_{s-r}(y'-y) z'
\wh{N}(ds, dy', dz') \\
& =  G_{t-r}(x-y)z + I_1(G_{t, x, 2}(r, y, z; \bul)).
\end{align*}
That is, the claim \eqref{claimV} also holds for $n=0$.
Now assume \eqref{claimV}
holds for $n = m$ with $m\geq 0$.
Then, we write by using \eqref{wave_dl6b} with \eqref{def_Vn},
and the induction hypothesis
that

\noi
\begin{align}
\begin{aligned}
v_{m+2}(t,x)
&= G_{t-r}(x-y)z + \dl\big( V_{t,x,m+1} \big)
\end{aligned}
\notag
\end{align}
with

\noi
\begin{align}
\begin{aligned}
V_{t,x,m+1}(s, y', z')
&=   \ind_{(r,t)}(s)G_{t-s}(x-y') z' v_{m+1}(s, y') \\
&= \ind_{(r,t)}(s)G_{t-s}(x-y') z'
\Big[ G_{s-r}(y'-y)z + \sum_{k=1}^{m+1} I_k\big( G_{s,y' ,k+1}(r, y, z; \bul ) \big) \Big] \\
&= \ind_{\{r < s < t \}} G_{t,x, 2}(r, y, z; s, y', z') \\
&\qquad + \sum_{k=1}^{m+1}
I_k\big(      \ind_{(r,t)}(s)G_{t-s}(x-y') z'  \wt{G}_{s,y' ,k+1}(r, y, z; \bul ) \big),
\end{aligned}
\label{wave_dl10}
\end{align}

\noi
where $\wt{G}_{s,y' ,k+1}(r, y, z; \bul ) $ denotes the symmetrization of
the function
$G_{s,y' ,k+1}(r, y, z; \bul )$.
Note that the kernel of the $k$-th multiple integral in \eqref{wave_dl10}
can be rewritten as follows:

\noi
\begin{align}
\begin{aligned}
 &  \ind_{(r,t)}(s)G_{t-s}(x-y') z'  \wt{G}_{s,y' ,k+1}(r, y, z;  \pmb{t_k},\pmb{y_k}, \pmb{z_k}) \\
 &\quad
 = G_{t-s}(x-y') z' \frac{1}{k!} \sum_{\s\in\mathfrak{S}_k}
 G_{s- t_{\s(k)}}(y' - y_{\s(k)} ) z_{\s(k)}  \\
 &\qquad\quad\cdot G_{t_{\s(k)} -t_{\s(k-1)}  }( y_{\s(k)}-y_{\s(k-1)} ) z_{\s(k-1)} \cdots   G_{t_{\s(1)} - r }( y_{\s(1)} - y ) z \\
 &= \frac{1}{k!} \sum_{\pi\in\mathfrak{S}_{k}}
       G_{t- t_{\pi(k+1)}}(x - y_{\pi(k+1)} ) z_{\pi(k+1)}  \\
 &\qquad\quad\cdot G_{t_{\pi(k+1)} -t_{\pi(k)}  }( y_{\pi(k+1)}-y_{\pi(k)} ) z_{\pi(k)} \cdots   G_{t_{\pi(1)} - r }( y_{\pi(1)} - y ) z
\end{aligned}
\label{wave_dl11}
\end{align}

\noi
with  $(t_{\pi(k+1)}, y_{\pi(k+1)}, z_{\pi(k+1)} ) = (s, y', z')$
and the convention \eqref{convention}, where we point out that
the second sum in \eqref{wave_dl11} can be viewed
as a sum running over all permutations $\pi\in\mathfrak{S}_{k+1}$
such that $ t_{\pi(k+1)}=s$ is the biggest time  among all
$\{ t_{\pi(j)}: j=1, ..., k+1\}$. Therefore, the symmetrization
of  the function  \eqref{wave_dl11}

\noi
\begin{align*}
&(s, y', z', \pmb{t_k},\pmb{y_k}, \pmb{z_k} )
\equiv  ( \pmb{t_{k+1} },  \pmb{y_{k+1}}, \pmb{z_{k+1}} )  \\
&\qquad\qquad
\longmapsto
 \ind_{(r,t)}(s)G_{t-s}(x-y') z'  \wt{G}_{s,y' ,k+1}(r, y, z;  \pmb{t_k},\pmb{y_k}, \pmb{z_k})
\end{align*}
coincides with $\wt{G}_{t,x, k+2}(r,y,z; \pmb{t_{k+1} },  \pmb{y_{k+1}}, \pmb{z_{k+1}}   )$.
As a consequence, we deduce from \eqref{wave_dl10} and
Lemma \ref{lem:dl} that
$V_{t, x, m+1}\in\dom(\dl)$
with
\[
 \dl\big( V_{t,x,m+1} \big) = \sum_{k=1}^{m+2} I_k( \wt{G}_{t,x, k+1}(r,y,z; \bul  ) \big)
 =\sum_{k=1}^{m+2} I_k( G_{t,x, k+1}(r,y,z; \bul  ) \big).
\]
Hence, we just proved that the claim \eqref{claimV} holds for
$n=m+1$, and thus for all $n$.
Then, the proof of part (ii) is concluded by
sending $n$ to infinity.

 \medskip

(iii)  Finally, we prove the equality \eqref{wave_dl5}.
Recall from \eqref{FSol} that $G_{t-r}(x-y) = \frac{1}{2} \ind_{\{ | x-y| < t-r \}}$.
Then, it suffices to show that

\noi
\begin{align}
v^{(r,y,z)}(t,x) = 0 \,\,\, \text{when $|x-y| \geq t-r$.}
\label{suff1}
\end{align}
Indeed, by triangle inequality and \eqref{KER:F2}, we know that
$G_{t,x, n+1}(r, y,z; \bul) = 0$ when  $|x-y| \geq t-r$,
which, together with  the chaos expansion in \eqref{wave_dl4}, implies \eqref{suff1}. 
Hence,  the proof of Proposition \ref{prop:dl} is completed.
\qedhere

\end{proof}

\subsection{Estimates of Malliavin derivatives}  \label{SUB32}

In this subsection, our goal is to derive the $L^p(\Om)$-bound
for the Malliavin derivatives of the solution to the hyperbolic
Anderson model \eqref{wave}.

From the chaos expansion \eqref{u1} with \eqref{KER:F},
we deduce that

\noi
\begin{align}
\label{chaos-D}
D_{r,y,z}u(t,x) = \sum_{n\geq 1}n I_{n-1}\big(\wt{F}_{t,x,n}(r, y, z, \bul)\big),
\end{align}
where   $\wt{F}_{t,x,n}(r,y,z, \pmb{t_{n-1}},\pmb{x_{n-1}},\pmb{z_{n-1}})$
is obtained by first symmetrizing the kernel $F_{t,x,n}$ and then
putting $(r,y,z)$ in any of the $n$ `arguments'.\footnote{Here we view $(r,y,z)\in Z$
as one argument.}
It is not difficult to see that

\noi
\begin{align}
\wt{F}_{t,x,n}(r,y,z, \bul )
=\frac{1}{n} \sum_{j=1}^{n} H_{t,x,n}^{(j)}(r,y,z; \bul ),
\notag 
\end{align}

\noi
where $H_{t,x,n}^{(j)}(r,y,z; \bul)$ is the symmetrization of the function
$F_{t,x,n}^{(j)}(r,y,z; \bul)$ given by

\noi
\begin{align}
\begin{aligned}
 F_{t,x,n}^{(j)}(r,y,z; \pmb{t_{n-1}, x_{n-1}},\pmb{z_{n-1}})
 &=  G_{t-t_{n-1}}(x-x_{n-1})z_{n-1}\ldots
G_{t_j-r}(x_j-y)z  \\
&\qquad \cdot G_{r-t_{j-1}}(y-x_{j-1})z_{j-1}\ldots G_{t_2-t_1}(x_2-x_1)z_1;
\end{aligned}
\notag 
\end{align}
that is, $F_{t,x,n}^{(j)}(r,y,z; \bul)$ is obtained from $F_{t,x,n}$ by
putting $(r,y,z)$ at the $j$-th argument.
And it follows immediately that

\noi
\begin{align}
F_{t,x,n}^{(j)}(r,y,z; \bul )=  F_{r,y,j-1} \otimes G_{t,x, n-j+1}(r,y,z; \bul )
\label{decomp1}
\end{align}
with $G_{t,x,n-j+1}(r,y,z; \bul )$ as in \eqref{KER:F2};
 see Footnote \ref{ft12b} and also \cite[page 784]{BNQSZ}.

With the above notations, we can write
\begin{align}
\begin{aligned}
D_{r,y, z} u(t,x)
& = \sum_{n\geq 1}  \sum_{j=1}^{n} I_{n-1}\big( F_{t,x,n}^{(j)}(r,y,z;\bullet)\big) \\
&= \sum_{n\geq 1}  \sum_{j=1}^{n}
I_{n-1}\big(   F_{r,y,j-1} \otimes G_{t,x,n-j+1}(r,y,z; \bul )        \big).
\end{aligned}
\label{D-chaos}
\end{align}

Similarly, we can obtain the following chaos expansion
for the second Malliavin derivative:
for $r_1 < r_2 \leq t$,

\noi
\begin{align}
\begin{aligned}
&D^2_{\pmb{r_2}, \pmb{y_2}, \pmb{z_2}} u(t, x)
\equiv D_{r_1, y_1, z_1} D_{r_2, y_2, z_2} u(t, x) \\
&\quad =  \sum_{n= 2}^\infty \sum_{1\leq i < j \leq n}
I_{n-2} \big(  F_{r_1, y_1, i-1} \otimes G_{r_2, y_2, j- i}(r_1, y_1,z_1; \bul)
\otimes G_{t,x, n-j+1}(r_2, y_2,z_2; \bul)   \big);
\end{aligned}
\label{D2-chaos}
\end{align}

\noi
while for $r_2 < r_1 \leq t$, we can get a similar equality  by noting that
$D^2_{\pmb{r_2}, \pmb{y_2}, \pmb{z_2}} u(t, x) $ is almost surely
symmetric in those two arguments $(r_1, y_1, z_1)$ and $(r_2, y_2,z_2)$.

\medskip

Now we are ready to state the main result in this subsection.

\begin{proposition} \label{prop:dec}
Suppose $m_2<\infty$ as in \eqref{m2}.
 Then, $u(t,x) \in \dom(D^2)$ \textup{(}i.e. twice Malliavin differentiable\textup{)}
 and
  the following statements hold.

\smallskip
\noi
{\rm (i)} Fix $(r, y,z)\in (0,t]\times\R\times\R_0$ and recall
the notation $v^{(r,y,z)}$ from Proposition \ref{prop:dl}.
Then,

\noi
\begin{align}
D_{r,y,z}u(t,x)=u(r,y)v^{(r,y,z)}(t,x) \,\,\, \text{almost surely.}
\label{dec1}
\end{align}

\smallskip
\noi
{\rm (ii)} Fix $(r_1,y_1,z_1), (r_2, y_2, z_2)\in \R_+\times\R\times\R_0$
with $r_1 < r_2\leq t$. Then,

\noi
\begin{align}
D_{r_1,y_1,z_1} D_{r_2, y_2, z_2} u(t,x)
=u(r_1, y_1) v^{(r_1,y_1,z_1)}(r_2, y_2) v^{(r_2, y_2,z_2)}(t,x).
\label{dec2}
\end{align}

\smallskip
\noi{\rm (iii)} Let $m_p<\infty$ for some finite $p\geq 2$ as in \eqref{mp}
and let $T \in(0, \infty)$.
Then, for any $0< r < t  \leq T$ and for any $(y,z)\in \R\times\R_0$,
we have

\noi
\begin{align}
\|D_{r,y,z}u(t,x)\|_p \leq C'_{T,p, \nu} G_{t-r}(x-y)|z|,
\label{D1est}
\end{align}

\noi
where $C'_{T,p, \nu}$ is given by  \eqref{CTPA}.
For any $(r_1,y_1,z_1)$, $(r_2, y_2, z_2)\in \R_+\times\R\times\R_0$,
we have

\noi
\begin{align}
\begin{aligned}
&\|D_{r_1,y_1,z_1 } D_{r_2, y_2,z_2} u(t,x) \|_p  \\
&\quad \leq C''_{T,p, \nu} |z_1z_2| \times
 \begin{cases}
G_{t-r_2}(x-y_2)G_{r_2-r_1}(y_2-y_1) & \mbox{if $r_1<r_2$} \\
G_{t-r_1}(x-y_1)G_{r_1-r_2}(y_1-y_2) & \mbox{if $r_2< r_1$,}
\end{cases}
\end{aligned}
\label{D2est}
\end{align}

\noi
where  $C''_{T,p, \nu}$  is given by  \eqref{CTPB}.

\end{proposition}

\begin{remark}\label{rem:dec}\rm

(a)
Note that in part (iii), assumption \eqref{mp} is used to
guarantee the uniform $L^p(\Om)$-bounds of $v^{(r, y,z)}$,
which are further applied in steps
\eqref{CKTP} and \eqref{CKTP2}.
Indeed,  this assumption is reflected 
in the expression of the bounds \eqref{D1est} and \eqref{D2est}
via the constants $C'_{T,p,\nu}$ and $C''_{T,p,\nu}$;
the dependency on the jump intensity $\nu$ arises 
through the constants $m_2$ and $m_p$.

\smallskip
\noi
(b) The upper bounds in \eqref{D1est}-\eqref{D2est} are optimal in the sense
that we can get matched lower bound. More precisely, using the orthogonality relation
\eqref{int3c} and \eqref{chaos-D},
we can get
\begin{align*}
\| D_{r,y,z} u(t,x) \|_2 
&= \bigg(   | \wt{F}_{t,x,1}(r, y, z, \bul) |^2
+ \sum_{n\geq 2} n^2 (n-1)!  \wt{F}_{t,x,n}(r, y, z, \bul) \|_{\fH^{\otimes (n-1)}}^2 \bigg)^{\frac12} \\
&\geq \wt{F}_{t,x,1}(r, y, z, \bul) = G_{t-r}(x-y) z;
\end{align*}
and similarly, we can get (with the convention \eqref{convention} in mind)
 \[
 \| D^2_{\pmb{r_2,y_2,z_2}} u(t,x) \|_2 \geq \big[G_{t-t_1}(x-y_1) G_{t_1-t_2}(y_1-y_2)
 + G_{t-t_2}(x-y_2) G_{t_2-t_1}(y_2-y_1)   \big] \cdot |z_1z_2|.
 \]

\end{remark}

\begin{proof}[Proof of Proposition \ref{prop:dec}]

We first prove the decomposition \eqref{dec1} in part (i).
Recall the chaos expansion \eqref{D-chaos}. Note that
the kernels $F_{r, y, j-1}$ and $G_{t,x, n-j+1}(r, y, z; \bul)$ in
\eqref{D-chaos} and \eqref{decomp1} have disjoint temporal supports,
which implies immediately that

\noi
\begin{align}
\begin{aligned}
\wt{F}_{r, y, j-1} \star^0_k \wt{G}_{t,x, n-j+1}(r, y, z; \bul) & =0 \\
\wt{F}_{r, y, j-1} \star^1_k \wt{G}_{t,x, n-j+1}(r, y, z; \bul) & =0
\end{aligned}
\label{dec3}
\end{align}
for $1\leq k  \leq (j-1) \wedge (n-j)$, where $\wt{G}_{t,x, n-j+1}(r, y, z; \bul)$
denotes the symmetrization of $G_{t,x, n-j+1}(r, y, z; \bul)$ given by \eqref{KER:F2}.
Thus, we can deduce from  \eqref{D-chaos},
Proposition \ref{prop:prod2} with \eqref{dec3}, \eqref{u1}, and
\eqref{wave_dl4} in Proposition \ref{prop:dl}  that

\noi
\begin{align}
\begin{aligned}
D_{r,y, z} u(t,x)
&= \sum_{n\geq 1}  \sum_{j=1}^{n}
I_{j-1}(   F_{r,y,j-1} )  I_{n-j}\big( G_{t,x,n-j+1}(r,y,z; \bul )        \big) \\
&= \bigg( \sum_{j=1}^\infty I_{j-1}(   F_{r,y,j-1} )  \bigg)
\sum_{n\geq 0} I_{n}\big( G_{t,x,n+1}(r,y,z; \bul )        \big) \\
&= u(r, y) \cdot  v^{(r, y,z)}(t,x).
\end{aligned}
\notag
\end{align}
That is, the decomposition \eqref{dec1} holds. 
 Moreover,
due to the disjoint temporal supports of $F_{r, y, j}$ and $G_{t,x, n}(r, y, z; \bul)$,
we obtain that the random variables $u(r,y)$ and $v^{(r, y,z)}(t,x)$
are independent,\footnote{For each $j, n\in\mathbb{N}_{\geq 1}$, 
the multiple integrals $I_j(F_{r, y, j})$ and $I_n(G_{t,x, n+1} )$ can be approximated
in $L^2(\Omega)$ by $I_j(F^{(k)})$ and $I_n(G^{(k)})$ as $k\to\infty$,
where $F^{(k)}\in \cE_j$ and  $G^{(k)}\in\cE_n$ as in \eqref{int3a}-\eqref{int3b}.
As in \eqref{int3b}, $I_j(F^{(k)}) \in\R[ Y_i; i\in I ]$ and $I_n(G^{(k)})\in\mathbb{R}[ Y_j; j\in J]$ are multilinear polynomials
in centered, independent Poisson random variables $\{ Y_\ell: \ell\in I \cup J \}$, 
where due to disjoint temporal support of $F_{r, y, j}$ and $G_{t,x, n+1}(r, y, z; \bul)$,
the two families  $\{ Y_i : i\in I \}$ and $\{ Y_j:j  \in J \}$ of  centered Poisson random variables  
are independent; see Definition \ref{def:PRM}.
This implies the independence of $I_j(F^{(k)})$ and $I_n(G^{(k)})$
and thus the    independence of $I_j(F_{r, y, j})$ and $I_n(G_{t,x, n+1} )$ by passing
$k\to\infty$. Therefore, the desired independence of $u(r, y)$ and $v^{(r, y,z)}(t,x)$
follows immediately. }
 and thus, together with
the uniform bound \eqref{wave_dl3} and the equality \eqref{wave_dl5} in Proposition \ref{prop:dl},
we can further get  

\noi
\begin{align}
\begin{aligned}
\|D_{r,y, z} u(t,x) \|_p  &= \| u(r,y) \|_p  \|v^{(r, y,z)}(t,x)\|_p \\
&\leq    2  K_p(T) C_{T,p, \nu} G_{t-r}(x-y) |z| ,
\end{aligned}
\label{CKTP}
\end{align}

\noi
where $C_{T,p, \nu}$ and $K_p(T)$ are as in \eqref{def:CTPNU} and
\eqref{KPT} respectively. This proves the bound \eqref{D1est}
in part (iii) with 
\begin{align} \label{CTPA}
C'_{T,p,\nu} = 2  K_p(T) C_{T,p, \nu}.
\end{align}

\medskip

Next, we prove \eqref{dec2} in part (ii).  Similarly,
we can rewrite the chaos expansion \eqref{D2-chaos} with $r_1 < r_2$
as follows:

\noi
\begin{align}
\begin{aligned}
&D_{r_1, y_1, z_1} D_{r_2, y_2, z_2} u(t, x) \\
& =  \sum_{n= 2}^\infty \sum_{1\leq i < j \leq n}
I_{i-1} \big(  F_{r_1, y_1, i-1}   \big)
I_{j-i-1} \big(    G_{r_2, y_2, j- i}(r_1, y_1,z_1; \bul) \big)  \\
&\qquad\qquad\cdot
I_{n-j} \big(   G_{t,x, n-j+1}(r_2, y_2,z_2; \bul)   \big) \\
&= \bigg( \sum_{i\geq 1} I_{i-1} \big(  F_{r_1, y_1, i-1}   \big)   \bigg)
 \bigg( \sum_{j\geq 0} I_{j} \big(    G_{r_2, y_2, j+1}(r_1, y_1,z_1; \bul)\big)   \bigg) \\
 &\qquad\qquad \cdot \sum_{n\geq 0}I_{n} \big(   G_{t,x, n+1}(r_2, y_2,z_2; \bul)   \big) \\
 &= u(r_1, y_1) v^{(r_1, y_1, z_1)}(r_2, y_2) v^{(r_2, y_2, z_2)}(t, x),
\end{aligned}
\notag
\end{align}
which is exactly the decomposition \eqref{dec2} in part (ii).
And it is also clear that the random variables
$u(r_1, y_1)$,  $v^{(r_1, y_1, z_1)}(r_2, y_2)$,
and $v^{(r_2, y_2, z_2)}(t, x)$
are independent.
Therefore, we deduce from \eqref{KPT}, \eqref{wave_dl3}, 
and \eqref{wave_dl5} in Proposition \ref{prop:dl}
that

\noi
\begin{align}
\begin{aligned}
&\|D_{r_1, y_1, z_1} D_{r_2, y_2, z_2} u(t, x) \|_p \\
&\quad
=\|  u(r_1, y_1)\|_p  \| v^{(r_1, y_1, z_1)}(r_2, y_2) \|_p  \| v^{(r_2, y_2, z_2)}(t, x)\|_p \\
&\quad
\leq   4  K_p(T) C^2_{T,p, \nu} G_{r_2- r_1}(y_2-y_1) G_{t-r_2}(x-y_2)   |z_1 z_2| \cdot
\end{aligned}
\label{CKTP2}
\end{align}
This proves \eqref{D2est} with 
\begin{align}
C''_{T,p, \nu}   =   4   K_p(T) C^2_{T,p, \nu}
\label{CTPB}
\end{align}
when $r_1 < r_2$.
Note that when $r_1 > r_2$, the proof is identical and thus omitted.

Hence the proof of Proposition \ref{prop:dec} is completed.
\qedhere

\end{proof}

\section{Proof of main results}  \label{SEC4}

\subsection{Spatial ergodicity}

We first establish the following strict stationarity.

\begin{lemma} \label{lem:stat}
Let  $t>0$ and $(x_1,\ldots,x_k,y) \in \R^{k+1}$. Then,
\begin{align}
\big( u(t,x_1),\ldots,u(t,x_k) \big) \eqd \big( u(t,x_1+y),\ldots,u(t,x_k+y) \big).
\label{stat1}
\end{align}

\end{lemma}

\begin{proof}  To show \eqref{stat1}, it suffices to prove
\begin{align}
 \sum_{i=1}^k c_i u(t,x_i) \eqd   \sum_{i=1}^k c_i u(t, x_i +y)
\label{stat2}
\end{align}
for any $(c_1, ... , c_k)\in\R^k$. By a limiting argument with \eqref{u1}-\eqref{KER:F},
we can reduce the verification of \eqref{stat2}
to showing
\begin{align}
 \sum_{i=1}^k c_i \sum_{n=1}^M I_n(F_{t,x_i,n})   \eqd  \sum_{i=1}^k c_i \sum_{n=1}^M I_n(F_{t, x_i+y,n})
\notag 
\end{align}
for any $(c_1, ... , c_k)\in\R^k$ and for any $M\in\N_{\geq 1}$.

Note that
\[
F_{t,x +y ,n}(\pmb{t_n},\pmb{x_n}, \pmb{z_n})
= F_{t,x,n}(\pmb{t_n},\pmb{x_n}-y, \pmb{z_n})
\]
with $\pmb{x_n}-y := (x_1 - y, x_2 - y, ..., x_n - y)$. This motivates us to
define a Poisson random measure $N_y$  on $Z$
by setting
\[
N_y(A\times B \times C) = N( A \times B_y \times C)
\quad{\rm with}\quad B_y:= \{ b-y: b\in B\}
\]
for every $(A, B, C)\in \cB(\R_+)\times \cB(\R) \times \cB(\R_0)$.
Then, it follows from the translational invariance of Lebesgue measure that
\begin{align}
N_y \eqd N.
\label{stat4}
\end{align}
Let $I^y_n$ denote the $n$-th multiple integral with respect to the compensated
version of   $N_y$; see Subsection \ref{SUB22}.
Therefore, we deduce from the definition of multiple integrals with
\eqref{int3a}, \eqref{int3b}, and \eqref{int4} that
\begin{align*}
 \sum_{i=1}^k c_i \sum_{n=1}^M I_n(F_{t, x_i+y,n})
& = \sum_{i=1}^k c_i \sum_{n=1}^M I^y_n(F_{t, x_i,n})  \\
&\eqd  \sum_{i=1}^k c_i \sum_{n=1}^M I_n(F_{t, x_i,n}),
\end{align*}

\noi
where the last step is a consequence of \eqref{stat4}.
Hence the proof of Lemma \ref{lem:stat} is completed now.
\qedhere

\end{proof}

The above Lemma \ref{lem:stat}
indicates that $\{u(t,x)\}_{x\in\R}$ is strictly stationary for every $t\in\R_+$.
The main goal of this subsection is to show the (spatial)
ergodicity of $\{u(t,x)\}_{x\in\R}$, and thus
answer the question \eqref{quest} affirmatively.
 To achieve this goal,
we exploit a criterion from \cite{CKNP21} (see Lemma \ref{lem:CKNP})
and take advantage of tools from Malliavin calculus on the Poisson space.
In particular, we need the  $L^2(\Om)$-bound  \eqref{D1est}
for the Malliavin derivatives of the solution $u(t,x)$.
Let us first recall the following variant of     \cite[Lemma 7.2]{CKNP21}.

\begin{lemma}\label{lem:CKNP}

 A strictly  stationary process $\{ Y(x)\}_{x\in\R^d}$
is ergodic provided that 
\begin{align}
\lim_{R\to\infty} \frac{1}{R^{2d}}
\Var\bigg(  \int_{[0,R]^d}  g \bigg(\sum_{j=1}^k b_j  Y(x + \ze_j)  \bigg) dx \bigg)
=0
\label{ERG1}
\end{align}

\noi
for all integers $k\geq 1$, for every $b_1, ... , b_k, \zeta_1, ... , \zeta_k\in\R^d$,
and for $g \in\{ x\mapsto \cos(x), x \mapsto \sin(x) \}$.

\end{lemma}

Lemma \ref{lem:CKNP} is essentially contained in (the proof of)
\cite[Lemma 7.2]{CKNP21}. But the statement of  \cite[Lemma 7.2]{CKNP21}
imposes more restrictive assumptions that are not useful in the current
Poisson setting, due to the lack of neat chain rule and enough moments. 
In fact, by directly applying  \cite[Lemma 7.2]{CKNP21}, 
 we can also obtain the spatial ergodicity
of $\{u(t,x): x\in\R\}$ (Theorem \ref{thm:main} (i))
but we have to  assume  ``$m_p <\infty$ for any finite $p\geq 2$''.

\medskip

In what follows, we present a proof of Lemma \ref{lem:CKNP}
for the sake of  completeness.

\begin{proof}[Proof of Lemma \ref{lem:CKNP}]

 This proof is essentially taken from   \cite[Lemma 7.2]{CKNP21}.
Since the condition \eqref{ERG1} holds for sine and cosine functions,
we can deduce from strict stationarity that

\noi
\begin{align}\label{ERG1a}
 \frac{1}{R^d} \int_{[0,R]^d} \exp \bigg( i\sum_{j=1}^k b_j  Y(x + \ze_j)  \bigg) dx
 \xrightarrow[R\to\infty]{\textup{in $L^2(\mathbb{P})$}}
 \bE \bigg[ \exp \bigg( i\sum_{j=1}^k b_j  Y( \ze_j)  \bigg) \bigg].
\end{align}
Let $\mathscr{I}$ denote the $\sigma$-algebra of invariant sets with respect to the
shifts $\{ \th_y: y\in\R\}$ in \eqref{shifty}. We argue as in the proof of
  \cite[Lemma 7.2]{CKNP21}: by invoking
von Neumann's mean ergodic  theorem (see, e.g.,  \cite[Chapter 2]{Peter89}),
we can get

\noi
\begin{align}\label{ERG1b}
 \frac{1}{R^d} \int_{[0,R]^d} \exp \bigg( i\sum_{j=1}^k b_j  Y(x + \ze_j)  \bigg) dx
 \xrightarrow[R\to\infty]{\textup{in $L^2(\mathbb{P})$}}
 \bE \bigg[ \exp \bigg( i\sum_{j=1}^k b_j  Y( \ze_j)  \bigg) \big|  \, \mathscr{I} \bigg].
\end{align}
Therefore, the right sides of \eqref{ERG1a} and \eqref{ERG1b} are equal for any $b_j, \zeta_j\in\R$.
This leads to the conclusion that
$\big(Y(\ze_1),\ldots, Y(\ze_k)\big)$ is independent of $\mathscr{I}$. 
Therefore, $\mathscr{I}$ is independent of the $\s$-algebra generated by $Y$, 
and in particular $\mathscr{I}$ is independent of itself. 
Hence $\mathscr{I}$ is the trivial $\sigma$-algebra.
This in turn completes our proof.
\qedhere

\end{proof}

\begin{proof}[Proof of Theorem \ref{thm:main}  \textup{(i)}]
By Lemma \ref{lem:stat}, $\{ u(t,x)\}_{x\in\R}$ is strictly stationary.
Then, we need to verify the condition \eqref{ERG1} in Lemma \ref{lem:CKNP}
to show the spatial ergodicity.

\smallskip

In what follows, we only consider the case where $g(x) = \cos(x)$, as the other case
can be proved verbatim. Let $k\in\N_{\geq1}$ and $b_1, ... ,b_k, \zeta_1, ... , \zeta_k\in\R$.
Recall from \eqref{FSol} that $G_t(x) = \tfrac{1}{2} \ind_{\{ |x| < t \}}$
and from Proposition \ref{prop:dec} (iii) that

\noi
\begin{align} \label{ERG1c}
\| D_{s, y, z} u(t, x) \| _2 \les  G_{t-s}(x-y) |z| \quad\textup{provided $m_2<\infty$.}
\end{align}

\noi
Therefore, we can deduce from
Poincar\'e inequality \eqref{Poi1}, Lemma \ref{lem:Fubi}, and Minkowski's inequality
with \eqref{add1b}
that

\noi
\begin{align}
\begin{aligned}
&\Var\bigg(  \int_{-R}^R  \cos \bigg( \sum_{j=1}^k b_j  u(t, x + \ze_j)  \bigg) dx \bigg) \\
&\quad
\leq \bE  \bigg[  \Big\| \int_{-R}^R D \cos \bigg( \sum_{j=1}^k b_j  u(t, x + \ze_j)  \bigg)  dx  
       \Big\|_{L^2(Z,\cZ, \fm)}^2  \bigg]\\
&\quad
= \int_{(0,t)\times\R\times\R_0}
\bigg\|  \int_{-R}^R D_{s, y, z} \cos \bigg( \sum_{j=1}^k b_j  u(t, x + \ze_j)  \bigg)  dx \bigg\|^2_2
dsdy \nu(dz) \\
&\quad
\leq
\int_{(0,t)\times\R\times\R_0}
  \bigg(  \int_{-R}^R  \Big\| \sum_{j=1}^k b_j  D_{s, y, z}     u(t, x + \ze_j)    \Big\|_2 dx
  \bigg)^2 dsdy \nu(dz),
\end{aligned}
\label{ERG2}
\end{align}

\noi
where the equality in \eqref{ERG2} follows essentially
from the fact that $D_{r, y, z}u(t,x)=0$ when $r\geq t$;
and this fact can be derived easily
from the explicit  chaos expansion \eqref{D-chaos}
(see also Lemma \ref{lem:CE2} (ii)).
Finally, in view of  the bound \eqref{ERG1c} and triangle inequality,
we can reduce the proof of \eqref{ERG1} to
showing for any $\ze\in\R$ that

\noi
\begin{align}
  \frac{1}{R^2} \int_{\R_+\times\R}
\bigg(  \int_{-R}^R  G_{t-s}(x+\ze -y) dx \bigg)^2 dsdy \to 0 \,\, \text{as $R\to\infty$.}
\label{ERG4}
\end{align}

It is clear that  with $\varphi_{t,R}$ as in \eqref{Rosen6b} and \eqref{Rosen7c},
\begin{align}
\begin{aligned}
\text{LHS of \eqref{ERG4}}
&=    \frac{1}{R^2} \int_0^t  \int_\R     \varphi^2_{t,R}(s, \ze -y) ds dy
 \ \to 0  \,\,\, \text{as $R\to\infty$.}
\end{aligned}
\notag
\end{align}

\noi
This proves \eqref{ERG4} and hence the spatial ergodicity of $\{u(t,x)\}_{x\in\R}$.
\qedhere
\end{proof}

\subsection{Central limit theorems}  \label{SUB42}

Recall from \eqref{FRT} the definition of the spatial integral
$F_R(t)$.  In view of Lemma \ref{lem:Fubi}, we can write
\begin{align}
F_R(t) =\int_{-R}^R \big[ u(t,x) -1 \big] dx
= \sum_{n=1}^\infty I_n\bigg( \int_{-R}^R F_{t,x, n} dx \bigg)
\notag 
\end{align}
with $F_{t,x,n }$ as in \eqref{KER:F}.

 \smallskip
 
This section is divided into three parts: in Part $\1$,
we establish the limiting covariance structure of the process
$\{F_R(t)\}_{t\in\R_+}$ stated in Theorem \ref{thm:main} (ii), and in particular the limiting
variance at fixed time $t > 0$ that will be used
in Part $\II$;
then  Part $\II$ is devoted to the proof of  Theorem \ref{thm:main} \rm (iii),
while we prove the functional CLT (Theorem \ref{thm:main} \rm (iv)) in Part $\III$.

\medskip
\noi
$\bul$ {\bf Part $\1$: Limiting covariance structure.}

\begin{proof}[Proof of Theorem \ref{thm:main} \rm (ii)]
In this part,   we only
assume  $m_2 < \infty$.
We begin with the covariance of $u(t,x)$ and $u(s,y)$:
\begin{align}
\begin{aligned}
\bE[ u(t,x) u(s,y) ] -1
& = \sum_{n\geq 1} n! \langle \wt{F}_{t,x,n}, \wt{F}_{s,y,n}\rangle_\fH \\
& = \sum_{n\geq 1} n! m_2^n \langle \wt{f}_{t,x,n},
\wt{f}_{s, y,n}\rangle_{L^2(\R_+\times\R)^{\otimes n}},
\end{aligned}
\label{COV1}
\end{align}

\noi
where $f_{t,x,n}$, given as in \cite[equations (1.7), (1.8)]{BNQSZ}, is determined
by
\begin{align}
F_{t,x,n}(\pmb{t_n}, \pmb{x_n},\pmb{z_n}) = f_{t,x, n}(\pmb{t_n}, \pmb{x_n})
\prod_{j=1}^n z_j .
\label{KER:f}
\end{align}

\noi
Observe that the RHS of \eqref{COV1} coincides with 
the covariance of $U(t,x)$ and $U(s,y)$, when 
$U$ is the unique mild solution to the following
stochastic wave equation with space-time Gaussian white noise $\dot{W}$ on $\R_+\times\R$:

\noi
\begin{align} \label{wave_U}
\begin{cases}
 \dt^2 U(t,x) = \dx^2 U(t,x) + \sqrt{m_2} \,U(t,x) \dot{W}(t,x), \quad (t,x)\in (0,\infty)\times\R \\[0.8em]
 U(0,x) = 1 \quad{\rm and}\quad \dt U(0,x) = 0, \quad x\in\R;
\end{cases}
\end{align}
see also \cite[(1.1)]{DNZ20} with $\s(x) = \sqrt{m_2} \,x$.
With 
\[
G_R(t) = \int_{-R}^R  [ U(t,x) - 1 ] dx,
\]
it is easy to see that 

\noi
\begin{align}
\begin{aligned}
\bE\big[ G_R(t)  G_R(s) \big] 
&=   \int_{-R}^R   \int_{-R}^R  {\rm Cov}\big(  U(t,x) , U(s, y) \big) dxdy \\
&=   \int_{-R}^R   \int_{-R}^R  {\rm Cov}\big(  u(t,x) , u(s, y) \big) dxdy 
=  \bE\big[ F_R(t)  F_R(s) \big]. 
\end{aligned}
\label{COV3}
\end{align}

\noi
That is, it suffices to find the limiting covariance structure of $\{G_R(t): t\in\R_+\}$ now. 
And it has been established in \cite{DNZ20} that 
\begin{align}\label{COV_DNZ1}
\frac{1}{R} \bE\big[ G_R(t)  G_R(s) \big] 
\xrightarrow{R\to+\infty} 
2 m_2	\int_0^{t\wedge s} (t-r) (s-r) \bE[ U^2(r,0) ] dr ;
\end{align}

\noi
see Proposition 3.1 (on page 3025)   
and 
 Remark 2 (on pages 3029-3030)  therein
 with particularly $\s(x) = \sqrt{m_2}\, x$.
Meanwhile, the second moment formula for $ \bE[ U^2(r,0) ] $ can be found in the literature:
\begin{align}
 \bE[ U^2(r,0) ]  = \cosh\bigg(r \sqrt{ \frac{m_2}{2} } \, \bigg); 
\label{cosh}
\end{align}

\noi
see \cite[Example 2.2]{CGS}.\footnote{In our case, $u_0 =1, u_1=0,  \lambda^2 = m_2$, and $\nu = 2$
so that the formula in the reference reduces to 
$ \bE[ U^2(r,0) ] =  E_2(m_2 t^2 /2  )$, with $E_2(z) = \cosh(\sqrt{z})$ given in  \cite[(A.5)]{CGS}.
This leads to the formula \eqref{cosh}.}
Then, combining \eqref{COV_DNZ1} and \eqref{cosh} yields
\begin{align}
\Sigma_{t,s} = 2 m_2	\int_0^{t\wedge s} (t-r) (s-r) \cosh\bigg(r \sqrt{ \frac{m_2}{2} } \, \bigg) dr.
\label{COV_DNZ}
\end{align}

\noi
In particular, we have for any fixed $t\in(0,\infty)$,
\begin{align}
\sigma_R(t) := \sqrt{ \Var \big( F_R(t)  \big)  }  \sim \sqrt{\Sigma_{t,t}R}
\label{COV_s}
\end{align}
as $R\to\infty$;   while it is clear that $\Sigma_{t, t} > 0$ for every $t > 0$.
\qedhere

\end{proof}

\begin{remark}\label{rem43}
\rm
(i) As the first step in establishing the central limit theorems,
 we find the exact order of the limiting variance \eqref{COV_s}. 
Using the available expressions of chaos expansion \eqref{u1}--\eqref{rho_t},
one can perform similar computations as in \cite[Subsection 4.1.1]{BNQSZ}
and obtain a formula for the limiting covariance structure $\Sigma$, 
which is however not explicit.
In the above proof, we  used 
a trick of transferring to the setting of Gaussian white noise, where
exact computations would lead to the explicit formula \eqref{COV_S}
for the limiting covariance structure $\Sigma$.
 
\smallskip
\noi
(ii) One can see from \cite[Lemma 3.4 on page 3028]{DNZ20} that for every $t >0$,
${\rm Var}(  G_R(t) ) >0 $
 for every $R > 0$. Then we deduce from \eqref{COV3} that $\s_R(t)$, defined as in \eqref{COV_s},
 is strictly positive for every $R > 0$.

\end{remark}

\bigskip
\noi
$\bul$ {\bf Part $\II$: Quantitative central limit theorems.}

\begin{proof}[Proof of Theorem \ref{thm:main} \rm (iii)]
Throughout this proof, we assume that  $m_{1+\al}$ and $m_{2+2\al}$ are finite
for some  $\al\in(0,1]$. By interpolation, $m_2$ is finite automatically.
Recall that $D^+ = D$ on $\dom(D)$
and $D^+ D^+ = D^2$ on $\dom(D^2)$.
Then it is easy to see from Proposition \ref{prop:dec} (iii) and Lemma \ref{lem:Fubi}
that $F_R(t)\in \dom(D^2)$
such that 

\noi
\begin{align}
\begin{aligned}
\| D_{r, y, z} F_R(t) \|_{2+2\al}
&\leq  \int_{-R}^R \|  D_{r, y, z} u(t,x) \|_{2+2\al} \, dx \\
&\leq C'_{t, 2+2\al, \nu}  |z| \cdot \int_{-R}^R  G_{t-r}(x-y) dx \\
&=C'_{t, 2+2\al, \nu} \varphi_{t,R}(r,y) |z|
\end{aligned}
\label{Q1}
\end{align}

\noi
 with $C'_{t, 2+2\al, \nu}$ as in \eqref{CTPA},
and

\noi
\begin{align}
\begin{aligned}
\| D_{r_1, y_1, z_1}D_{r_2, y_2, z_2} F_R(t) \|_{2+2\al}
&\leq  \int_{-R}^R \| D_{r_1, y_1, z_1}D_{r_2, y_2, z_2}  u(t,x) \|_{2+2\al} \, dx \\
&\leq C''_{t, 2+2\al, \nu} |z_1z_2| \cdot \int_{-R}^R  \wt{f}_{t,x, 2}(r_1, y_1, r_2, y_2)   dx,
\end{aligned}
\label{Q2}
\end{align}

\noi
 with $C''_{t, 2+2\al, \nu}$ as in \eqref{CTPB},
where $\varphi_{t,R}$ is as in \eqref{Rosen6b} and  $f_{t,x,2}$ is as in \eqref{KER:f}
with
\begin{align}
\begin{aligned}
&\wt{f}_{t,x,2}(r_1, y_1, r_2, y_2 )  \\
&\quad =
 \frac{1}{2} \big[ G_{t-r_1}(x-y_1) G_{r_1- r_2}(y_1 -y_2) +
 G_{t-r_2}(x-y_2) G_{r_2- r_1}(y_2 -y_1) \big]
 \label{tf2}
\end{aligned}
\end{align}

\noi
with the convention \eqref{convention} in mind. 
Note that in the steps \eqref{Q1}-\eqref{Q2}, we 
need to assume the finiteness of $m_{2+2\al}$
for applying Proposition \ref{prop:dec}. 

In what follows, we apply Proposition \ref{prop:tara}
to derive the desired quantitative CLTs. More precisely,
we will compute the seven quantities $\ga_1, ... , \ga_7$
as in \eqref{gamma16}-\eqref{gamma7} with $F =  F_R(t) / \s_R(t) $ and $p =q = 1+\al$.
In the following, we will show that
\[
\text{$\ga_i^{1+\al} \les R^{-\al}$ for $i\neq 3$ and
$\ga_3 \les R^{-\al}\ind_{\{ 0< \al \leq \frac12 \}} + R^{-\frac12}\ind_{\{   \frac12  < \al \leq1\}}
$}.
\]
The above bounds, together with Proposition \ref{prop:tara}, will conclude the proof
of Theorem \ref{thm:main} (iii).

\medskip

To ease the notations, we write $\xi_i = (r_i, y_i, z_i)\in Z$
and $\fm(d\xi_i) = dr_i dy_i \nu(dz_i)$ for $i=1,2,3$.

\medskip
\noi
$\bul$ {\bf Estimation of $\ga_1$.}
We can first deduce from  \eqref{COV_s},
\eqref{Q1}, and \eqref{Q2} that

\noi
\begin{align}
\begin{aligned}
\ga_1^{1+\al}
&\les \frac{1}{R^{1+\al}} \int_{Z} \bigg( \int_Z
\| D_{r_2, y_2,z_2} F_R(t) \|_{2+2\al} \\
&\qquad\qquad\qquad \cdot \| D_{r_1, y_1,z_1} D_{r_2, y_2,z_2} F_R(t)  \|_{2+2\al} \, \fm(d\xi_2)
  \bigg)^{1+\al}  \, \fm(d\xi_1) \\
&\les  \frac{  \big( C'_{t, 2+2\al, \nu}C''_{t, 2+2\al, \nu} m_2 \big)^{1+\al} m_{1+\al}}{R^{1+\al}}
\int_0^t \int_\R    \bigg( \int_0^t \int_\R   dr_2 dy_2
\int_{[-R, R]^2} dx_1dx_2
\\
&\qquad\qquad\qquad \cdot G_{t-r_2}(x_2 - y_2)
\wt{f}_{t, x_1, 2} (r_1, y_1, r_2, y_2 )  \bigg)^{1+\al}  dr_1 dy_1,
\end{aligned}
\notag
\end{align}
where $\wt{f}_{t, x, 2}$ is as in \eqref{tf2}.
It is easy to verify that 
\begin{align}
\begin{aligned}
G_{t-r_2}(x-y_2) &\leq G_t(x-y_2)   \\
\wt{f}_{t, x, 2} (r_1, y_1, r_2, y_2 )
&\leq  G_{t}(x-y_2) G_t(y_1-y_2).
\end{aligned}
\label{bdd:G}
\end{align}
for any $(r_1, r_2, x, y_1, y_2)\in [0,t]^2 \times\R^3$.
Therefore, in view of the above bounds,
 it is then sufficient to show 
\begin{align} \label{ga1aa}
\begin{aligned}
&  \int_\R    \bigg(  \int_\R     dy_2
\int_{[-R, R]^2} dx_1dx_2 G_{t}(x_2 - y_2)
 G_{t}(x_1-y_2) G_t(y_1-y_2)  \bigg)^{1+\al}   dy_1
  \les R,
\end{aligned}
\end{align}

\noi
while the omitted temporal integration will yield a factor $t^{2+\al}$. Note that the integral
in \eqref{ga1aa} with respect to $dx_1dx_2 dy_2$ is uniformly bounded by $t^3$.
It follows that
\begin{align*}
&\textup{LHS of \eqref{ga1aa}}  \\
&\leq t^{3\al}   \int_\R  dy_1  \int_\R     dy_2
\int_{[-R, R]^2} dx_1dx_2 G_{t}(x_2 - y_2)
 G_{t}(x_1-y_2) G_t(y_1-y_2)  \\
 &\leq    t^{3\al}  \cdot t^3 \cdot 2R = 2 t^{3+3\al} R
\end{align*}
by performing integration in the order of $dy_1, dx_1, dy_2$, and then $dx_2$.
Hence, \eqref{ga1aa} is proved. That is, we just proved that

\noi
\begin{align}
\ga_1^{1+\al} \les R^{-\al}
\quad\textup{and equivalently}
\quad
\ga_1  \les R^{- \frac{\al}{1+\al}}.
\label{ga1d}
\end{align}
In this step, we need to assume the finiteness of $m_{1+\al}$
  and $m_{2+2\al}$.

\medskip

\noi
$\bul$ {\bf Estimation of $\ga_2$.}
We can deduce from  \eqref{Q2}, \eqref{bdd:G}, and  \eqref{COV_s} that

\noi
\begin{align}
\begin{aligned}
&\ga_2^{1+\al}
\les \frac{ (C''_{t, 2+2\al, \nu})^2  }{R^{1+\al}} \int_Z \bigg( \int_Z |z_1 z_2|^2
 \bigg[ \int_{-R}^R G_{t}(x-y_2) G_t(y_1-y_2) dx \bigg]^2   \\
&\qquad\qquad\qquad\qquad\qquad
 \ind_{  \{  0\leq r_1,r_2\leq t  \}   }  dr_2 dy_2 \nu(dz_2) \bigg)^{1+\al}   dr_1 dy_1 \nu(dz_1)   \\
 &\leq \frac{m_2^{1+\al} m_{2+2\al}  (C''_{t, 2+2\al, \nu})^2 }{R^{1+\al}} t^{2+\al} \int_\R \bigg( \int_\R
 \bigg[ \int_{-R}^R G_{t}(x-y_2) dx \bigg]^2 G_t(y_1-y_2)   dy_2  \bigg)^{1+\al}
  dy_1.
  \end{aligned}
    \label{ga2b}
 \end{align}

\noi
It is easy to see from \eqref{factsG} that
\begin{align}
\int_{-R}^R G_{t}(x-y_2)dx \leq t
\quad{\rm and}\quad
\int_\R
 \bigg[ \int_{-R}^R G_{t}(x-y_2)dx \bigg]^2 G_t(y_1-y_2)   dy_2  \leq t^3.
 \label{bdd:G2}
\end{align}
Therefore, we continue with \eqref{ga2b}:

\noi
\begin{align*}
\ga_2^{1+\al}
&\les \frac{1}{R^{1+\al}}   \int_\R \bigg( \int_\R
 \bigg[ \int_{-R}^R G_{t}(x-y_2) dx \bigg]^2 G_t(y_1-y_2)   dy_2  \bigg)   dy_1 \\
 &\les  \frac{1}{R^{1+\al}}   \int_\R \bigg( \int_\R
 \bigg[ \int_{-R}^R G_{t}(x-y_2) dx \bigg] G_t(y_1-y_2)   dy_2  \bigg)   dy_1 \\
 &\les R^{-\al}
\end{align*}
by performing the integration in the order of $dy_1, dy_2$, and $dx$.
That is, we just proved that

\noi
\begin{align}
\ga_2^{1+\al} \les  R^{-\al}
\quad\textup{and equivalently}
\quad
\ga_2  \les R^{- \frac{\al}{1+\al}}.
\label{ga2c}
\end{align}
In this step, we   need  to assume the finiteness of $m_{2+2\al}$.

\medskip
\noi
$\bul$ {\bf Estimation of $\ga_3$.} In this step, we fix
\begin{align}
q=
\begin{cases}
{\rm(i)} \,\,\,\, 1+2\al  \,\,\,&\text{if $\al\in(0,\frac12]$}  \\[1em]
{\rm(ii)} \,\,\,\, 2  \,\,\,&\text{if $\al\in(\frac12, 1]$}
 \end{cases}
 \label{ga3a}
 \end{align}
(so that $q\in(1,2]$) and estimate the quantity
$\ga_3$ defined in \eqref{gamma16} with $F = F_R(t)/\s_R(t)$.
We deduce from \eqref{Q1} and \eqref{COV_s} with \eqref{bdd:G} and \eqref{factsG} that

\noi
\begin{align}
\begin{aligned}
\ga_3
& = 2 \frac{1}{\s^{q+1}_R(t)} \int_Z \| D_{r, y, z}F_R(t) \|_{q+1}^{q+1} \, dr dy \nu(dz) \\
&\les \frac{m_{q+1}   (C'_{t, q+1, \nu})^{q+1}  }{R^{\frac{q+1}{2}}}
\int_0^t \bigg(  \int_\R \bigg| \int_{-R}^R G_{t}(x-y) dx \bigg|^{q+1} \,  dy \bigg) dr \\
&\les \frac{1}{R^{\frac{q+1}{2}}}  \int_\R    \bigg( \int_{-R}^R G_{t}(x-y) dx \bigg) dy 
= 2t \cdot R^{-\frac{q-1}{2}}.
\end{aligned}
\label{ga3b}
\end{align}

\noi
Therefore, it follows from \eqref{ga3a} and \eqref{ga3b} that
\begin{align}
\ga_3
\les R^{-\al}\ind_{\{ 0< \al \leq \frac12 \}} + R^{-\frac12}\ind_{\{   \frac12  < \al \leq1\}}.
\label{ga3}
\end{align}
\noi
In this step, the finiteness of $m_{q+1}$ is guaranteed by that of  
$m_{2+2\al}$. 
Note that the rate in \eqref{ga3} is faster than those in \eqref{ga1d} and \eqref{ga2c}.

\medskip

Therefore, we can deduce from \eqref{2nd:WassB} in Proposition \ref{prop:tara}
with \eqref{ga1d}, \eqref{ga2c}, and \eqref{ga3} that
\[
d_{\rm FM}\Big( \frac{F_R(t)}{\s_R(t)}, \NN(0,1)   \Big)
\leq d_{\rm Wass}\Big( \frac{F_R(t)}{\s_R(t)}, \NN(0,1)   \Big) \les R^{-\frac\al{1+\al}}.
\]

Next, we continue to estimate $\ga_4, \ga_5, \ga_6$, and $\ga_7$
for getting the Kolmogorov bound
 and we will just hide the constants $C'_{t, 2+2\al, \nu}$
and $C''_{t, 2+2\al, \nu}$ in the estimations.

\medskip
\noi
$\bul$ {\bf Estimation of $\ga_4$.} The estimation of the quantity
$\ga_4$ can be done in the same way as in \eqref{ga3b}:

\noi
\begin{align*}
\ga_4^{1+\al}
&\les  \frac{1}{\s^{2+2\al}_R(t)} \int_Z \| D_{r, y, z}F_R(t) \|_{2+2\al}^{2+2\al} \, dr dy \nu(dz) \\
&\les  \frac{m_{2+2\al}}{R^{1+\al}} \int_0^t dr  \int_\R dy \bigg| \int_{-R}^R G_{t}(x-y) dx \bigg|^{2+2\al}  \\
&\les R^{-\al}.
\end{align*}

\noi
That is, we have
\begin{align}
\ga_4^{1+\al} \les R^{-\al}
\quad\textup{and equivalently}
\quad
\ga_4  \les R^{- \frac{\al}{1+\al}}.
\label{ga4c}
\end{align}
In this step, we need to assume the finiteness of $m_{2+2\al}$.

\medskip
\noi
$\bul$ {\bf Estimation of $\ga_5$.} We first deduce from
\eqref{Q2}, \eqref{COV_s}, and \eqref{bdd:G} with \eqref{bdd:G2}
 that

\noi
\begin{align}
\begin{aligned}
\ga_5^{1+\al}
&  \les  \frac{m^2_{2+2\al}}{R^{1+\al}}
\int_0^t dr_1 \int_\R dy_1 \int_0^t dr_2 \int_\R dy_2
\bigg(  \int_{-R}^R G_{t}(x-y_2)G_t(y_1-y_2) dx \bigg)^{2+2\al} \\
&\les \frac{1}{R^{1+\al} }
  \int_\R dy_1  \int_\R dy_2
\bigg(  \int_{-R}^R G_{t}(x-y_2)G_t(y_1-y_2) dx \bigg)  \\
&\les R^{-\al},
\end{aligned}
\label{ga5a}
\end{align}
by performing integration in the order of   $dy_1, dy_2$, and $dx$.
That is, we have
\begin{align}
\ga_5^{1+\al} \les R^{-\al}
\quad\textup{and equivalently}
\quad
\ga_5  \les R^{- \frac{\al}{1+\al}}.
\label{ga5}
\end{align}
In this step, we need to assume the finiteness of $m_{2+2\al}$.

\medskip
\noi
$\bul$ {\bf Estimation of $\ga_6$.}   Note that $\| D_{r_1,y_1,z_1} F_R(t)\|_{2+2\al} \les t |z_1|$.
Similarly as in \eqref{ga5a}, we can write

\noi
\begin{align*}
\ga_6^{1+\al}
&\les \frac{m_{1+\al} m_{2+2\al}   }{R^{1+\al}}
\int_0^t dr_1 \int_\R dy_1 \int_0^t dr_2 \int_\R dy_2
\bigg(  \int_{-R}^R G_{t}(x-y_2)G_t(y_1-y_2) dx \bigg)^{1+\al} \\
&\les R^{-\al}.
\end{align*}
That is, we have
\begin{align}
\ga_6^{1+\al} \les R^{-\al}
\quad\textup{and equivalently}
\quad
\ga_6  \les R^{- \frac{\al}{1+\al}}.
\label{ga6a}
\end{align}
In this step, we need to assume the finiteness of $m_{2+2\al}$ and $m_{1+\al}$.

\medskip
\noi
$\bul$ {\bf Estimation of $\ga_7$.}   Similarly as in the estimation of $\ga_6$,
we roughly bound
\[
\| D_{r_1,y_1, z_1} F_R(t) \|_{2+2\al} \| D_{r_2,y_2, z_2} F_R(t) \|^{1+2\al}_{2+2\al}
\les t^{2\al+2} |z_1|  \cdot |z_2|^{1+2\al},
\]
and we can write
\begin{align*}
\ga_7^{1+\al}
&\les \frac{m_2 m_{1+2\al}   }{R^{1+\al}}
\int_0^t dr_1 \int_\R dy_1 \int_0^t dr_2 \int_\R dy_2
\bigg(  \int_{-R}^R G_{t}(x-y_2)G_t(y_1-y_2) dx \bigg) \\
&\les R^{-\al}.
\end{align*}

\noi
That is, we have
\begin{align}
\ga_7^{1+\al} \les R^{-\al}
\quad\textup{and equivalently}
\quad
\ga_7  \les R^{- \frac{\al}{1+\al}}.
\label{ga7}
\end{align}
In this step, we need to assume the finiteness of $m_{1+2\al}$
 and $m_{2+2\al}$, while  the finiteness of $m_{1+2\al}$  is guaranteed
by the finiteness of $m_{1+\al}$ and $m_{2+2\al}$.

\medskip

Therefore, it follows from \eqref{2nd:KolB} in  Proposition \ref{prop:tara}
with \eqref{ga1d}, \eqref{ga2c},    \eqref{ga4c}, \eqref{ga5}, \eqref{ga6a}, and \eqref{ga7}
 that
\[
 d_{\rm Kol}\Big( \frac{F_R(t)}{\s_R(t)}, \NN(0,1)   \Big) \les  R^{-\frac{\al}{1+\al}}.
\]

Hence the proof of part  (iii) in Theorem \ref{thm:main} is completed.
\qedhere

\end{proof}

\medskip
\noi
$\bul$
{\bf Part $\III$: Functional central limit theorems.}  

In this part, we present the proof of Theorem \ref{thm:main} (iv).
The remaining part of the proof consists of two steps: we first show the
convergence  in finite-dimensional distributions and then
conclude this section by proving the tightness of the process
$\{   \frac{1}{\sqrt{R}} \{F_R(t)\}_{t\in\R_+} : R\geq 1\}$.

\medskip
\noi
$\bul$ {\bf  Step 1: Convergence  in finite-dimensional distributions.}
Fix any $0< t_1 <  ... <  t_m <\infty$ with $m\in\N_{\geq 2}$.
We need to show that
\[
  \Big( \frac{1}{\sqrt{R}} F_R(t_1), ... , \frac{1}{\sqrt{R}} F_R(t_m)\Big)
\]
converges in law to a centered Gaussian vector on $\R^m$
with covariance matrix $(\Sigma_{t_i, t_j})_{i,j=1,..., m}$,
where $\Sigma$ is as in \eqref{COV_S}.
Then, it suffices to show that
\begin{equation}
\label{XR-conv}
X_R:=\sum_{j=1}^m b_j  \frac{F_R(t_j)}{\sqrt{R}}
\quad\textup{converges in law to}
\quad
\sum_{j=1}^m b_j  \mathcal{G}_{t_j}, \quad \mbox{as $R \to \infty$}
\end{equation}
for any integer $m\geq 1$, for any $b_1, ... , b_m\in\R$, and for any
$t_1, ... , t_m\in\R_+$,
where $\mathcal{G}$ is a centered continuous Gaussian process with covariance
structure $\Sigma$ given as in   \eqref{COV_S}. 
Let
\[
\tau^2:= \Var \left(\sum_{j=1}^m b_j  \mathcal{G}_{t_j} \right)= \sum_{j,k=1}^m b_j b_k \Sigma_{t_j, t_k}.
\]
Then \eqref{XR-conv} is equivalent to 
\begin{equation}
\label{XR-conv1}
X_R  \,\,\, \mbox{converges in law to $\NN(0,\tau^2)$} 
\,\,\, \mbox{as $R \to \infty$}.
\end{equation}
Moreover, by \eqref{COV3}-\eqref{COV_DNZ},
\[
\tau_R^2:= \Var (X_R)=\frac{1}{R}\sum_{j,k=1}^{m}b_j b_k \bE[F_R(t_j) F_R(t_k)] \to \tau^2 
\,\,\, \mbox{as $R \to \infty$}.
\]
 The rest of the proof is trivial if $\tau^2 = 0$. It is also easy to see from the above limit  that
if $\tau^2 > 0$, then $\tau_R^2 > 0$ for large $R$.
Then, without losing any generality, we will assume that  both 
$\tau^2_R$ and $\tau^2$ are strictly positive for every $R$.

From the Wasserstein bound \eqref{2nd:WassB} in Proposition \ref{prop:tara}, we deduce that

\noi
\begin{align}
d_{\rm Wass}\left( \frac{X_R}{\tau_R}, \NN(0,1) \right)
\leq  \ga_1 +  \ga_2 +  \ga_3,
\label{fddW}
\end{align}
where $  \ga_1,  \ga_2$, and $ \ga_3$ are defined as in \eqref{gamma16} with
$F =  X_R$.
The rest of the arguments are almost identical to those in Part $\II$
that we  sketch in the following.
First, we write
\begin{align*}
\ga_1^{1+\al}
&\les \frac{1}{R^{1+\al}} \int_Z \bigg[ \int_Z \Big\| D_{\xi_2}  \sum_{j=1}^m b_j  F_R(t_j) \Big\|_{2+2\al}
\Big\| D_{\xi_1}D_{\xi_2}  \sum_{k=1}^m b_k  F_R(t_k) \Big\|_{2+2\al} \fm(d\xi_2) \bigg]^{1+\al} \fm(d\xi_1)\\
&\les  \frac{1}{R^{1+\al}} \sum_{j,k=1}^m \int_Z \bigg[ \int_Z \big\| D_{\xi_2}   F_R(t_j) \big\|_{2+2\al}
\big\| D_{\xi_1}D_{\xi_2}  F_R(t_k) \big\|_{2+2\al} \fm(d\xi_2) \bigg]^{1+\al} \fm(d\xi_1).
\end{align*}
Note that our estimations in Part $\II$ can be carried out in the same way for $t_j \neq t_k$,
and therefore, we still get $\ga_1^{1+\al} \les R^{-\al}$. In the same manner,
we can obtain the asymptotical negligibility of $\ga_2$ and $\ga_3$, and hence that of the Wasserstein
distance in \eqref{fddW} under the assumption \eqref{cond:al}.

Finally,
\begin{align*}
d_{\rm Wass}\bigg( \frac{X_R}{\tau}, \NN(0,1) \bigg)& \leq
d_{\rm Wass}\bigg( \frac{X_R}{\tau}, \frac{X_R}{\tau_R} \bigg)+
d_{\rm Wass}\bigg( \frac{X_R}{\tau_R}, \NN(0,1) \bigg)\\
& \leq \left|\frac{1}{\tau}-\frac{1}{\tau_R} \right| \bE|X_R|+d_{\rm Wass}\bigg( \frac{X_R}{\tau_R}, \NN(0,1) \bigg) \to 0 \,\,\, \mbox{as $R \to \infty.$}
\end{align*}
This implies \eqref{XR-conv1} and
concludes the proof of the convergence of the finite-dimensional distributions.

\medskip
\noi
$\bul$ {\bf Step 2: Tightness.} For tightness, we only need to assume the finiteness of $m_2$.
We first deduce from Proposition \ref{Prop:Rosen} (ii)  (with $p=2$)
and   Kolmogorov's continuity theorem
(see, e.g., \cite[Theorem 4.23]{Kall21})
that for each $R \geq 1$, the process $F_R:=\{ F_R(t)\}_{t\in\R_+}$
admits a continuous modification that is almost surely
locally $\be$-H\"older continuous for any $\be\in(0, \frac12)$.
Moreover, the bound \eqref{Rosen2a} in Proposition \ref{Prop:Rosen} (ii)  (with $p=2$),
together with the tightness criterion of Kolmogorov-Chentsov
(see, e.g., \cite[Theorem 23.7]{Kall21}),
implies that
$
 \big\{  \frac{1}{\sqrt{R} } F_R \big\}_{R\geq 1}
$
 is a tight family of continuous processes;
 that is, a tight family of random variables
 with values in $C(\R_+; \R)$.

Combining the above two steps, we conclude the desired functional
CLT under the assumption \eqref{cond:al}.
Hence,  we just finished the proof of Theorem \ref{thm:main}.
\hfill $\square$

\appendix

\section{Proof of the equivalence  (\ref{def:MP})}
\label{APP}

Recall from \eqref{LbA} and \eqref{IDlaw} that
$
L_b(A) = b\cdot\Leb(A) + M(A) + K(A),$
with
\[
M(A): = \int_{A\times\{|z| \leq 1\} } z \wh{N}(dt, dx, dz)
\,\,\,
\textup{independent of}
\,\,\,
K(A): = \int_{A\times\{|z| > 1\} } z N (dt, dx, dz).
\]

In what follows, we record a few facts on $M(A)$ and $K(A)$:

\noi
\begin{itemize}
\item[(i)]

the  characteristic function of $M(A)$ is given by
\begin{align*}
\bE\big[ e^{i \lambda M(A)} \big]
= \exp\Big( \Leb(A) \int_{\{|z| \leq 1 \}} (e^{i\lambda z} - 1 - i\lambda z ) \nu(dz) \Big),
\end{align*}
and by Lebesgue's differentiation theorem with the dominance
condition \eqref{LevyM}, we deduce that the above characteristic
function is infinitely  differentiable and can be extended to an entire
function on $\mathbb{C}$. This implies in particular
that the random variable $M(A)$ has finite exponential moments:

\noi
\begin{align}
\bE\big[ e^{  c | M(A) |} \big] <\infty
\label{APP0}
\end{align}

\noi
for every $c>0$; see Lemma 25.7 in \cite{Sato99}.

\item[(ii)] $K(A)$ is compound Poisson random variable that can be expressed
as follows:
\begin{align}
K(A) = \sum_{j=1}^Q Y_j
\label{APP0b}
\end{align}
with $\{Y_j\}_{ j \geq 1}$ independent   random variables
with common distribution $\frac{1}{\nu( \{ |z| >1 \}) } \nu|_{\{|z| >1 \}}$,
and $Q$ a Poisson random variable with mean $\mathbf{M}:=\Leb(A)\nu( \{ |z| >1 \})$ that is independent
of $\{Y_j\}_{ j \geq 1}$.

\end{itemize}

Put $\langle x\rangle = \sqrt{ 1+ x^2 }$. It is easy to see that for any finite $p>0$
and for any  {\it finite} measure $\mu$ on $\R$,
we have
\[
 \int_{\R }  \langle x\rangle^p \mu(dx)  \sim  \int_{\R }  |x|^p \mu(dx),
\]
from which we deduce that the equivalence \eqref{def:MP} can be rewritten as
\begin{align}
\bE\big[   \langle L_b(A) \rangle^p \big] <\infty
\,\,\,
\Longleftrightarrow
\,\,\,
\int_{\{|z| > 1 \}}  \langle x\rangle^p \nu(dz) <\infty.
\label{APP1}
\end{align}

\begin{proof}[Proof of \eqref{APP1} and \eqref{def:MP}]
We use the same argument as in the proof of Theorem 25.3 of \cite{Sato99}.
Fix $p\in(0,\infty)$. Observe first that  the function   $x\in\R\mapsto \langle x \rangle^{p}$
 is sub-multiplicative meaning
that
\begin{align}
\langle x + y\rangle^p \leq 2^p \langle x\rangle^p \cdot \langle y\rangle^p.
\label{APP2}
\end{align}

First, assume that $\bE\big[   \langle L_b(A) \rangle^p \big] <\infty$, i.e.
\[
\bE\big[   \langle b \cdot \Leb(A) + M(A) + K(A) \rangle^p \big] <\infty.
\]
It follows that for {\it some} $x_0\in\R$, we have
$
\bE\big[   \langle x_0  + K(A) \rangle^p \big] <\infty.
$
Then, we can deduce from \eqref{APP2} with
$K(A) = K(A) + x_0 + (-x_0)$ that

\noi
\begin{align}
\bE\big[   \langle  K(A) \rangle^p \big]
& \leq  2^p  \langle  - x_0   \rangle^p     \bE\big[   \langle x_0  + K(A) \rangle^p \big]  <\infty.
\label{APP2b}
\end{align}
Note that we can get  from \eqref{APP0b}
that

\noi
\begin{align}
\begin{aligned}
\bE\big[   \langle  K(A) \rangle^p \big]
&= \sum_{n=0}^\infty e^{- \mathbf{M} } \frac{\mathbf{M}^n  }{n!}
\bE\big[
 \langle  Y_1 + ... + Y_n   \rangle^p \big] \\
 &\geq e^{- \mathbf{M} } \mathbf{M} \cdot
\bE\big[  \langle  Y_1    \rangle^p \big],
\end{aligned}
\label{APP3}
\end{align}
which,  together with \eqref{APP2b}, implies
$\int_{\{|z| > 1 \}}  \langle x\rangle^p \nu(dz) <\infty.$

For the other direction, we can write by using
\eqref{APP2}, \eqref{APP0}, and \eqref{APP3} with independence
among $Y_j$'s
that

\noi
\begin{align*}
\bE\big[  \langle L_b(A) \rangle^p \big]
&\les 1 + \bE\big[   \langle  K(A) \rangle^p \big]  \\
&\leq 2  +e^{- \mathbf{M} } \mathbf{M} \cdot
\bE\big[  \langle  Y_1    \rangle^p \big]
+ \sum_{n=2}^\infty e^{- \mathbf{M} } \frac{\mathbf{M}^n  }{n!}
\bE\big[ 2^{p(n-1)} \langle  Y_1 \rangle^p \cdots  \langle Y_n \rangle^p \big] \\
&\leq 2 +  e^{- \mathbf{M} } \mathbf{M} \cdot \bE\big[  \langle  Y_1    \rangle^p \big]
+ \sum_{n=2}^\infty e^{- \mathbf{M} } \frac{\mathbf{M}^n  }{n!}
 2^{p(n-1)} \big( \bE [  \langle  Y_1 \rangle^p  ] \big)^n <\infty,
\end{align*}
provided that $ \bE [  \langle  Y_1 \rangle^p  ] \sim  \int_{\{|z| > 1 \}}  \langle x\rangle^p \nu(dz) <\infty$. Hence the equivalence \eqref{APP1} is verified, and so is
the equivalence \eqref{def:MP}.
\end{proof}

\begin{ackno}\rm
G.Z. would like to thank David Nualart for
his  continuous support and also for the inspiring
collaborations that brought   him to the field of stochastic partial
differential equations.
The authors would also like to thank Giovanni Peccati
for bringing the recent  work \cite{Tara} to our attention that leads
an improvement of our work. 
 The authors 
are very grateful to the referees for their 
meticulous review and numerous suggestions, 
which have greatly improved the quality of the current paper. 

\end{ackno}

\end{document}